\documentclass[]{amsart}

\usepackage{amsmath}
\usepackage{amsthm}
\usepackage{amssymb}
\usepackage{tikz-cd}
\usepackage{geometry}
\usepackage{multirow}
\usepackage{lscape}

\usepackage{graphicx}
\usepackage{mathabx}
\usepackage{amsfonts}                                     
\usepackage{mathrsfs}                                       
\usepackage[title]{appendix}     
\usepackage{xcolor}
\usepackage{textcomp}
\usepackage{manyfoot}
\usepackage{booktabs}                                       
\usepackage{algorithm}                                       
\usepackage{algorithmicx}  
\usepackage{algpseudocode}   
\usepackage{listings}

\usetikzlibrary{decorations.markings}
\tikzstyle{base}=[circle, draw, fill=black,inner sep=0pt, minimum width=4pt]
\tikzstyle{affine}=[circle, draw, fill=red,inner sep=0pt, minimum width=4pt]
\tikzstyle{affine2}=[circle, draw, fill=red,inner sep=0pt, minimum width=8pt]
\tikzstyle{affine3}=[circle, draw, fill=red,inner sep=0pt, minimum width=12pt]
\tikzstyle{affine4}=[circle, draw, fill=red,inner sep=0pt, minimum width=16pt]
\tikzstyle{invis}=[circle,inner sep=0pt, minimum width=4pt]
\tikzstyle{fat2}=[circle,draw,fill=black,inner sep=0pt, minimum width=8pt]
\tikzstyle{fat3}=[circle,draw,fill=black,inner sep=0pt, minimum width=12pt]
\tikzstyle{fat4}=[circle,draw,fill=black,inner sep=0pt, minimum width=16pt]
\usepackage{graphicx}
\usepackage{subfiles}
\usepackage{subcaption}
\usepackage{mathtools}
\usepackage{hyperref}
\hypersetup{pdfauthor={KaufmanGreenberg}
	colorlinks=true,
	linkcolor=blue,
	hypertexnames=false,
	filecolor=black,      
	urlcolor=blue,
	citecolor=blue
}
\usepackage[nameinlink]{cleveref}

\newcommand{\lcm}{\mathrm{lcm}}

\newcommand{\gr}[2]{\mathrm{Gr}(#1, #2)}
\newcommand{\ClusterAlgebra}[1]{\mathcal{A}_{#1}}

\newcommand{\ClusterComplex}[1]{\mathcal{M}_{#1}}
\newcommand{\ClusterModularGroup}[1]{\Gamma_{#1}}
\newcommand{\Affine}[1]{{#1}^{(1)}}
\newcommand{\TwistedAffine}[2]{{#1}^{(#2)}}

\newcommand{\NormalClosure}[1]{\mathcal{N}(#1)}

\newcommand{\cyclicgroup}[1]{\mathbb{Z}_{#1}}

\newcommand{\FreeGroup}[1]{F_{#1}}
\newcommand{\Aut}[1]{\mathrm{Aut}(#1)}
\newcommand{\Mut}[1]{\mathrm{Mut}(#1)}
\newcommand{\PSL}[2]{\mathrm{PSL}(#1,\mathbb{#2})}
\newcommand{\SL}[2]{\mathrm{SL}(#1,\mathbb{#2})}

\newcommand{\db}[2]{#1_#2^{(1,1)}}
\newcommand{\dbf}[4]{#1_#2^{(#3,#4)}}

\newcommand{\revPath}[1]{\overline{#1}}
\renewcommand{\vec}[1]{\mathbf{#1}}
\newcommand{\MappingClassGroup}[1]{\mathrm{Mod}(#1)}
\newcommand{\groupgenby}[1]{\langle #1 \rangle}
\newcommand{\Z}{\mathbb{Z}}
\newcommand{\A}{\mathcal{A}}
\newcommand{\T}{T_{\vec{n},\vec{w}}}
\DeclarePairedDelimiter\floor{\lfloor}{\rfloor}
\newcommand{\downmapsto}{\rotatebox[origin=c]{-90}{$\scriptstyle\mapsto$}\mkern2mu}

\theoremstyle{definition}
\newtheorem{example}{Example}[section]
\newtheorem{definition}[example]{Definition}
\newtheorem{conjecture}[example]{Conjecture}
\newtheorem{remark}[example]{Remark}
\theoremstyle{plain}
\newtheorem{theorem}[example]{Theorem}
\newtheorem{corollary}[example]{Corollary}
\newtheorem{claim}[example]{Claim}
\newtheorem{lemma}[example]{Lemma}

\renewcommand{\tilde}{\widetilde}
\renewcommand{\hat}{\widehat}

\title{Cluster Modular Groups of Affine and Doubly Extended Cluster Algebras}

\date{}

\author{Zachary Greenberg}
\address{Max Planck Institute for Mathematics in the Sciences\\
Inselstr. 22\\04103 Leipzig, 
Germany }
\email{greenberg@mis.mpg.de}

\author{Dani Kaufman}
\address{Max Planck Institute for Mathematics in the Sciences\\
Inselstr. 22\\04103 Leipzig, 
Germany }
\email{kaufman@mis.mpg.de}

\begin{document}

\begin{abstract}
    We calculate the cluster modular groups of affine and doubly extended type cluster algebras in a uniform way by introducing a new family of quivers. We use this uniform description to construct a natural finite quotient of the cluster complex of each affine and doubly extended cluster algebra. Using this construction, we introduce the notion of affine and doubly extended generalized associahedra, and count their facets. 
\end{abstract}
\maketitle

\setcounter{tocdepth}{1}
\tableofcontents

\section{Introduction}

Cluster algebraic structures underlying the coordinate rings of various algebraic varieties have proven to be vital to many problems in geometry, number theory and mathematical physics. When a cluster algebra has finitely many clusters, it is called a \emph{finite type} cluster algebra, and all of the relevant information it provides can be easily computed and studied. This paper is motivated by considering the properties of cluster algebras and cluster complexes as they transition from finite to infinite type.

One of the foundational results in the theory is that finite type cluster algebras are exactly associated with the Dynkin diagrams of finite root systems. It should be expected that the simplest possible infinite type cluster algebras are associated with affine Dynkin diagrams, and our study begins in this situation. Thinking of affine Dynkin diagrams as ``singly extended'' Dynkin diagrams, we also take our analysis one step further to the ``doubly extended'' case. Doubly extended Dynkin diagrams were introduced by Saito in \cite{Saito:Extended_affine_root_systems}, and are used to classify ``elliptic root systems''.

We will study the cluster complexes of the affine and doubly extended cluster algebras by explicitly describing their associated cluster modular groups. We compute these particular groups by introducing a family of weighted quivers called $\T$ quivers. For various choices of vectors of positive integers $\vec{n}$ and $\vec{w}$ these quivers provide initial seeds for each of the affine and doubly extended cluster algebras. 

The primary results of this discussion are the following theorems:

\begin{theorem}
Let $\vec{n},\vec{w}$ be $m$ dimensional vectors of positive integers. Let $\chi(\T) = \sum(w_i(n_i^{-1}-1))+2$. Then we have the following:
\begin{enumerate}
    \item If $\chi>0$, then $\T$ provides a seed of an affine cluster algebra.
    \item If $\chi=0$, then $\T$ provides a seed of a doubly extended cluster algebra.
    \item If $\chi<0$, then $\T$ provides a seed of an infinite mutation type cluster algebra.
\end{enumerate}
Moreover almost\footnote{The twisted Dynkin diagrams that are Langlands dual to standard diagrams have ``dual'' $\T$ quivers. However their cluster structure is identical to their duals, so we mostly don't need to treat them. The $\db{A}{1}$ and $\TwistedAffine{BC}{4}_n $ cluster algebras are simple to treat as special cases.} every affine and doubly extended cluster algebra has a seed with underlying quiver isomorphic to a $T_{\vec{n},\vec{w}}$ for some $\vec{n},\vec{w}$.
\end{theorem}

Informally, the cluster modular group is the automorphism group of the mutation structure of the cluster algebra. We show that there is an abelian subgroup, $\Gamma_\tau$, of the cluster modular group of cluster algebras coming from $\T$ quivers generated by ``twists'' $\tau_i$ for each ``tail'' $i=1, \dots, m$ and an element $\gamma$ satisfying $\tau_i^{n_i} = \gamma^{w_i}$ for all $i$. Let $H = \Aut{T_{\vec{n},\vec{w}}}$ be the automorphism group of a $T_{\vec{n},\vec{w}}$ quiver. This group acts on $\Gamma_\tau$ by permuting twists $\tau_i$ and $\tau_j$ whenever $n_i=n_j$ and $w_i=w_j$. 

\begin{theorem}
\begin{enumerate} 
    \item The cluster modular group of an affine cluster algebra is isomorphic to $\Gamma_\tau \rtimes H$.
    \item The cluster modular group of a doubly extended cluster algebra is generated by the elements of $\Gamma_\tau \rtimes H$ and one new generator, $\delta$. 
\end{enumerate}
\end{theorem}
See \Cref{sec:TnwModularGroup,sec:DoubleExtendedModularGroup} for the full definitions of $\tau_i, \gamma,\delta$.\\
We conjecture the following about infinite mutation type $\T$ quivers, i.e. when $\chi<0$.
\begin{conjecture}
If $\chi<0$, then the cluster modular group of a cluster algebra with initial seed given by a $T_{\vec{n},\vec{w}}$ quiver is isomorphic to $\Gamma_\tau \rtimes H$.
\end{conjecture}

We use the computation of the cluster modular group of $T_{\vec{n},\vec{w}}$ cluster algebras to construct natural finite quotients of the cluster complex. 

In the affine case, the element $\gamma$ generates a finite index subgroup of the cluster modular group. We define the \emph{quotient cluster complex} where cells are equivalence classes up to the action of $\gamma$. The dual to the quotient complex is analogous to the generalized associahedron associated to finite type cluster algebras. We compute the basic properties of this \emph{affine generalized associahedron} including the number of codimension 1-cells and dimension 0-cells and we conjecture that they are each homomorphic to a sphere. 

We have the following theorems,
\begin{theorem}[\Cref{thm:AffineCountingFacets}]
The number of distinct cluster variables in an affine cluster algebra up to the action of $\groupgenby{\gamma}$ is given by 
\begin{equation}
\sum_i(n_i-1)n_i + \frac{n}{\chi}
\end{equation}
The number of distinct clusters in an affine cluster algebra up to the action of  $\groupgenby{\gamma}$ is given by 
\begin{equation}
\frac{2}{\chi}\prod_i\binom{2n_i-1}{n_i}
\end{equation}
\end{theorem}

These two equations provide the number of codimension 1-cells and dimension 0-cells of an affine generalized associahedron respectively. 

In the doubly extended case, the element $\gamma$ no longer generates a normal subgroup. Instead, we find that the normal closure of this element, in most cases, is a free, finite index normal subgroup of the cluster modular group. We compute the number of clusters in the quotient cluster complex by this group in \Cref{fig:DoublyExtendedCountingClusters}. We define doubly extended generalized associahedra to be the dual of this quotient complex. 

We conjecture that affine and doubly extended generalized associahedra are each homeomorphic to a product of spheres.

\begin{conjecture}
\begin{enumerate}
    \item The affine generalized associahedron of an affine cluster algebra of rank $n+1$ is homeomorphic to a sphere of dimension $n$.
    
    \item The cluster complex of a doubly extended cluster algebra of rank $n+2$ is homotopy equivalent to $S^{n-1}$. 
    
    \item The doubly extended associahedron associated with a doubly extended cluster algebra is homeomorphic to $S^{n-1} \times S^2$ in all cases other than $\db{E}{8}$ where it instead is homeomorphic to $S^7 \times S^1 \times S^1$.
\end{enumerate}

\end{conjecture}

\subsection{Structure of the paper}
    
The structure of the paper is as follows:

\Cref{sec:Prelim,sec:ClusterModularGroup} review the background necessary to define the cluster complex and cluster modular group associated with a cluster algebra. 
\Cref{sec:TnwQuivers} studies the cluster modular group of a type $\T$ cluster algebra. \Cref{sec:AffineClusterAlgebra} deals with affine cluster algebras and studies affine generalized associahedra. Finally \Cref{sec:DoubleExtended} deals with doubly extended cluster algebras and doubly extended associahedra. 

The appendix contains a list of all of the Dynkin diagrams referenced in this paper and a review of the relationship between cluster algebras and triangulations of surfaces. The proofs of theorems which use this relationship are included in the appendix.

\subsection{Acknowledgements}
This research was partially supported by by the Deutsche Forschungsgemeinschaft (DFG, German Research Foundation) under both Project-ID 281071066-TRR 191 and Project ID 338644254-SPP2026  as well as PosLieRep ERC 101018839.\\
We would also like to give special thanks to Christian Zickert for his helpful advice and encouragement, to Chris Fraser for correction of a mistake, and to the anonymous referee for their very helpful and thorough comments and suggestions. 

\section{Preliminaries on Dynkin Diagrams\texorpdfstring{,}{} Quivers\texorpdfstring{,}{} and Mutations}\label{sec:Prelim}

While the goal of this paper is to study cluster algebras, most of our computations will be phrased  entirely in terms of quivers and mutations. We will give a definition of the cluster modular group entirely in terms of sequences of mutations and quiver isomorphisms, without any reference to cluster variables or cluster algebras. Firstly, we will discuss Dynkin diagrams and their associated quivers.
    
\subsection{Dynkin diagrams}

Dynkin diagrams are known to classify a number of important algebraic objects. These diagrams are graphs consisting of nodes and weighted edges. The weights are represented by drawing single, double, triple, or quadruple\footnote{Quadruple edges only appear in the $\TwistedAffine{BC}{4}_1$ Dynkin diagram.} edges between the corresponding nodes. Any non-single edge is given an orientation. We call this style of presenting Dynkin diagrams \emph{oriented Dynkin diagrams}. Each oriented Dynkin diagram we will consider is either acyclic or has no multi-edges. 

We will use a modified form of these diagrams where instead of using oriented multi edges, we will instead assign weights to the nodes and only use non-oriented edges. We call such a diagram a \emph{weighted Dynkin diagram}. See \Cref{app:DynkinDiagrams}  for a full list of weighted diagrams referenced in this paper. In the doubly extended case, some nodes are linked with a double edge; this should not be confused with a double edge in the oriented Dynkin diagrams. 

To construct a weighted Dynkin diagram from a oriented diagram, we make a diagram with the same underlying graph, and add weights to each node via the following procedure: We first select a sink of the oriented diagram and assign the corresponding node in the weighted diagram weight 1. Then we assign weights to the other nodes by traversing the graph from our starting node and assigning the same weight when we cross a single edge, assigning weight times the multiplicity of a multi-edge when crossed against its orientation and dividing the weight when crossed with the orientation. 

To show the weights of nodes visually, we display nodes larger when they have higher weight. See \Cref{fig:DynkinWeightedVsOriented} for examples of converting between oriented and weighted Dynkin diagrams.

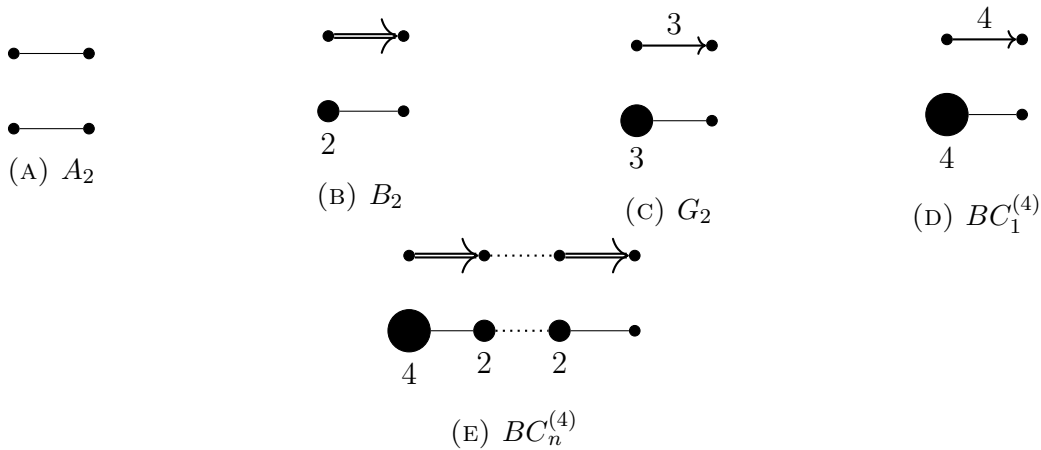
\begin{figure}[hb]
    \centering
    \begin{subfigure}{.24\textwidth}
     \centering
     \begin{tikzpicture}
        \node[base] (1) [] {};
        \node[base] (2) [right of = 1] {};
        \node[base] (3) [label=below:\vphantom{1},below of = 1] {};
        \node[base] (4) [right of = 3] {};
        \path[-] (1) edge [] node {} (2);
        \path[-] (3) edge [] node {} (4);
    \end{tikzpicture}
    \caption{$A_2$}
    \end{subfigure}
     \begin{subfigure}{.24\textwidth}
     \centering
     \begin{tikzpicture}
        \node[base] (1) [] {};
        \node[base] (2) [right of = 1] {};
        \node[fat2] (3) [label=below:2, below of = 1] {};
        \node[base] (4) [right of = 3] {};
        \path[->] (1) edge [double,thick] node {} (2);
        \path[-] (3) edge [] node {} (4);
    \end{tikzpicture}
    \caption{$B_2$}
    \end{subfigure}
    \begin{subfigure}{.24\textwidth}
     \centering
     \begin{tikzpicture}
        \node[base] (1) [] {};
        \node[base] (2) [right of = 1] {};
        \node[fat3] (3) [label=below:3,below of = 1] {};
        \node[base] (4) [right of = 3] {};
        \path[->] (1) edge [thick,above] node {3} (2);
        \path[-] (3) edge [] node {} (4);
    \end{tikzpicture}
    \caption{$G_2$}
    \end{subfigure}
    \begin{subfigure}{.24\textwidth}
     \centering
     \begin{tikzpicture}
        \node[base] (1) [] {};
        \node[base] (2) [right of = 1] {};
        \node[fat4] (3) [label=below:4,below of = 1] {};
        \node[base] (4) [right of = 3] {};
        \path[->] (1) edge [thick,above] node {4} (2);
        \path[-] (3) edge [] node {} (4);
    \end{tikzpicture}
    \caption{$BC_1^{(4)}$}
    \end{subfigure}
    \begin{subfigure}{.9\textwidth}
         \centering
     \begin{tikzpicture}
        \node[base] (1) [] {};
        \node[base] (2) [right of = 1] {};
        \node[base] (3) [right of = 2] {};
        \node[base] (4) [right of = 3] {};
        \node[fat4] (5) [label=below:4,below of = 1] {};
        \node[fat2] (6) [label=below:2,right of = 5] {};
        \node[fat2] (7) [label=below:2,right of = 6] {};
        \node[base] (8) [right of = 7] {};
        \path[->] (1) edge [thick, double] node {} (2);
        \path[-] (2) edge [thick, dotted] node {} (3);
        \path[->] (3) edge [thick, double] node {} (4);
        \path[-] (5) edge [] node {} (6);
        \path[-] (6) edge [thick, dotted] node {} (7);
        \path[-] (7) edge [] node {} (8);
    \end{tikzpicture}
    \caption{$\TwistedAffine{BC}{4}_n$}
    \end{subfigure}
    \caption{Correspondence between oriented and weighted Dynkin diagrams.}
    \label{fig:DynkinWeightedVsOriented}
\end{figure}

The ``finite'' Dynkin diagrams classify irreducible root systems, simple complex Lie algebras, and finite type cluster algebras. 
Finite Dynkin diagrams, shown in \Cref{fig:DynkinFiniteSimplyLaced}, can be simply laced (types $A_n$, $D_n$, $E_6$, $E_7$, $E_8$) or non-simply laced. 

Each of the non simply laced diagrams (\Cref{fig:DynkinFiniteFolded}) can be constructed by folding the appropriate simply laced diagrams. Given a graph automorphism of a diagram, $\sigma$, we may fold the diagram along $\sigma$ by combining nodes in the same orbit of the action of $\sigma$ into a single node with weight given by the size of the orbit.

\begin{example}
For example, $C_n$ comes from folding an $A_{2n+1}$ diagram in half, $B_n$ is obtained by folding the two tails of a $D_{n+1}$, $F_4$ comes from folding the length three tails of an $E_6$ diagram and $G_2$ comes from folding all three tails of a $D_4$ diagram. 
\end{example}

Affine Dynkin diagrams can be obtained by connecting an extra node to a finite Dynkin diagram. We call this node an extending node and we color it red in \Cref{fig:DynkinAffineSimplyLaced,fig:DynkinAffineFolded}. As in the finite case, the non-simply laced diagrams also come from folding corresponding simply laced diagrams. 

There is also a notion of a doubly extended or elliptic Dynkin diagram introduced by Saito in \cite{Saito:Extended_affine_root_systems}. Of these elliptic diagrams only the ``1 co-dimensional'' elliptic Dynkin diagrms have associated cluster algebras. We show all the relevant doubly extended weighted Dynkin diagrams in \Cref{fig:DynkinDouble,fig:DynkinDoubleFolded}. These weighted diagrams contain double edges. 

There is a notion of ``Langlands duality'' of non-simply laced Dynkin diagrams. The Langlands dual of a diagram is obtained by replacing each node of weight $W$ with a node of weight $M/W$ where $M$ is the maximum of all the node weights. 

\begin{remark}
The cluster algebras associated with a Dynkin diagram and its Langlands dual will have the same cluster structure and thus the same cluster modular group. For this reason,  we will not consider the ``twisted'' affine types, which are generally Langlands dual to the normal affine types described above. The only exception to this is the affine type $\TwistedAffine{BC}{4}_n$ shown in \Cref{fig:DynkinAffineTwisted}. This type is its own Langlands dual, so we will consider it as a special case when necessary. 
\end{remark}

\subsection{Quivers}
    
    We will define cluster algebras and cluster modular groups in terms of \emph{weighted quivers} and show how to construct a weighted quiver out of a Dynkin diagram (\Cref{sec:DynkinQuiver}).
    First, we will define unweighted quivers, which will come from simply laced Dynkin diagrams.
    
    \begin{definition}
    
    An \emph{unweighted quiver} is a finite directed graph without self loops or 2-cycles. An unweighted quiver is equivalent to a skew symmetric matrix with entries $e_{ij}$  equal to the number of arrows from node $i$ to node $j$ with arrows from $j$ to $i$ counted negatively. We write $V(Q)$ for the set of nodes of of the quiver $Q$.
    
    \end{definition}
    
    Each of the simply-laced Dynkin diagrams can be transformed into an unweighted  quiver by assigning an orientation to each edge. In the non simply-laced case, we will need the notion of a weighted quiver. 
    
    \begin{definition}
    A \emph{weighted quiver} is a quiver where each node, $i$, is assigned an integer weight $w_i > 0$. 
    \end{definition}
    
    The exchange matrix encoded by such a weighted quiver will not be skew symmetric, but will be ``skew-symetrizeable''.
    
    \begin{definition}
    A \emph{skew-symmetrizable} matrix is a matrix $M$ for which there exists a diagonal matrix $D$ such that $MD^{-1}$ is skew symmetric. 

    The \emph{exchange matrix} of a weighted quiver is a skew-symmetrizable matrix $\epsilon_{ij}$ with entries $\epsilon_{ij}=e_{ij}w_j/\gcd(w_j,w_i)$, where $e_{ij}$ is the number of arrows from node $i$ to node $j$ counted with signs.  The diagonal matrix which symmetrizes $\epsilon_{ij}$ is given by $D_{ii}=w_i$
    \end{definition}
    
    For convenience, we will generally refer to weighted/unweighted quivers simply as ``quivers'' unless we want to emphasize the distinction. 
    
    \begin{definition}
    An \emph{isomorphism} of quivers is a graph isomorphism which preserves all arrow directions and node weights. Two quivers are isomorphic if and only if their exchange matrices are identical after conjugation by a permutation matrix. 
    \end{definition}
    
    Let $Q$ be a quiver.
    
    \begin{definition}
    \emph{Mutation} of $Q$ at node $k$ is an operation which produces a new quiver, $\mu_k(Q)$, where the edges of $Q$ change by the following rule:
    \begin{itemize}
        \item For each pair of arrows $(i \rightarrow k)$, $(k \rightarrow j)$ add $w_k \gcd(w_i,w_j) / (\gcd (w_i,w_k) \gcd(w_k,w_j))$ arrows from node $i$ to node $j$, cancelling any 2-cycles between nodes $i$ and $j$. 
        \item Reverse every arrow incident to $k$. 
    \end{itemize}
    
    The exchange matrix of $\mu_k(Q)$ given by $[\epsilon'_{ij}]$ is obtained from the exchange matrix of $Q$ by the following formula:
    \begin{align}
        \epsilon'_{ij} &= -\epsilon_{ij} &\text{ if }  i \text{ or } j = k \\
        \epsilon'_{ij} &= \epsilon_{ij} + \frac{|\epsilon_{ik}|\epsilon_{kj}+\epsilon_{ik}|\epsilon_{kj}|}{2} &\text{ otherwise }
    \end{align}
    
    The set of all quivers up to isomorphism which may be obtained from $Q$ by any sequence of mutations is called the \emph{mutation class} of $Q$ and is denoted $\Mut{Q}$.
    \end{definition}
    
   \subsection{Quivers from Dynkin diagrams}\label{sec:DynkinQuiver}
  
    A choice of an orientation of the edges of a weighted Dynkin diagram, $X$, gives a weighted quiver, $Q$. The choice of orientations gives the set of arrows of $Q$. 
    
    It is easy to see that for any weighted Dynkin diagram without double edges, other than $X=\Affine{A}_n$, that any two choices of orientations give mutation equivalent quivers.
    
    When $X=\Affine{A}_n$ there is a family $A_{p,q}$, $p\geq q , p+q = n+1$, of mutation classes of quivers obtained by choosing $p$ and $q$ arrows in each direction around the cycle. One may check that $A_{n+1,0}$ is mutation equivalent to $D_{n+1}$.
    
    Doubly extended Dynkin diagrams include a single double edge $e$. In these cases, the orientation of edges incident to the endpoints of $e$ must be chosen to form an oriented cycle with the orientation of $e$.

\subsection{Cluster algebras from quivers}

Cluster algebras (of geometric type) are formed by starting with an initial ``seed'' and applying all possible mutations. 

Let $Q$ be a quiver with $n$ nodes and let $\boldsymbol{z} =\{z_1, \dots,z_n\}$ be algebraically independent elements  of $\mathbb{Q}(x_1,\dots,x_n)$.
\begin{definition}
 A \emph{seed}, $\boldsymbol{i}=(Q,\boldsymbol{z})$ is a pair of a quiver $Q$ and the set of variables $\boldsymbol{z} = \{z_1, \dots,z_n\}$ with each variable associated to a distinct node of $Q$. The set $\boldsymbol{z}$ is called a $\emph{cluster}$ and the variables are called \emph{cluster variables}. Given any quiver, $Q$, we can make an \emph{initial seed} associated with $Q$ simply by associating the variables $z_1,\dots,z_n$ to the nodes of $Q$. 
\end{definition}

We can mutate a seed at a node $k$ by mutating $Q$ as before and replacing the variable $z_{k}$ associated with node $k$ with a new variable $z_{k}'$ satisfying
\begin{equation}
    z_{k}\cdot z_{k}' = \prod_{\epsilon_{ki}<0} z_{i}^{-\epsilon_{ki}} + \prod_{\epsilon_{kj}>0} z_{j}^{\epsilon_{kj}}
\end{equation} 

\begin{definition}
The $\mathbb{Z}-$Algebra generated by all of the cluster variables obtained from all possible mutations of a seed is the \emph{cluster algebra} associated with that seed. We write $\ClusterAlgebra{Q}$ for the cluster algebra associated with an initial seed with quiver $Q$.
\end{definition}

 \subsection{Folding quivers and cluster algebras}\label{sec:FoldingQuivers}

The relationship between the classical ``folding'' of the simply laced Dynkin diagrams to form the non simply laced diagrams can be extended to quivers. This is used to classify the finite mutation class skew symmetrizable exchange matrices in \cite{FST:finite_mutation_via_unfoldings}.
The definitions of folded quivers and folded cluster algebras we review here can be found in \cite{KaufmanSpecialFoldingQuivers2024a}. 

Essentially, to fold a quiver we will group its nodes into disjoint sets and do mutations by requiring that we mutate all the nodes of a given set together. We call the operation of mutating each of the elements of a set of nodes in turn a ``group mutation.'' In general a group mutation will depend on the order the nodes are mutated. However if all the mutations commute, the order doesn't matter. A sufficient condition for all the mutations to commute is that there are no arrows between any two nodes in each set. 

With this in mind, we have the following definition:
\begin{definition}
A \emph{folding} of a quiver, $Q$, with $n$ nodes is a choice of $k$ non empty and disjoint sets of nodes $\mathbb{K}= \{K_1,\dots,K_k\}$ whose union contains all of the nodes of $Q$ satisfying the following conditions called the \emph{folding conditions}:
\begin{enumerate}
    \item The nodes contained in a given set have no arrows between themselves.
    \item Condition 1 is satisfied after any number of group mutations of these fixed sets.
\end{enumerate}
We call the pair $(Q,\mathbb{K})$ a \emph{valid folding} of the quiver $Q$ if it satisfies the folding conditions. The \emph{folded mutation class} of a folding of a quiver is the set of folded quivers which are group mutation equivalent it.
\end{definition}

\begin{lemma}[lemma 2.8 \cite{KaufmanSpecialFoldingQuivers2024a}]\label{lem:fold_aut_2}
    If $\sigma \in \Aut{Q}$ is an involution, then the sets consisting of orbits of the action of $\sigma$ is a valid folding of $Q$. 
\end{lemma}
\begin{proof}
    One easily checks that the folding conditions are satisfied since any arrow between nodes in a orbit of $\sigma$ would induce a 2-cycle. This is still true after group mutations since $\sigma$ still acts as an automorphism after group mutation. 
\end{proof}
Let $G \subset \Aut{Q}$ be a subgroup of automorphisms of $Q$ and let $\mathbb{K}_G$ denote the orbits of the action of $G$ on $Q$.

\begin{definition}
    We call a subgroup $G \subseteq \Aut{Q}$ \emph{saturated} if any automorphism of $Q$ which fixes all the orbits of $G$ is actually in $G$ and if for any two nodes $i,j$ in the same orbit there is an element $\sigma \in G$  with $\sigma(i)= j$ and $\sigma(j)=i$. 
\end{definition}

When $G$ is saturated the orbits of the action of $G$ form a valid folding and group mutations still have $G$ as a subgroup of automorphisms by \Cref{lem:fold_aut_2}.

\begin{definition}
    Given $(Q,\mathbb{K}_G)$ a folding by a saturated subgroup of automorphisms $G$ as above, we can define a \emph{folded seed} by assigning $k$ variables to the nodes in $Q$ where we assign the same initial cluster variable for all nodes in the same orbit.    
    The \emph{folded cluster algebra} $\ClusterAlgebra{(Q,\mathbb{K})}$ is defined by beginning with the folded seed associated to $(Q,\mathbb{K})$.  The algebra $\ClusterAlgebra{(Q,\mathbb{K}_G)}$ is then generated by performing all possible sequences of group mutations of this folded quiver. 
\end{definition}

Since this quiver is folded by an automorphism group, we can see that group mutations of this folded seed actually produce new folded seeds since the group of automorphisms naturally acts on the exchange relations. 

In some cases this folded cluster algebra is related to the cluster algebra associated to a particular weighted quiver:

\begin{definition}
     Let $G \subset \Aut{Q}$ be a saturated subgroup of automorphisms and $\mathbb{K}_G$ the set of orbits as above. We denote by $Q_{G}$ the weighted quiver with $k$ nodes of weight $|K_i|$ respectively and $\frac{m_{ij}}{\lcm{(|K_i|,|K_j|)}}$ arrows from node $K_i$ to node $K_j$ where $m_{i,j}$ is the total number of arrows between the sets $K_i$ and $K_j$ in $Q$ counted with signs.\\
     We call the quiver $Q$ an \emph{unfolding} of $Q_G$ if $\ClusterAlgebra{(Q,\mathbb{K}_G)}$ is isomorphic to $\ClusterAlgebra{Q_G}$.
\end{definition}

\begin{theorem}\label{thm:unfolding}
    Every cluster algebra associated to a weighted quiver associated to an affine or doubly extended Dynkin diagram $R$ has a unique unfolding $Q$ such that $\ClusterAlgebra{R} \simeq \ClusterAlgebra{Q_G}$ for some saturated subgroup of automorphisms of $Q$.
\end{theorem}
\begin{proof}
    \cite{FST:finite_mutation_via_unfoldings} gives an unfolding for each skew-symmetrizable finite mutation type exchange matrix, which gives an unfolding of the corresponding weighted quivers. The uniqueness only needs to be checked for a single quiver in each mutation class, which is easily checked for the quivers corresponding to Dynkin diagrams. 
\end{proof}

\begin{remark}
     We define an unfolding to preserve the weights of the initial quiver $R$. Without this restriction, one can obtain other quivers $Q'$ with an isomorphic folded cluster algebra to $\ClusterAlgebra{R}$. For example consider $Q_n$ the quiver consisting on $n$ disjoint copies of $Q$ and fold by the additional permutation group of these copies. While $\ClusterAlgebra{(Q_n)_G}$ is isomorphic to $\ClusterAlgebra{R}$ every node in $(Q_n)_G$ will have weight $n$ times the weight of the node in $R$ and thus is not an unfolding of $R$ by our definition. 
\end{remark}

\subsection{C-vectors and reddening sequences}

    Now we will recall some definitions needed to understand c-vectors and reddening mutation sequences. 
    
    \begin{definition}
    A \emph{frozen} node of a quiver is a node which we do not allow mutations at. A \emph{mutable} node is a node which is not frozen. We write $F(Q)$ for the set of frozen nodes of $Q$ and $\mu(Q)$ for the sub-quiver of $Q$ consisting of mutable nodes. 
    \end{definition}
    Let $Q$ be a quiver with frozen vertices. Let $R$ be obtained by a sequence of mutations from $Q$.
    \begin{definition}
      A \emph{frozen isomorphism}  between $Q$ and $R$ is directed graph isomorphism $\sigma:Q \rightarrow R$ which preserves node weights and restricts to the identity on the frozen nodes. 
    The \emph{mutation class} of a quiver with frozen vertices is the set of quivers obtained from mutations of $Q$ up to frozen isomorphism. We also denote this mutation class by $\Mut{Q}$. 
    \end{definition}
    
    Let $n=|V(\mu(Q))|$ and $f=|F(Q)|$. We call $n$ the ``rank'' of $Q$ and we will generally number the mutable nodes of $Q$ with $1, \dots ,n$ and the frozen nodes with $n+1, \dots ,n+f$. 

    \begin{definition}
    A \emph{framing} of a quiver $Q$ is any quiver $\tilde{Q}$ such that $\mu(\tilde{Q})=\mu(Q)$.
    The \emph{c-vectors} of $Q$ is the collection, $\{\vec{c}_{i}| 0\leq i \leq n\}$, of  $f$-dimensional vectors given by $\vec{c}_{i}^j = \epsilon_{i,(j+n)}$
    \end{definition}
    
    Let $Q$ be a quiver which consists of only mutable nodes. There is a canonical framing, $\hat{Q}$, obtained from $Q$ by adding a frozen node $F_i$ with matching weight $w_i$ for each node $N_i$ and a single arrow from $N_i$ to $F_i$. $\hat{Q}$ is called the ``ice'' quiver associated with $Q$. The cluster algebra formed by starting with $\hat{Q}$ is called the cluster algebra with \emph{principal coefficients}.

\begin{remark}
    There are two possible conventions of c-vectors, the other possibility is $\vec{c}_{i}^j = \epsilon_{(j+n),i)}$. This is the convention used by Bernhard Keller's quiver mutation applet\footnote{\url{https://webusers.imj-prg.fr/~bernhard.keller/quivermutation/}}. With the convention we chose, the matrix of c-vectors $[\vec{c}_i^j]$ associated to $\hat{Q}$ is the identity matrix.
    \end{remark}
    
    \begin{theorem}[\cite{Zelevinsky:Tropical_dualities}]
    The sets of c-vectors of quivers in $\Mut{\hat{Q}}$ are in one-to-one correspondence with the clusters in the cluster algebra with principal coefficients associated with $Q$. 
    \end{theorem}
    
    Via this theorem, we see that by considering sets of c-vectors, one may understand whether a mutation sequence returns to a cluster with the same cluster variables without actually computing them. We only need to check that their sets of c-vectors are the same. 
    
    \begin{definition}
    Let $k$ be a node of a quiver $Q$ with frozen vertices. We call $k$ \emph{green} (resp. red) if the c-vector associated with $k$ has all positive (resp. negative) entries. 
    \end{definition}
    
    In $\hat{Q}$ every node is green. 
    
    \begin{theorem}[sign coherence \cite{DWZ:quivers_with_potentials,GHHK:canonical}]
    Let $Q$ be a quiver without frozen variables. Then every quiver $R \in \Mut{\hat{Q}}$ also has the property that every node of $R$ is either red or green. 
    \end{theorem}
    
    Let $\Check{Q}$ be the framing of $Q$ by adding a frozen node $F_i$ with matching weight $w_i$ for each node $N_i$ and a single arrow from $F_i$ to $N_i$.
    
    \begin{theorem}[\cite{Muller_maximalgreensequences}]
    Suppose there is $R \in \Mut{\hat{Q}}$ satisfying that every node of $R$ is red. Then $R$ is frozen isomorphic to $\Check{Q}$. 
    \end{theorem}
    
    \begin{definition}
    Suppose that $\Check{Q} \in \Mut{\hat{Q}}$. We call a sequence of mutations taking $\hat{Q}$ to $\Check{Q}$ a \emph{reddening sequence}. 
    \end{definition}
    
    The existence of a reddening sequence is an important property of a given quiver. We will explicitly construct reddening sequences for the family of quivers introduced in \Cref{sec:TnwQuivers}. 
    
    \begin{theorem}[\cite{Muller_maximalgreensequences}]\label{thm:MullerReddening}
    Let $Q$ be a quiver with no frozen vertices and let $R \in \Mut{Q}$. Then $\hat{Q}$ has a reddening sequence if and only if $\hat{R}$ does. 
    \end{theorem}
    
\section{The Cluster Modular Group}\label{sec:ClusterModularGroup}  

    We will now review how to associate a group to any quiver or cluster cluster algebra called the \emph{cluster modular group}.
    This group is essentially the automorphism group of the mutation structure of a cluster algebra associated with a given quiver. We can use our definitions of c-vectors to give a definition of this group without any reference to the cluster variables.

    Let $Q$ be a quiver without frozen vertices. By identifying the mutable nodes of $Q$ with the integers $[n] = {1,...,n}$, we obtain a right action of $\cyclicgroup{2}^{*n}$ on quivers in the mutation class, $\Mut{Q}$, by mutating at each node in sequence. We refer to elements of $\cyclicgroup{2}^{*n}$ as mutation paths.
    
    We would now like to focus on the subset of paths that return $Q$ to an isomorphic quiver. In order to define a group structure on this subset, we need to consider pairs $(P,\sigma)$ of mutation paths $P$ and quiver isomorphisms $\sigma:Q \rightarrow P(Q)$. We write quiver isomorphisms as elements of the symmetric group $S_n$. The symmetric group acts on mutation paths on the left by $\sigma(\mu_i)= \mu_{\sigma(i)}$ and on itself by conjugation.

    Given two such pairs $(P,\sigma)$ and $(R,\tau)$ we can multiply by forming the composite path $P \circ \sigma(R)$ and the composite quiver isomorphism $\sigma(\tau) \circ \sigma$:
    \begin{equation}
        \begin{tikzcd}
        Q \arrow[r,"\sigma"] \arrow[rrr,bend right ,"\sigma \tau"]& P(Q)\arrow[rr,"\sigma(\tau) = \sigma \tau \sigma^{-1}"] & & (P,\sigma(R))(Q).
    \end{tikzcd}
    \end{equation}
    This multiplication rule can also be obtained by viewing these pairs as elements of the semidirect product 
    \begin{equation}
        \cyclicgroup{2}^{*n} \rtimes S_n.
    \end{equation} 
    This convenient embedding is why we consider quiver isomorphism from $Q$ to $P(Q)$.
    
    This gives a group structure on the set of mutation paths which return $Q$ to an isomorphic quiver paired with isomorphisms from the starting to ending quiver; we call this group the \emph{quiver modular group} associated with $Q$ denoted $\tilde{\Gamma}_Q$. 

    Elements of the quiver modular group act on the cluster variables of a seed $\boldsymbol{i}$ associated with $Q$. The path $P$ provides a path to a new seed, and $\sigma$ gives a map from the cluster variables on $\boldsymbol{i}$ to those on $P(\boldsymbol{i})$
    
    \begin{definition}
    A pair $(P,\sigma)$ which acts trivially on the cluster variables of any initial seed associated with $Q$ is called a \emph{trivial cluster transformation}. Let $T$ be the group of trivial cluster transformations; this is a normal subgroup of $\tilde{\Gamma}_Q$. The group $\Gamma_Q = \tilde{\Gamma}_Q/T$ is called the \emph{cluster modular group} associated with the quiver $Q$. 
    \end{definition}
    
    Equivalently, a trivial cluster transformation is an element  $(P,\sigma)$ of $\tilde{\Gamma}_Q$ for which $\sigma$ is a frozen isomorphism $\hat{Q} \rightarrow P(\hat{Q})$. In this way, we may define $\Gamma_Q$ without any regard to cluster variables.

\begin{remark}
    The automorphisms of $Q$ are the subgroup of $\Gamma_Q$ given by $(\langle\rangle,\Aut{Q})$, the set of elements with an empty mutation path.
\end{remark}
    
    \begin{remark}
    Our notion of a quiver isomorphism requires that all of the arrow directions are preserved. In other definitions of the cluster modular group, such as those in \cite{FockGonch:cluster_ensembles,Schiffler:cluster_automorphisms,fraser_braid_2020}, one includes arrow reversing quiver automorphisms. Our version of the cluster modular group is an index two subgroup of this more general notion. 
    \end{remark}
    
    \begin{example}
    Consider the quiver $Q$ with two nodes and a single edge between them (\Cref{fig:simpleQuiverQ}).
    Mutation at $1$ in $Q$ yields a quiver with the edge now going from $2$ to $1$ (\Cref{fig:simpleQuiverQ'})
    If we want to perform the ``same'' mutation in $Q'$ that we did in $Q$ we want to mutate at the vertex corresponding to 1 under the isomorphism $f:Q \rightarrow Q'$, which is 2. In this case there is a unique isomorphism, but in general each choice of isomorphism gives rise to a different element of the cluster modular group. It is convenient to write these isomorphisms as permutations in $S_n$. The element described above would be written $g = (1, (12))$. In this case $g$ generates the cluster modular group and $g^5 = \text{id}$.  
    \begin{figure}
        \centering
        \begin{subfigure}{.4\textwidth}
        \centering
        \begin{tikzpicture}
        \node[base] (1) [label=1] {};
        \node[base] (2) [right of = 1, label=2] {};
        \path[->] (1) edge [] node {} (2); 
        \end{tikzpicture}
        \caption{Q}
        \label{fig:simpleQuiverQ}
        \end{subfigure}
         \begin{subfigure}{.4\textwidth}
            \centering
            \begin{tikzpicture}
            \node[base] (1) [label=1] {};
            \node[base] (2) [right of = 1,label=2] {};
            \path[->] (2) edge [] node {} (1); 
            \end{tikzpicture}
            \caption{Q'}
            \label{fig:simpleQuiverQ'}
        \end{subfigure}
        \caption{A simple quiver before and after mutation.}
        \label{fig:simpleQuiver}
    \end{figure}
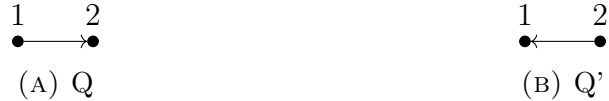
    
    \end{example}
    
    \subsection{The cluster complex}
    
    Recall that for any cluster algebra, $\ClusterAlgebra{Q}$, there is an associated simplicial complex $\ClusterComplex{Q}$ called the cluster complex. This complex is defined in detail in \cite{fomin_clusterI,FockGonch:Moduli_of_local_systems}. We will review the basic definitions of this complex here. First we will need the notion of compatibility of cluster variables. 
    
    \begin{definition}
    Two cluster variables are \emph{compatible} if they appear in a cluster together.
    \end{definition}
    
    The $k$-dimensional simplices of $\ClusterComplex{Q}$ correspond to size $k$ collections of mutually compatible cluster variables in $\ClusterAlgebra{Q}$. In other words, the cluster complex is the ``clique complex'' of the compatibility rule for cluster variables. In particular each vertex corresponds to an individual cluster variable and each edge connects two cluster variables when they can be found in a cluster together. The maximal dimension simplices correspond to the clusters of $\ClusterAlgebra{Q}$.
    
    \begin{remark}
     In \cite{FockGonch:cluster_ensembles} the cluster modular group is defined to be the simplicial symmetry group of the cluster complex. This symmetry group contains the cluster modular group as described in this paper as a proper subgroup. The distinction between these groups does not affect the main results of this paper.
    \end{remark}
    
    The 1-skeleton of the dual complex of the cluster complex is called the ``exchange graph'' of the cluster algebra. The vertices of this graph correspond to clusters and the edges correspond to mutations between clusters.

    \subsection{Computing cluster modular groups}
    
    We would like to have an algorithm to compute the cluster modular group. For general quivers, this can be very difficult since the mutation class can be infinite. When the quiver in question has finitely many quivers in its mutation class, there is an algorithmic construction of the cluster modular group, see Ishibashi's paper \cite{Ishibashi:cluster_modular}. We present a simplified version of the algorithm which only computes a generating set without computing all the relations.
    
    \begin{definition}
    The \emph{directed quiver mutation graph}, $G$, associated to a finite mutation class cluster algebra is a multi graph with a node for each quiver isomorphism class and a directed edge for each single mutation between isomorphism classes. The \emph{(undirected) quiver mutation graph} replaces directed two cycles corresponding to inverse mutations with a single undirected edge.
    \end{definition}
  
   Note, unlike the graph in \cite{Ishibashi:cluster_modular}, in our formulation the degree of each node is the rank of the cluster algebra.

    Each element $(P,f)$ of the cluster modular group corresponds to a cycle in $G$ by following $P$ in $G$. Furthermore the set of cycles in $G$ is finitely generated with one generator for each edge not in a fixed spanning tree of $G$. Since the automorphism group of each quiver is finite, this gives a finite list of  generators of the cluster modular group.
    
    In practice this method doesn't give the shortest possible list of generators of the cluster modular group. However it places an upper bound on how long the shortest path representing a generator of the cluster modular group can be. If $d$ is the diameter of the spanning tree for $G$, then the maximum length of the mutation path of a generator is $2d + 1$.
    
    \begin{remark}\label{rem:FiniteMutationTypeAlgorithm}
    To check if a group surjects onto the cluster modular group it suffices to check that it reaches every quiver isomorphic to the starting quiver within distance $2d + 1$.
    \end{remark}
    
    \begin{example}
    
    The mutation class of an $A_{2,1}$ quiver has two quiver isomorphism classes $Q_1, Q_2$, shown in \Cref{fig:QuiverMutationGraphExample}. It is easy to compute the directed and undirected quiver mutation graphs for this quiver simply by performing each of the three mutations on each quiver isomorphism class. 
    
    We can then compute a set of generators of the cluster modular group. There are two generators $e_1, e_2$ corresponding to the two loops from $Q_1$ and $Q_2$ to themselves.
    
    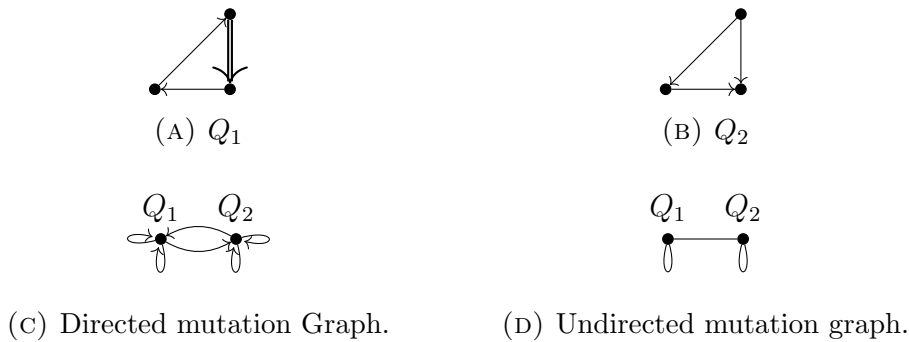
\begin{figure}[hb]
        \centering
        \begin{subfigure}{.4\textwidth}
            \centering
            \begin{tikzpicture}
            \node[base] (1) [] {};
            \node[base] (2) [right of = 1] {};
            \node[base] (3) [above of = 2] {};
            \path[->] (2) edge [] node {} (1); 
            \path[->] (1) edge [] node {} (3);
            \path[->] (3) edge [double, thick ] node {} (2);
            \end{tikzpicture}
            \caption{$Q_1$}
        \end{subfigure}
        \begin{subfigure}{.4\textwidth}
            \centering
            \begin{tikzpicture}
            \node[base] (1) [] {};
            \node[base] (2) [right of = 1] {};
            \node[base] (3) [above of = 2] {};
            \path[->] (1) edge [] node {} (2); 
            \path[->] (3) edge [] node {} (1);
            \path[->] (3) edge [] node {} (2);
            \end{tikzpicture}
            \caption{$Q_2$}
        \end{subfigure}\\
        \vspace{1pc}
        \begin{subfigure}{.4\textwidth}
            \centering
            \begin{tikzpicture}
            \node[base] (1) [label=$Q_1$] {};
            \node[base] (2) [right of = 1,label=$Q_2$] {};
            \path[->] (2) edge [bend right] node {} (1); 
            \path[->] (1) edge [bend right] node {} (2);
            \path[->] (1) edge [loop left ] node {} (1);
            \path[->] (1) edge [loop below] node {} (1);
            \path[->] (2) edge [loop right ] node {} (2);
            \path[->] (2) edge [loop below] node {} (2);
            \end{tikzpicture}
            \caption{Directed mutation Graph.}
        \end{subfigure}
        \begin{subfigure}{.4\textwidth}
            \centering
            \begin{tikzpicture}
            \tikzset{every loop/.style={}}
            \node[base] (1) [label=$Q_1$] {};
            \node[base] (2) [right of = 1,label=$Q_2$] {};
            \path[-] (2) edge [] node {} (1); 
            \path[-] (1) edge [loop below ] node {} (1);
            \path[-] (2) edge [loop below ] node {} (2);
            \end{tikzpicture}
            \caption{Undirected mutation graph.}
        \end{subfigure}
        \caption{The quiver mutation graphs for $\Affine{A}_2$.}
        \label{fig:QuiverMutationGraphExample}
    \end{figure}

   \end{example}

\subsection{Reddening elements}    

    If a quiver $Q$ has a reddening sequence, then there is a unique element $r \in \Gamma_Q$ called the ``reddening element'' of $\Gamma_Q$.
    
    Explicitly, $r = (P_r,\sigma_P)$ where $P_r$ is any reddening sequence and $\sigma_P : Q \rightarrow P(Q) $ is the isomorphism which extends to an isomorphism $\Check{Q} \rightarrow P(\hat{Q})$ by adding the identity permutation on all of the frozen vertices.  
    
    The following theorem is proved as Corollary 3.7 of \cite{goncharov_donaldsonthomas_2018}, but we give a proof for completeness. 
    \begin{theorem}
    The reddening element (when it exists) is in the center of $\Gamma_Q$.
    \end{theorem}
    \begin{proof}
    To show $r$ is in the center we take any other group element $g = (P,f)$. Using the labeling induced by the initial framing the permutation $\sigma_r$ is the identity. Then 
    \begin{equation}
    g\cdot r \cdot g^{-1} = (P \cdot f(P_r)\cdot f(\sigma_r(f^{-1}(\revPath{P}))), f\circ \sigma_r \circ f^{-1}) = (P\cdot f(P_r) \cdot\revPath{P}, \text{id})
    \end{equation} Conjugating the reddening path $P_r$ by any other path again produces a reddening sequence (see \cite{Muller_maximalgreensequences}) so 
    \begin{equation}
        P_r \sim P \cdot f(P_r)\cdot \revPath{P}
    \end{equation} and we have $r = g r g^{-1}$ as needed.
    \end{proof}

\subsection{Folding cluster modular groups}
We will use folding to understand the cluster modular groups of non-simply laced affine and doubly extended cluster algebras. 

Let $R=Q_G$ be a weighted quiver with unique unfolding given by $Q$ and $G \subset \Aut{Q}$ a saturated subgroup. 
We denote by $N_{\Gamma_Q}(G)$ the normalizer of the subgroup $G \subset \ClusterModularGroup{Q}$.
\begin{theorem}\label{thm:FoldingClusterModularGroups}
    Every non identity element $\gamma \in \ClusterModularGroup{R}$ lifts to a non identity element of $\ClusterModularGroup{Q}$ which normalizes $G$.  This gives a well defined injection $\ClusterModularGroup{R} \rightarrow N_{\Gamma_Q}(G)/G$.
\end{theorem}
\begin{proof}
    Given a mutation path, $P$, in $\ClusterAlgebra{R}$ we obtain a lift $\tilde{P}$ to $\ClusterAlgebra{Q}$ by replacing each instance of mutation at a node $i$ with the corresponding group mutation $K_i$, noting that the order mutations within each group does not matter. 
    
    For an isomorphism $\tau:R \to P(R)$ we can define a lift $\tilde{\tau}: Q\to \tilde{P}(Q)$ as follows: Since $Q$ is a unique unfolding of $R$ we know that $\tilde{P}(Q)$ must be isomorphic to $Q$. Therefore we simply pick $\tilde{\tau}$ to be an isomorphism which restricts to $\tau$ after folding.   

    We now show that $(\tilde{P},\tilde{\tau})$ normalizes $G$. We compute
    \begin{align*}
        (\tilde{P},\tilde{\tau})\cdot (\langle\rangle, G) =& (\tilde{P},\tilde{\tau}G)\\
        (\langle\rangle, G) \cdot (\tilde{P},\tilde{\tau}) =& (G(\tilde{P}),G\tilde{\tau}) = (\tilde{P},G\tilde{\tau})
    \end{align*}
    The final equality is due to the fact that $G$ simply shuffles the order of elements inside a group mutation which does not effect the resulting path. Thus it suffices to show $\tilde{\tau}$ normalizes $G$. Given a labeling of $Q$ we realize $G \subset \Aut{Q}$ as a subset of $S_n$. We also view $\tilde{\tau}$ as an element of $S_n$ by recording the permutation on the labels of $\tilde{P}(Q)$ induced by the isomorphism. By \Cref{lem:fold_aut_2} and the fact that $G$ is saturated, we see that a group mutation preserves the subgroup of $S_n$ corresponding to $G$. Therefore for any $g\in G$, the element $\tilde{\tau}^{-1}g \tilde{\tau}$ fixes all the groups of the folding and thus by saturation is again an element of $G$. Thus $\tilde{\tau}$ normalizes $G$ as needed.

    Next we verify the lift is well defined in $N_{\Gamma_Q}(G)/G$. Let $\tilde{\tau}$ and $\tilde{\tau}'$ be two lifts of $\tau$. Then $\tilde{\tau}^{-1}\tilde{\tau}$ is an element of $\Aut{Q}$ which restricts to the identity on $R$. Since $G$ is saturated this implies $\tilde{\tau}^{-1}\tilde{\tau}' \in G$ and the coset of $(\tilde{P},\tilde{\tau})G$ is well defined.

    Finally, we can see that this element is not the identity since it is not the identity on the folded cluster algebra $\ClusterAlgebra{R}$ which is isomorphic to $\ClusterAlgebra{(Q,\mathbb{K}_G)}$ where the lifted element more clearly acts. 
    
\end{proof}

\begin{remark}
    We expect this map to actually be an isomorphism of cluster modular groups. The injection is enough for our purposes since it gives a faithful description of the cluster modular group of each non-simply laced affine and doubly-extended type cluster algebra within the cluster modular group of a simply laced one. 
\end{remark}

\section{Type \texorpdfstring{$\T$}{Tnw} Cluster Algebras}\label{sec:TnwQuivers}

In this section we will consider a family of quivers  $\T$ for $\vec{n},\vec{w}$ equal length vectors of positive integers, and their associated cluster algebras. These algebras each have a canonical subgroup of the cluster modular group with a simple description in terms of ``twist mutation paths'' and automorphisms of quivers. We call a cluster algebra ``type $\T$'' if it has a seed with a $\T$ quiver underlying it.

We then show in \Cref{sec:AffineClusterAlgebra} and \Cref{sec:DoubleExtended} that each of the affine-type and doubly extended cluster algebras are type $\T$ for certain values of $\vec{n}$ and $\vec{w}$. We show that that the canonical subgroup is the cluster modular group of each affine type cluster algebra. In the doubly-extended case, we will find that this subgroup along with one extra element generates the cluster modular group. We conjecture that in all other cases, this canonical subgroup is exactly the cluster modular group.

\subsection{\texorpdfstring{$\T$ quivers}{T quivers}}

Let $\vec{n}= (n_1,n_2,\dots,n_m)$, $n_i>1$ and $\vec{w}= (w_1,w_2,\dots,w_m)$ be $m$ tuples of positive integers. We consider a weighted quiver, $\T$, with $n=\sum(n_i-1)+2$ nodes constructed in the following way: 
First consider the star shaped quiver $T'_{\vec{n},\vec{w}}$ with $n-1$ nodes consisting of one central node, $N_1$ of weight 1 and $m$ tails of length $n_i-1$ of weight $w_i$ nodes $i_2, \dots, i_{n_i}$ connected in a source-sink pattern with $N_1$ as a source (\Cref{fig:QT'}).

\begin{figure}
    \centering
    \begin{tikzcd}
                  & 2_{n_2} & \dots \arrow[r] \arrow[l] & 2_2 &                                               & 3_2 \arrow[d, no head, dotted] & \dots \arrow[l] \\
1_{n_1} \arrow[r] & \dots   & 1_3 \arrow[r] \arrow[l]   & 1_2 & N_1 \arrow[l] \arrow[lu] \arrow[ru] \arrow[r] & m_2                            & \dots \arrow[l]
\end{tikzcd}
    \caption{The quiver $T'_{\vec{n},\vec{w}}$.}
    \label{fig:QT'}
\end{figure}

\begin{figure}
    \centering
    \begin{tikzcd}
                  & 2_{n_2} & \dots \arrow[r] \arrow[l] & 2_2 \arrow[r]  & N_\infty \arrow[d, Rightarrow]                & 3_2 \arrow[d, no head, dotted] \arrow[l] & \dots \arrow[l] \\
1_{n_1} \arrow[r] & \dots   & 1_3 \arrow[r] \arrow[l]   & 1_2 \arrow[ru] & N_1 \arrow[l] \arrow[lu] \arrow[ru] \arrow[r] & m_2 \arrow[lu]                           & \dots \arrow[l]
\end{tikzcd}
    \caption{The quiver $\T$.}
    \label{fig:QT}
\end{figure}

$\T$ is constructed from $T'_{\vec{n},\vec{w}}$ by adding an additional weight 1 node $N_\infty$ along with a double arrow from $N_\infty$ to $N_1$ and single arrows from each of the $m$ other neighbors of $N_1$ to $N_\infty$, as shown in \Cref{fig:QT}.

When $m\leq 3$ and $w_i=1$ for all $i$, we let $(p,q,r)=(n_1,n_2,n_3)$ with $p,q,r$ possibly equal to 1 and write $T_{p,q,r}$ for $\T$. 

\begin{definition}
The nodes $i_j$ are called the \emph{tail nodes} of $\T$. The nodes $i_2$ are called the \emph{boundary} tail nodes. The \emph{$i$th tail subquiver} is the quiver obtained by removing all of the tail nodes $k_j, k \neq i$.
\end{definition}

Our motivation for considering these quivers is based on the following remark: 

\begin{theorem}\label{rem:dynkin_types}
Let $\chi(\T) = \sum(w_i(n_i^{-1}-1))+2$.
If $\chi>0$ then $\T$ has a (non-twisted) affine Dynkin quiver in its mutation class and $T'_{\vec{n},\vec{w}}$ is a finite Dynkin quiver. If $\chi=0$ then $\T$ is a doubly extended Dynkin quiver and $T'_{\vec{n},\vec{w}}$ is an affine Dynkin quiver.
\end{theorem}
The first statement will be proved in \Cref{sec:AffineClusterAlgebra}. The second statement can be verified by checking the finitely many cases where $\chi = 0$ (\Cref{fig:doubleExtendedOptions}).
\begin{remark}\label{rem:FoldingT_nw}
$\chi$ is preserved by replacing a length $n$ tail with weight $w$ with $w$ weight 1 tails of length $n$. This follows the idea that higher weight nodes can be analyzed by folding larger quivers. 
\end{remark}

\begin{remark}\label{rem:middle_nodes_TQ}
The middle two nodes of a $\T$ quiver as we have described always have weight 1. The twisted affine types will have quivers which look like $\T$ quivers in their mutation classes, but with weighted nodes in the middle positions. The non-BC twisted affine types are dual to ordinary affine quivers. However the type BC twisted affine quivers are special. For example, the type $\TwistedAffine{BC}{4}_n$ quivers have the following quiver in their mutation class:
\begin{equation}
\begin{tikzpicture}
\node[fat4] (1) [] {};
\node[base] (2) [below of = 1] {};
\node[fat2] (3) [right of = 2] {};
\node[fat2] (4) [right of = 3] {};
\node[invis] (5) [right of = 4] {\dots};
\node[fat2] (6) [right of = 5] {};

\path[->] (1) edge [] node {} (2); 
\path[->] (2) edge [] node {} (3);
\path[->] (3) edge [] node {} (1); 
\path[->] (4) edge [] node {} (3); 
\path[->] (4) edge [] node {} (5);
\path[->] (6) edge [] node {} (5); 
\end{tikzpicture}
\end{equation}
\end{remark}

In light of this remark we will define a $BC$ variant of $\T$ quiver denoted $T^{BC}_{\vec{n}}$ which will have the $\TwistedAffine{BC}{4}_n$ types in their mutation class. 
\begin{definition}
A $T^{BC}_{\vec{n}}$ quiver consists of two middle nodes of weight 4 and 1 with a single arrow between them and tails of weight 2 nodes of length $n_i$, see \Cref{fig:TBCQuiver}. We define 
\begin{equation}
    \chi(T^{BC}_{\vec{n}}) = \sum_i (\frac{1}{n_i} - 1) + 1
\end{equation}
\end{definition}

\begin{figure}
    \centering
\begin{tikzpicture}
\node[fat4] (1) [] {};
\node[base] (2) [below of = 1] {};
\node[fat2] (3) [right of = 2] {};
\node[invis] (4) [right of = 3] {\dots};
\node[fat2] (5) [right of = 4] {};
\node[fat2] (6) [left of = 2] {};
\node[invis] (7) [left of = 6] {\dots};
\node[fat2] (8) [left of = 7] {}; 
\node[fat2] (9) [left of = 1] {};
\node[invis] (10) [left of = 9] {\dots};
\node[fat2] (11) [left of = 10] {}; 

\path[->] (1) edge [] node {} (2); 
\path[->] (2) edge [] node {} (3);
\path[->] (3) edge [] node {} (1); 
\path[->] (3) edge [] node {} (4);
\path[->] (5) edge [] node {} (4); 
\path[->] (2) edge [] node {} (6);
\path[->] (6) edge [] node {} (1); 
\path[->] (6) edge [] node {} (7);
\path[->] (8) edge [] node {} (7); 
\path[->] (2) edge [] node {} (9);
\path[->] (9) edge [] node {} (1); 
\path[->] (9) edge [] node {} (10);
\path[->] (11) edge [] node {} (10); 
\end{tikzpicture}
    \caption{A $T^{BC}_{\vec{n}}$ quiver with 3 tails}
    \label{fig:TBCQuiver}
\end{figure}
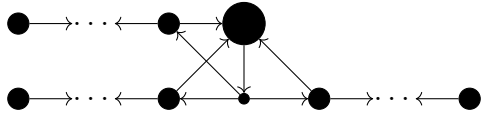

\subsection{The cluster modular group of a \texorpdfstring{$\T$}{Tnw} cluster algebra}\label{sec:TnwModularGroup}
We will construct a subgroup, $\Gamma_\tau$, of the cluster modular group of a $\T$ cluster algebra generated by ``twist'' mutation paths associated with each tail. The automorphism group $\Aut{\T}$ acts on $\Gamma_\tau$ by permuting twists associated to tails of the same length and weight.

Let 
\begin{equation}
    i_{\text{odd}} = \{i_j |3\leq j \leq n_i, j \text{ odd}\}  \text{ and } i_{\text{even}} = \{i_j | 3\leq j \leq n_i, j \text{ even}\} .
\end{equation}

\begin{definition}\label{def:tau}
We have a twist $\tau_i \in \Gamma_\tau$ given by the following mutation paths depending on $w_i$:
\begin{align}
    w_i = 1 \hspace{5mm} & \text{let } \tau_i = \{i_{\text{odd}}i_{\text{even}}i_{2} N_\infty N_1 , (i_2 N_1 N_\infty)\} \\
    w_i = 2 \hspace{5mm} & \text{let } \tau_i = \{i_{\text{odd}}i_{\text{even}}i_{2} N_\infty N_1 i_2 N_1, id\} \\
    w_i = 3 \hspace{5mm} & \text{let } \tau_i = \{i_{\text{odd}}i_{\text{even}}i_{2} N_\infty N_1 i_2 N_\infty i_2 N_1, id\} 
\end{align}
When $w_i \geq 4$ there is no twist for tail $i$.
\end{definition}
Let $\gamma = \{N_\infty,(N_1N_\infty)\}$, which we think of as a twist of a tail of length 1.
\begin{definition}
  \emph{$\Gamma_\tau$ } is the group generated by all of twists, $\tau_i$, and $\gamma$.
\end{definition}
\begin{remark}\label{rem:FoldingTwists}
Once again we see the importance of using folding to understand weighted quivers. One can verify that when $w_i = 2$, $\tau_i$ is the same as replacing tail $i$ with two tails of the same length twisting each of them and then refolding into a tail of weight 2. 
The same holds for splitting into 3 tails when $w_i = 3$. However when $w_i = 4$  mutation at $i_2$ reverses the direction of the double edge without mutating at $N_1$ or $N_\infty$ and so there is no possible equivalent twist of 4 tails. When $w_i > 4$ mutation at $i_2$ results in edge of weight higher than 2; this situation only happens in infinite mutation type cluster algebras which we don't consider for the remainder of the paper.
\end{remark}
\begin{definition}
The group $\Gamma_{\T} = \Gamma_\tau \rtimes \Aut{\T}$ is the canonical subgroup of the cluster modular group of a $\T$ type cluster algebra. 
\end{definition}

\begin{conjecture}
If $\chi(\T) \neq 0$ then the cluster modular group of a type $\T$ cluster algebra is exactly $\Gamma_{T_{\vec{n},\vec{w}}}$.
\end{conjecture}

We have the following theorem:

\begin{theorem}\label{thm:TwistPowers}
$\Gamma_\tau $ is an abelian group and the only relations are $\tau_i^{n_i}=\gamma^{w_i}$.
\end{theorem}
\begin{proof}
In order to show that $\Gamma_\tau$ is abelian, we simply need to check that two twists tails of length 2 commute with each other and with $\gamma$. This is because the additional mutations which appear as the tail length increases always happen at sources. Thus they don't change the adjacency of the quiver and stay disconnected from the other tail through the entire path. Therefore all that remains is a simple computation to check commutativity for each possible combination of weights for tails of length 2. 

We now focus on a single tail of length $n$ and weight $1$ and show that  $\tau^{n} = \gamma$. It suffices to look at $T_{(n),(1)}$ since $\tau_i$ only mutates at vertices on tail $i$. In \Cref{sec:AffineProofs} we see that this quiver is associated to an annulus with $n$ marked points on the interior (labeled $v_1,\dots,v_{n}$ clockwise) and one marked out on the outer boundary component. In \Cref{lem:TwistGroup} we see that $\tau$ corresponds to rotating the interior circle by $\frac{2\pi}{n}$ radians and $\gamma$ is the full Dehn twist. So $\tau^n$ is a full rotation and is equal to $\gamma$.

The previous remark completes the theorem when $w_i > 1$.
\end{proof}

\begin{remark}\label{rem:Tgroup}
Let $\ell=\prod n_i$. We may view $\Gamma_\tau$ as the subgroup of $\mathbb{Z}\times\prod \cyclicgroup{n_i}$ generated by the elements $\gamma = (\ell,0,\dots,0)$ and $\tau_i = (w_i \ell/{n_i},0,\dots,1,\dots,0)$. Let $\Gamma_\tau^\circ$ be the kernel of the projection $\Gamma_\tau  \rightarrow \mathbb{Z}$. Then $\Gamma_\tau \simeq \Gamma_\tau^\circ \rtimes \mathbb{Z}$. We remark that this isomorphism relies on the identification of $d\Z$ with $\Z$ where $d = \gcd(\ell, w_i \ell/n_i)$, since the projection onto the first $\Z$ factor is not surjective.
\end{remark}

\begin{remark}
When there are zero tails, $T_{(),()}$ is just a double edge. It is clear in this case $\gamma^2$ is the reddening element. This generalizes to the following theorem.
\end{remark}
\begin{theorem}\label{thm:tnw_reddening}
The element $r \in \Gamma_\tau$ given by $r=\gamma^2\prod_i(\tau_i\gamma^{-w_i})$ is the reddening element of $T_{\vec{n},\vec{w}}$. 
\end{theorem}
\begin{proof}
Suppose that $m=1$.  It is a simple computation to check this statement for each possible weight when $n_1= 2$. Then, for $n_1 > 2$ we can see that the mutating at $i_{\text{even}}i_{\text{odd}}$ always mutates at a source and so is a reddening sequence for the non-boundary nodes of the tail. Since $1_3$ is initially connected towards $1_2$, we now have $1_2$ out to $1_3$. Finally, we can complete the reddening sequence by using the the $n=2$ case.

Now consider $m > 1$. Let $r_i = \gamma^2\tau_i\gamma^{-w_i} $ be the reddening element for the $i$th tail subquiver. Rewrite $r$ as follows
\begin{equation}\label{eq:r_factor}
    r = \gamma^2\prod_i(\tau_i\gamma^{-w_i}) = \big(\prod_i(r_i\gamma^{-2})\big)\gamma^2
\end{equation}
We can see that this element is reddening by noting that $r_i\gamma^{-2}$ has the effect of reddening the nodes on the tail $i$, while keeping the middle two nodes green. Moreover no additional frozen arrows are connected to any other tail. Thus for each $i$, the element $r_i$ always gets applied to an all green subquiver. Therefore, the effect of the product of elements of the right hand side of \Cref{eq:r_factor} is to make all of the nodes other than the middle two red. Finally applying $\gamma^2$ makes all the nodes red. This element returns us to an isomorphic quiver without permuting any of the frozen nodes, and is thus the reddening element.

\end{proof}

\begin{corollary}
When $\chi>0$, $r$ is a conjugation of the source-sink mutation path on the corresponding affine Dynkin diagram. 
\end{corollary}
This corollary follows since the reddening element of an affine Dynkin diagram is the source-sink mutation path. Then since $\chi > 0$ implies there is an quiver corresponding to an affine Dynkin diagram, \Cref{thm:MullerReddening} states the two reddening elements must be conjugate.

\begin{remark}
In terms of the group presentation of \Cref{rem:Tgroup}, the reddening sequence is given by the element $(\chi\ell ,1,1,\dots,1)$.
\end{remark}

\subsection{BC type quivers}

The $T^{BC}_{\vec{n}}$ type quivers have an analogous abelian subgroup, $\Gamma_\tau = \groupgenby{\tau_i, \gamma | \tau_i^{n_i} = \tau_j^{n_j} = \gamma}$, generated by twists of the tails. $\gamma$ is the mutation path consisting of mutation at the weight 4 node and then the weight 1 node and the twist paths are the same twist paths in the $w_i = 2$ case of a regular $\T$ quiver. 

The reddening element is given by 
\begin{equation}
    r=\gamma\prod_i(\tau_i\gamma^{-1}).
\end{equation}

\section{Affine Cluster Algebras}\label{sec:AffineClusterAlgebra}

Our analysis of the cluster modular group of the affine cluster algebras stems from the observation in \Cref{rem:dynkin_types}. Our primary goal is the following theorems:

\begin{theorem}\label{thm:AffineType}
The cluster algebra associated to the quiver $\T$ is of affine type if and only if $\chi > 0$. Furthermore, every affine type cluster algebra has a seed whose quiver is a  $\T$ or $T^{BC}_{\vec{n}}$ with $\chi > 0$. 
\end{theorem}
\begin{theorem}\label{thm:AffineModularGroup}
The cluster modular group of a cluster algebra of affine type is  $\Gamma_\tau \rtimes \Aut{\T}$, where the action of the automorphism group is by permuting the twists of tails of the same weight and length. 
\end{theorem}
The full proofs of \Cref{thm:AffineType,thm:AffineModularGroup} are given in \Cref{sec:AffineProofs}. The association between the affine types and values of $\vec{n}$ and $\vec{w}$ is given in \Cref{fig:affineOptions}. The following well known Lemma (included for completeness) proves that this is every possible option for $\chi > 0$.

\begin{figure}[hb]
    \centering
    \begin{subfigure}{0.49\textwidth}
        \begin{tabular}{c c c c }
     Type &  $\vec{n}$ & $\vec{w}$ & $\chi^{-1}$ \\
     \hline\hline
     $A_{1,1}$ & () & () & 2\\
      $A_{p,q}$   & $(p,q)$ & $(1,1)$ & $\frac{pq}{p+q}$\\
      $\Affine{D}_{n}$   & $(n-2,2,2)$ & $(1,1,1)$ & $n-2$ \\
      $\Affine{E}_6$ & $(3,3,2)$ & $(1,1,1)$ & $6$\\
      $\Affine{E}_7$ & $(4,3,2)$ & $(1,1,1)$ & $12$\\
      $\Affine{E}_8$ & $(5,3,2)$ & $(1,1,1)$ & $30$\\
      \end{tabular}
    \end{subfigure}
    \begin{subfigure}{0.5\textwidth}
    \begin{tabular}{c c c c}
     Type &  $\vec{n}$ & $\vec{w}$ & $\chi^{-1}$ \\
     \hline\hline
      $\Affine{C}_n $ & $(n)$ & $(2)$ & $\frac{n}{2}$\\
      $\Affine{B}_{n} $ & $(n-1,2)$ & $(1,2)$ & $n-1$\\
      $\Affine{F}_{4} $ & $(3,2)$ & $(2,1)$ & 6\\
      $\Affine{G}_{2}$ & $(2)$  & $(3)$ & 2\\
      \hline
      $\TwistedAffine{BC}{4}_{n}$ \\(BC-Type) & $(n)$ & - & $n$\\
    \end{tabular}
            
    \end{subfigure}
    \caption{All possible values of $\T$ that result in affine cluster algebras.}
    \label{fig:affineOptions}
\end{figure}

\begin{lemma}
There are finitely many families of $(\vec{n},\vec{w})$ such that $\chi > 0$.
\end{lemma}
\begin{proof}
Following \Cref{rem:FoldingT_nw}, we begin with the case where every tail has weight 1. If $\chi > 0$, we need: \[\sum\frac{1}{n_i} > m-2 \]
The only options for $\vec{n}$ are 
$(n)$, $(p,q)$, $(n,2,2)$, $(3,3,2)$, $(4,3,2)$ and $(5,3,2)$. \\
Then the higher weight tails come from folding the above cases. When $p=q$ we can fold to obtain $((p),(2))$. Similarly the two length 2 tails in $(n,2,2)$ and the length 3 tails of $(3,3,2)$ can be folded to obtain $((n,2),(1,2))$ and $((3,2),(2,1))$. Finally we can fold $(2,2,2)$ to obtain $((2),(3))$.

The $BC$ case follows easily by direct inspection.
\end{proof}

\begin{remark}
These cluster modular groups have already been computed \cite{Schiffler:cluster_automorphisms} based at the Dynkin type quivers. However the computations based at the $\T$ quivers allows for a uniform treatment of the affine and doubly extended cluster algebras.
\end{remark}

\subsection{The normal subgroup generated by \texorpdfstring{$\gamma$}{gamma}}

Our goal now is to construct a natural finite quotient of the exchange graphs and cluster complexes of each of the affine cluster algebras. We dualize the quotient cluster complexes to produce an ``affine generalized associahedron''.

    The subgroup of the cluster modular groups of the affine $T_{\vec{n},\vec{w}}$ quivers generated by $\gamma = \{N_\infty,(N_1N_\infty)\}$ is a normal, finite index subgroup. 
    
    \begin{remark}
    In the $\Affine{A}$ case, this subgroup can be seen to be given by the mapping class group action on the triangulations of an annulus. We therefore consider this subgroup to be an analog to the mapping class group in each of the affine cases. 
    \end{remark}
    
    \Cref{fig:AffineGroupsAndQuotients} shows the cluster modular groups and quotients by the subgroup generated by $\gamma$.
    
    \begin{figure}

    \begin{center}
         \begin{tabular}{c c c}
            Affine Type & Cluster Modular Group & Quotient\\
            \hline\hline
            $A_{p,p}$  & $D_{2p} \rtimes \mathbb{Z}$ &  $ (\cyclicgroup{p} \times \cyclicgroup{p} ) \rtimes \cyclicgroup{2}$\\ 
            
            $A_{p,q}$  & $\cyclicgroup{gcd(p,q)} \times \mathbb{Z}$ & $\cyclicgroup{p} \times \cyclicgroup{q}$  \\
            
            $\Affine{D}_4$ & $S_4 \times \mathbb{Z}$ & 
            $ (\cyclicgroup{2} \times \cyclicgroup{2} \times \cyclicgroup{2}) \rtimes S_3$ \\
            
            $\Affine{D}_n$ Even & $(\cyclicgroup{2} \times \cyclicgroup{2}) \rtimes \cyclicgroup{2} \times \mathbb{Z}$ 
            & $ \cyclicgroup{n-2} \times (\cyclicgroup{2} \times \cyclicgroup{2}) \rtimes \cyclicgroup{2}$ \\
            
            $\Affine{D}_n$ Odd  & $ (\cyclicgroup{2} \times \cyclicgroup{2}) \rtimes \mathbb{Z}$ 
            & $ \cyclicgroup{n-2} \times (\cyclicgroup{2} \times \cyclicgroup{2}) \rtimes \cyclicgroup{2}$ \\
            
            $\Affine{E}_6$ & $ S_3 \times \mathbb{Z}$ & 
            $ \cyclicgroup{2} \times ( \cyclicgroup{3} \times \cyclicgroup{3}) \rtimes \mathbb{Z}_2$  \\
            
            $\Affine{E}_7$ & $ \cyclicgroup{2} \times \mathbb{Z}$ & $\cyclicgroup{2} \times \cyclicgroup{3} \times \cyclicgroup{4}$\\
            
            $\Affine{E}_8$ & $ \mathbb{Z}$ & $ \cyclicgroup{2} \times \cyclicgroup{3} \times \cyclicgroup{5}$ \\
            
            \hline 
            
            $\Affine{C}_n $ & $\mathbb{Z}$ & $\cyclicgroup{n}$ \\
            $\Affine{B}_{n} $ & $\cyclicgroup{2} \times \mathbb{Z} $ & $\cyclicgroup{n-1} \times \cyclicgroup{2}$ \\
             $\Affine{F}_{4} $ & $\mathbb{Z} $ & $\cyclicgroup{2} \times \cyclicgroup{3}$ \\
             $\Affine{G}_{2}$ & $\mathbb{Z}$  & $\cyclicgroup{2}$ \\
             \hline
             $\TwistedAffine{BC}{4}_{n}$ & $\mathbb{Z}$ & $\cyclicgroup{n}$ 
         \end{tabular}

    \end{center}
        \caption{Affine cluster modular groups and their quotients }
        \label{fig:AffineGroupsAndQuotients}
    \end{figure}

    We wish to understand the quotient of the exchange graph of an affine cluster algebra by the action of the group $\groupgenby{\gamma}$. A possible way to accomplish this is by introducing a special framing of a $T_{\vec{n},\vec{w}}$ quiver, and compute the graph by identifying the clusters via their c-vectors as usual. 
    
    Consider the quiver $T^f_{\vec{n},\vec{w}}$ obtained from $\T$ by adding a frozen node for vertices $i_2,\dots,i_{n_i}$ in each tail and one vertex associated with the double edge. In particular for each tail $i$ add frozen nodes of weight $w_i$ labeled $f_{i,2}, \dots ,f_{i,n_i}$ with a single arrow from $i_j$ to $f_{i,j}$. Then add a frozen node $f_1$ of weight 1 along with single arrows $N_1$ to $f_1$ and $f_1$ to $N_\infty$. 
    
    \begin{conjecture}\label{con:SpecialFraming}
    Two quivers in the exchange graph of $Q=T_{\vec{n},\vec{w}}$ are in the same orbit of the action of $\groupgenby{\gamma}$ if and only if the projection of those quivers in the exchange graph of $T^f_{\vec{n},\vec{w}}$ is the same. 
    \end{conjecture}
    
    The ``if'' part of the statement follows since the framing is preserved by the action of $\gamma$. However, it is not clear that the only quivers which are identified are the ones which are in the same $\gamma$ orbit. 
    
\subsection{Affine associahedra}
    Recall the cluster complex associated to a finite cluster algebra has a dual complex called the ``generalized associahedon''. We cannot simply dualize an affine cluster complex immediately as there are vertices in the cluster complex with infinite degree. This is because there are cluster variables that are compatible with infinitely many other cluster variables, and so occur in infinitely many seeds. However, up to action of $\gamma$, there are only finitely many cluster variables. If we quotient the cluster complex by the action of $\gamma$, we obtain a finite cell complex. 
    
    In order to construct a dual complex, we need to see that the quotient complex is a ``combinatorial cell complex''. Technically, the quotient by $\gamma$ is not combinatorial because there are facets that contain multiple cluster variables in the same orbit. For example in \Cref{fig:A21CluserComplex} we see the cluster complex for a quiver of type $A_{2,1}$. Here $\gamma$ fixes the top and bottom vertex and translates the interior rows one step to the right. This results in triangles with multiple vertices in the same orbit.
    
    Instead we quotient by $\gamma^3$ which ensures that every maximal facet corresponds to a unique collection of distinct orbits. There is still a finite number of clusters up to $\gamma^3$ so by the work of \cite{basak:CombinatorialCellComplex} this complex has a dual cell complex. We then can quotient the dual by $\gamma$ to obtain the dual cell complex we originally desired.

\begin{figure}[hb]
    \centering
\begin{tikzcd}
   &                                                                                      &                                                                              & \bullet \arrow[lld, no head] \arrow[d, no head] \arrow[rrd, no head] \arrow[rrrrd, no head] \arrow[rrrrrrd, no head, dashed] &                                                                                                                            &                                                  &                             &                                                                                                   &                                                                 &    \\
{} & \bullet \arrow[rdd, no head] \arrow[l, no head, dashed] \arrow[ldd, no head, dashed] &                                                                              & \bullet \arrow[ll, no head] \arrow[rdd, no head]                                                                             &                                                                                                                            & \bullet \arrow[ll, no head] \arrow[rdd, no head] &                             & \bullet \arrow[ldd, no head] \arrow[rdd, no head] \arrow[ll, no head] \arrow[rr, no head, dashed] &                                                                 & {} \\
   &                                                                                      &                                                                              &                                                                                                                              &                                                                                                                            &                                                  &                             &                                                                                                   &                                                                 &    \\
{} &                                                                                      & \bullet \arrow[rr, no head] \arrow[ruu, no head] \arrow[ll, no head, dashed] &                                                                                                                              & \bullet \arrow[rr, no head] \arrow[ruu, no head]                                                                           &                                                  & \bullet \arrow[rr, no head] &                                                                                                   & \bullet \arrow[r, no head, dashed] \arrow[ruu, no head, dashed] & {} \\
   &                                                                                      &                                                                              &                                                                                                                              & \bullet \arrow[u, no head] \arrow[rru, no head] \arrow[rrrru, no head] \arrow[llu, no head] \arrow[llllu, no head, dashed] &                                                  &                             &                                                                                                   &                                                                 &   
\end{tikzcd}
    \caption{$A_{2,1}$ Cluster Complex}
    \label{fig:A21CluserComplex}
\end{figure}

    \begin{definition}
    Let $\ClusterComplex{A}$ be the cluster complex associated to the affine cluster algebra $\mathcal{A}$. The \emph{affine associahedron} is the dual complex to $\ClusterComplex{A}/\groupgenby{\gamma}$.  The 1-skeleton of an affine associahedron is the quotient exchange complex of an affine cluster algebra.
    \end{definition}
    
    \begin{remark}
    We could define an affine associahedron as a quotient of any power of $\gamma$. All of our analysis of the combinatorics of affine associahedra can be easily extended to a quotient by any other power of $\gamma$.
    \end{remark}
    
    \begin{example}
            The simplest example is the $A_{2,1}$ cluster algebra. In \Cref{fig:A2QuotientExchangeGraph} we see the full exchange graph extending infinitely in both directions. Below is the quotient associahedron. There are four folded 2-cells. Two correspond to the top and bottom pentagons and the remaining two correspond to the $A_{1,1}$ subalgebras. Despite the folding, this associahedron has the homology type of a sphere.  
        
        \begin{figure}[!hb]
        \centering
        \begin{subfigure}{.8\textwidth}
            \centering
            \includegraphics[width=\textwidth]{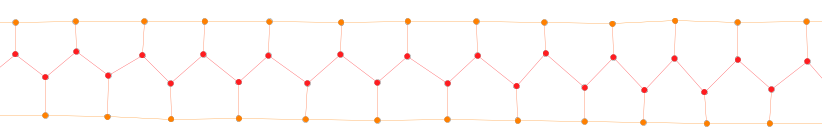}
            \caption{Exchange Graph}
            \label{fig:A2QuotientExchangeGraph}
        \end{subfigure}
        \begin{subfigure}{.5\textwidth}
            \centering
            \includegraphics[height=1in]{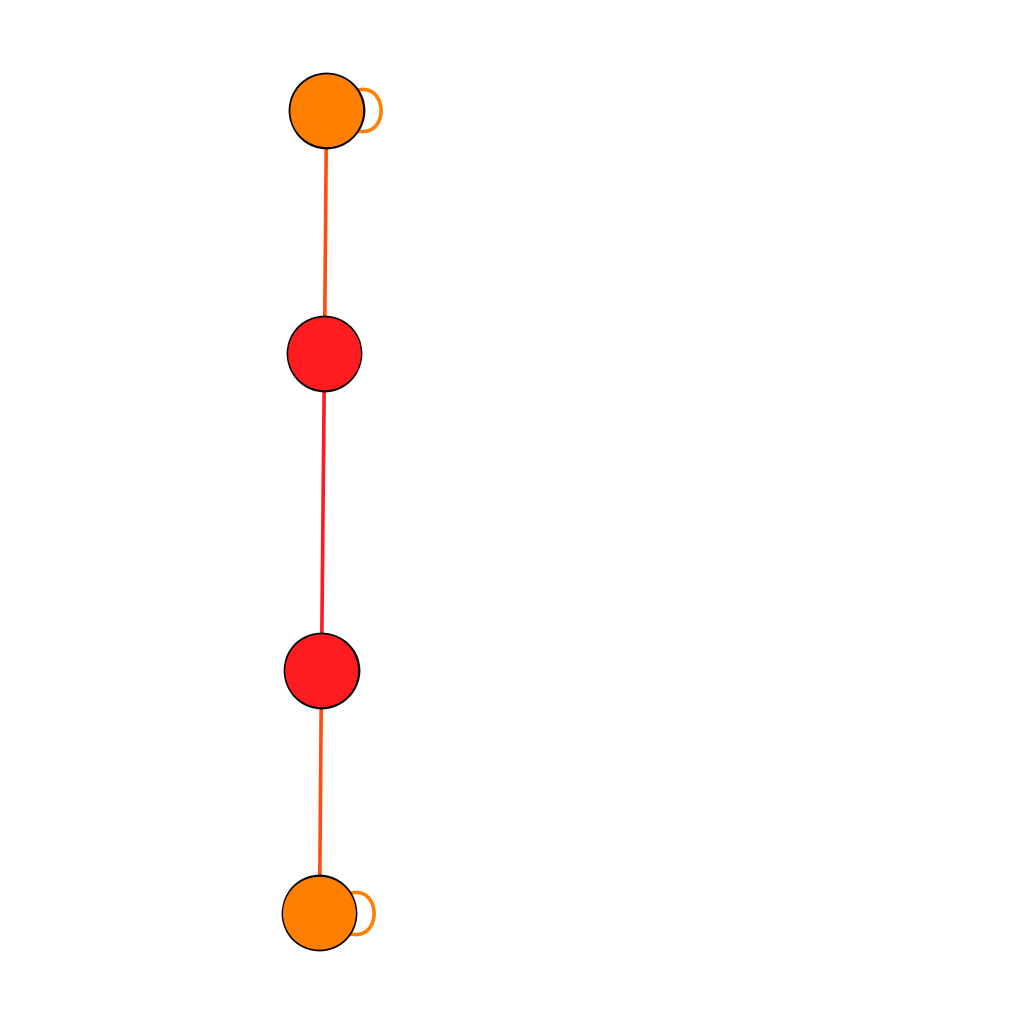}
            \caption{Associahedron}
            \label{fig:A2QuotientAssociahedron}
        \end{subfigure}
        \caption{$A_{2,1}$ Exchange Graph and Associahedron.}
        \label{fig:A2QuotientComplex}
    \end{figure}
    \end{example}
    
    \begin{example}
    A slightly more complicated example is $\Affine{D}_4$ in \Cref{fig:D4QuotientComplex}. Again we see the exchange graph extending infinitely in both directions. The 1-skeleton of the affine associahedron is shown. This graph was computed using the special framing mentioned in the previous section. This computation finds the correct number of $0-$cells in the associahedron, and thus confirms \Cref{con:SpecialFraming} in this case. The complete counts of all subalgebras in $\Affine{D}_4$ up to the action of $\gamma$ can be found in \Cref{fig:D4AffCounts}. The total counts of corank $k$ subalgebras is the number of codimension $k$ facets of the affine associahedron.
    
    \begin{figure}[!hb]
        \centering
        \begin{subfigure}{.5\textwidth}
            \centering
            \includegraphics[angle=90,width=\textwidth]{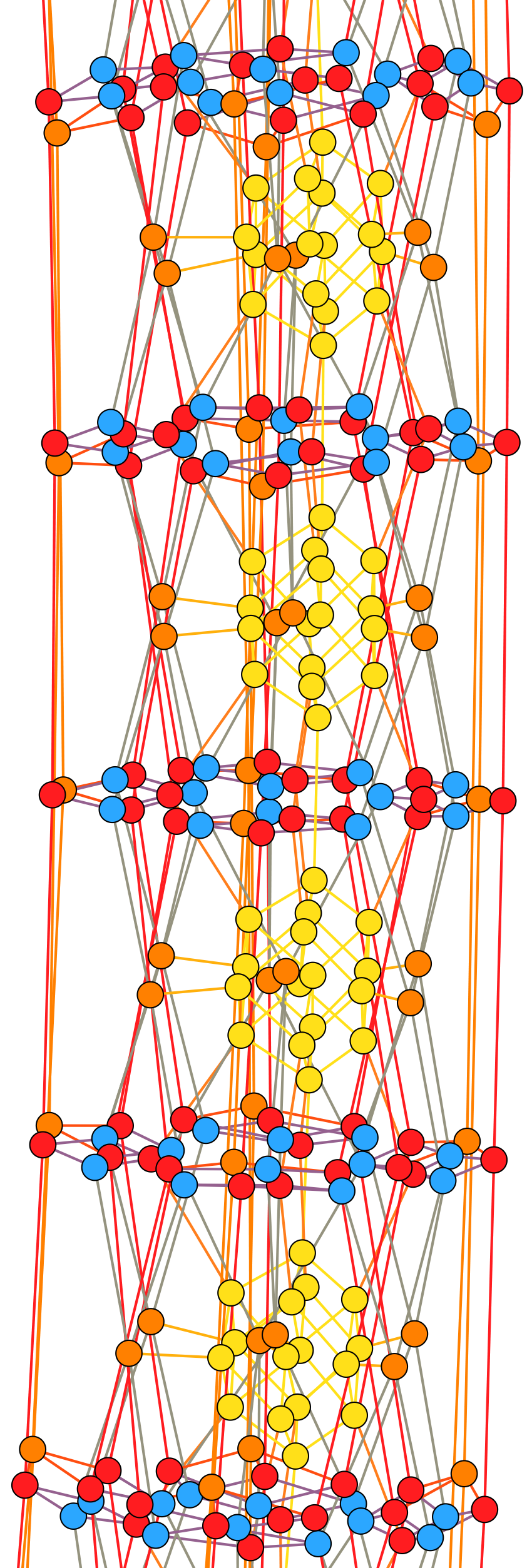}
            \caption{Exchange Graph}
        \end{subfigure}
        \begin{subfigure}{.4\textwidth}
            \centering
            \includegraphics[angle=15, width=\textwidth]{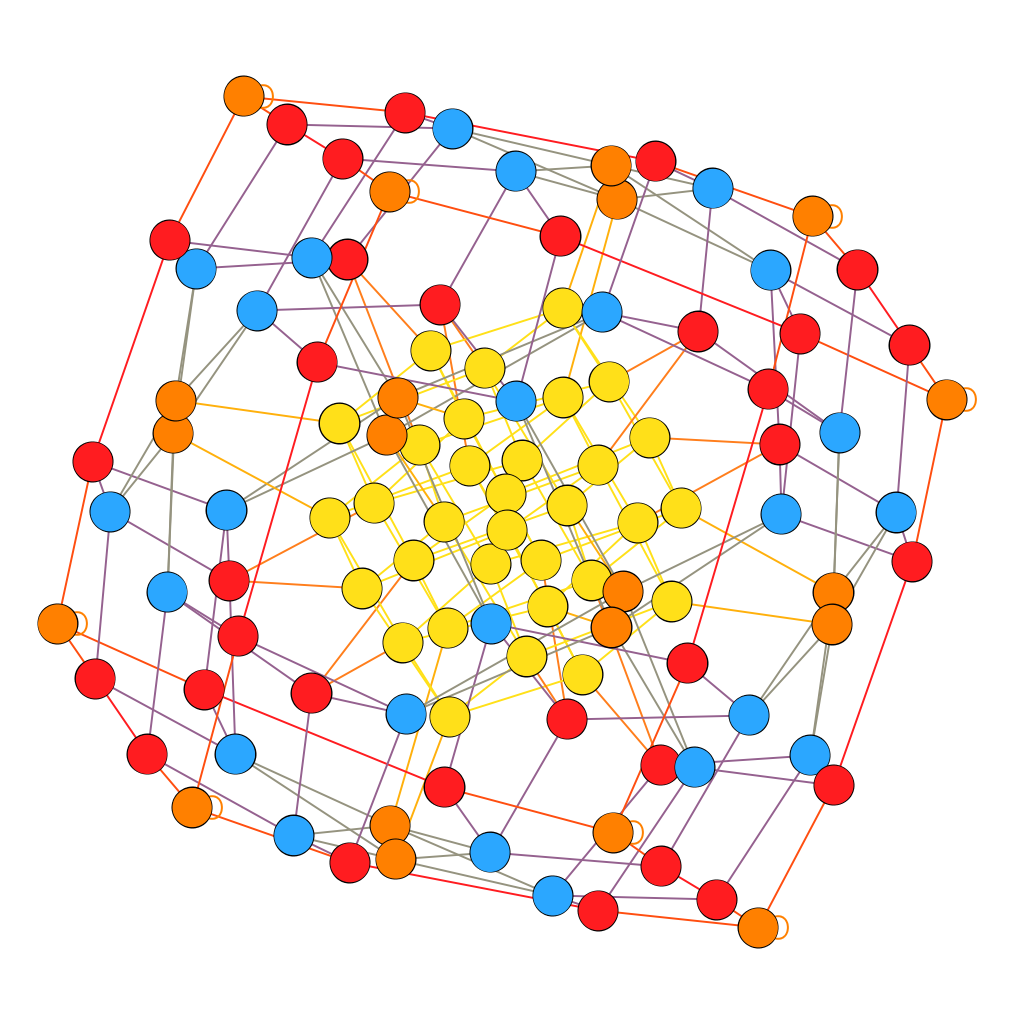}
            \caption{1-skeleton of the associahedron}
        \end{subfigure}
        \caption{$\Affine{D}_4$ exchange graph and affine associahedron.}
        \label{fig:D4QuotientComplex}
    \end{figure}
    
    \begin{figure}[!hb]
        \centering
        \begin{tabular}{c c c c c}
             Corank & \multicolumn{3}{c}{Subalgebra Types} & Total  \\
             \hline
             \hline
             \multirow{2}{*}{1} & $A_{2,2}$ & $ D_4$ & $D_2\times D_2$ & \multirow{2}{*}{16}\\
             & 6 & 8 & 2 & \\
             \hline
             \multirow{2}{*}{2} & $A_{2,1}$ &  $A_3$ & $A_1\times A_1\times A_1$ & \multirow{2}{*}{96}\\
             & 12 & 60 & 24 & \\
             \hline
             \multirow{2}{*}{3}  & $A_{1,1}$ &  $A_2$ & $A_1\times A_1$ & \multirow{2}{*}{244}\\
             & 8 & 128 & 108 & \\
             \hline
             \multirow{2}{*}{4} &  \multicolumn{3}{c}{$A_1$} & \multirow{2}{*}{270}\\
             & \multicolumn{3}{c}{270} & \\
            \hline
             \multirow{2}{*}{5} & 
             \multicolumn{3}{c}{$A_0$ (Clusters)} & \multirow{2}{*}{108}\\
             &  \multicolumn{3}{c}{108} & \\
        \end{tabular}
        \caption{Type and number of subalgebras in the $\Affine{D}_4$ cluster algebra up to the action of $\gamma$.}
        \label{fig:D4AffCounts}
    \end{figure}
    \end{example}
    
    \subsection{Counting facets in the affine associahedra}
    Let $Q$ be any quiver of affine mutation type of rank $n$, and let $\mathcal{A}$ be the cluster algebra associated to this quiver. Let $\vec{n} = (n_i), \vec{w} = (w_i)$ be the vectors defining a $\T$ quiver in the mutation class of $Q$ and let $\chi(\A)= \sum(w_i(n_i^{-1}-1))+2 $. 
    
    The affine associahedron associated to $\A$ will have a $k-$cell for each rank $k$ subalgebra of $\A$. Since a rank $k$ subalgebra is obtained by freezing $n-k$ nodes in a quiver underlying a seed of $\A$, we will count codimension 1 facets by counting cluster variables up to the action of $\gamma$.

    \begin{definition}
    Let $R \in \Mut{Q}$. We call a node $k$ of $R$ \emph{finite} if the quiver obtained by freezing $k$ is of finite type. We call $k$ \emph{affine} if the quiver obtained by freezing $k$ is of affine type. We call the cluster variables associated with affine or finite nodes affine or finite respectively.
    \end{definition}
    
    \begin{lemma}
    Every node of $R$ is either finite or affine. 
    \end{lemma}
    \begin{proof}
     If $Q$ is of type $A$ or $D$ then this follows by seeing that every possible arc in the associated marked surface cuts the surface into regions which are surfaces of type $A$ or $D$. In the $E$ case, this may be checked by brute force. Note that this also shows that all subalgebras of affine or finite algebras are affine or finite. 
     
     The non simply laced cases then follow by folding as freezing any node in the folded quiver corresponds to freezing the corresponding group in the unfolded quiver. This results in a subalgebra of the unfolded quiver, which we just show is finite or affine. 
    \end{proof}
    
    \begin{remark}
    The arcs which correspond to cluster variables on finite nodes in the $A$ and $D$ cases are exactly the arc which have non trivial intersection number with the arc that generates the Dehn twist $\delta$ i.e. the crossing arcs (\Cref{def:SurfaceArcClasses}). 
    \end{remark}
    
    \begin{lemma}\label{thm:AffineVariableLocation}
    The cluster variable associated with every finite node appears on a quiver which is an orientation of the associated affine Dynkin diagram. Every affine cluster variable appears on a tail node of a $T_{\vec{n},\vec{w}}$ quiver in the mutation class of $Q$. 
    \end{lemma}
    \begin{proof}
    The first statement follows in the $A$ and $D$  cases by noticing that each arc which intersects $\delta$ can be found in a triangulation which is an orientation of the Dynkin diagram. In the $E$ case, this follows by a slightly more sophisticated brute force calculation similar to the calculation of the previous lemma. 
     
     The second statement is proved in a similar way to the first. We notice that each affine cluster variable is associated with a boundary arc, and each boundary arc can be found in a triangulation corresponding to a $T_{\vec{n},\vec{w}}$ quiver. Again, in the $E$ case this is checked by brute force.

     The folded algebras require similar analysis. In type $\Affine{C}_n$ we observe the automorphism which folds type $\Affine{A}_{n,n}$ is the automorphism of the annulus which swaps the two boundary components and fixes $\delta$. Label the boundary points $o_1,\cdots o_n, i_1,\cdots,i_n$.  The orbits of this action come in four classes, pairs of boundary arcs ($(o_{k_1},o_{k_2})$ and $(i_{k_1},i_{k_2})$), pairs of self loops, pairs of crossing arcs between distinct marked points with the same winding number ($(o_{k_1},i_{k_2})$ and $(o_{k_2},i_{k_1})$ and a single crossing arc connecting marked points with the same index $(o_k,i_k)$. The first two cases are affine and appear simultaneously in the construction of a $T_{n,n}$ quiver given in \Cref{sec:AffineProofs}. Similarly the last two cases are finite and appear together in the Dynkin seed with one source and one sink constructed in \Cref{sec:AffineProofs}. A similar analysis of tagged arcs on a twice punctured disk invariant under changing the tagging at one puncture proves the statement for $\Affine{B}_n$ folded from $\Affine{D}_{n+1}$. The remaining cases $\Affine{F}_4$ and $\Affine{G}_2$ follow from explicit calculation as in type $E$.
    \end{proof}
    
    \begin{remark}
    Freezing an affine node produces an affine subalgebra of $\A$. Since these nodes always appear on the tail of a $T_{\vec{n},\vec{w}}$ quiver, we can see that $\gamma$ is also an element of the cluster modular group of every affine subalgebra of $\A$. Thus, the action of $\gamma$ on the cluster complex of $\A$ restricts to the action of $\gamma$ on the cluster complex of any affine subalgebra of $\A$. Thus it makes sense to consider the affine associahedra of subalgebras to be facets of the affine associahedra of $\A$. 
    \end{remark}

    \begin{definition}
    We write $C^k(\A)$ resp. $C_k(\A)$ for the sets of codimension resp. dimension $k$ facets of the affine  associahedron of $\A$.
    \end{definition}
    
    The size of $C^1(\A)$ is equal to the number of distinct cluster variables in $\A$ up to the action of $\gamma$.
    
    \begin{theorem} The number of distinct cluster variables in an affine cluster algebra up to the action of $\groupgenby{\gamma}$ is given by 
    \begin{equation}
        |C^1(\A)| = \sum_i(n_i-1)n_i + \frac{n}{\chi(\A)}
    \end{equation}
    \end{theorem}
    \begin{proof}
    
    We simply need to count the number of finite and affine cluster variables up to the action of $\groupgenby{\gamma}$. The action of $\gamma$ is trivial on the affine cluster variables, so we simply need to count them. By \Cref{thm:AffineVariableLocation}, each affine cluster variable appears on the tail of a $T_{\vec{n},\vec{w}}$. On tail $i$ there are $n_i-1$ affine cluster variables, and each application of $\tau_i$ gives an entirely new collection of affine cluster variables; This may be seen by examining the $A_{p,1}$ case. Thus, in total there are $\sum_i(n_i-1)n_i$ affine cluster variables. 
    
    To count the number of finite cluster variables up to the action of $\gamma$, we again use \Cref{thm:AffineVariableLocation}, so that we only need to count the number of cluster variables appearing on source-sink oriented Dynkin diagrams. The source-sink mutation path takes each collection of cluster variables to an entirely new collection \cite{Schiffler:cluster_automorphisms}. By \Cref{thm:tnw_reddening} we know that the source-sink mutation path is equivalent in the cluster modular group to $r$. We can calculate that the order of $r$ in $\Gamma_Q / \groupgenby{\gamma}$ is $\chi^{-1}$ using the presentation of \Cref{rem:Tgroup}. Thus since there are $n$ finite cluster variables on each source-sink oriented Dynkin quiver, there must be $\frac{n}{\chi}$ up to the action of $\gamma$. There is small subtlety when $\chi^{-1} = \frac{a}{b}$ is not an integer (types $\Affine{A}_{p.q}$ and $\Affine{C}_p$). In this case $a$ is the order of $r$ and $r^a = \gamma^b$ in $\Gamma_Q$. Thus the $n a$ cluster variables along the source sink mutation path over count the number of cluster variables in a single orbit $b$ times. Thus the number of finite cluster variables is $n \frac{a}{b} = n \chi^{-1}$ as needed. 
    \end{proof}
    
    \begin{remark}
    The number of distinct cluster variables up to the action of $\groupgenby{\gamma^\ell}$ is given by 
    \begin{equation}
        \sum_i(n_i-1)n_i + \frac{\ell n}{\chi(\A)}.
    \end{equation}
    This is because higher powers of $\gamma$ identify fewer finite cluster variables.
    \end{remark}

    \begin{lemma}\label{thm:AffineCountingFacets}
    \begin{equation}
        |C_k(\A)| = \frac{1}{n-k}\sum_{\mathcal{B} \in C^{1}(\A)}C_k(\mathcal{B})
    \end{equation}
    \end{lemma}
    This follows since each dimension $k$ facet appears $n-k$ times as a dimension $k$ facet of distinct corank 1 subalgebras. This lemma allows us to compute the number of facets of any particular affine associahedron inductively. 
    
    \begin{conjecture}
    Each affine associahedron is topologically a sphere.
    \end{conjecture}
    This conjecture is known to be true in the type-$A$ cases, see \cite{Penner:arc_complexes}. One may also check it case-by-case for the exceptional types. 
    
    We will now compute a uniform closed form expression for the number of vertices (number of clusters) of an affine associahedron. 
    
    \begin{theorem} The number of distinct clusters in an affine cluster algebra up to the action of  $\groupgenby{\gamma}$ is given by 
    \begin{equation}
        |C_0(\A)| = \frac{2}{\chi(\A)}\prod_i\binom{2n_i-1}{n_i}
    \end{equation}

    \end{theorem}

    We will prove this theorem in the simply laced cases. Each of the exceptional cases can be computed inductively by \Cref{thm:AffineCountingFacets}. The non-simply laced cases have similar proofs to the one for $\Affine{D}_n$ shown here. 

    First we review some facts about the Catalan numbers, $C_n = \frac{1}{n+1}\binom{2n}{n}$, and the middle binomial coefficients $B_i = \binom{2i}{i}$ that will be useful in proving this counting formula. Let $C(x) = \sum\limits_{i=0}^{\infty}C_i x^i$ and $B(x) = \sum\limits_{i=0}^{\infty} B_i x^i$ be the generating functions for the Catalan numbers and middle binomial coefficients respectively. Then we have the following identities that hold  wherever the sums converge.
    \begin{align}
        C(x) = \frac{1-\sqrt{1-4x}}{2x} ,& \hspace{4mm} 1-2xC(x) = \sqrt{1-4x} \label{eqn:CatalanFunction}\\
        (1-2xC(x))^{-1} &= (1-4x)^{-1/2} = B(x) = \sum_{i=0}^\infty \binom{2i}{i}x^i \label{eqn:BinomialFunction} \\
        2(1-4x)^{-3/2} &= \sum_{i=1}^\infty i\binom{2i}{i}x^{i-1} \label{eqn:DerivativeOfBinomial}
    \end{align}
     It will also be helpful to define the truncated generating function $C_{\floor{k}}(x) = \sum\limits_{i=0}^{k-1}C_i x^i$.\\
    We are now ready to consider the $A_{p,q}$ case. Let $A_{p,q}$ be the number of clusters in an $A_{p,q}$ cluster algebra up to $\gamma$.  In this case the formula for the number of distinct clusters simplifies to:
    \begin{equation}
        A_{p,q} = \frac{pq}{2(p+q)}\binom{2p}{p}\binom{2q}{q}
    \end{equation}
    \begin{proof}[Proof of \Cref{thm:AffineCountingFacets} for $\Affine{A}_{n}$]
    In \Cref{thm:ApqRecurrence} we establish the recurrence $A_{p,q} = 2\sum\limits_{i=0}^{p-1}C_i A_{p-i,q} + qC_{p+q}$.  Then for each $q$, let $A_q(x) = \sum\limits_{i=1}^\infty A_{i,q}x^{i+q}$. The recurrence corresponds to the following equation of generating functions:
    \begin{equation}
        A_q(x) = 2xC(x)A_q(x) + qx\big(C(x) - C_{\floor{q}}(x)\big)
    \end{equation}
    Solving for $A_q(x)$ gives 
    \begin{equation}\label{eq:aqx}
         A_q(x) =  \frac{q\vphantom{C_{\floor{q}}\big)}}{2}\cdot\frac{2x\big(C(x) - C_{\floor{q}}(x)\big)}{1-2xC(x)}
    \end{equation}
    Using \Cref{thm:ApqCountingFormula} we compute the powers series expansion of the right hand side is:
     \begin{equation}
        \frac{q\vphantom{C_{\floor{q}}\big)}}{2}\cdot\frac{2x(C(x)-C_{\floor{q}}(x))}{1-2xC(x)} = \frac{q}{2}\sum_{i=1}^{\infty} \frac{i}{(i+q)}\binom{2i}{i}\binom{2q}{q}x^{i+q}.
    \end{equation}
    As $A_{p,q}$ is the coefficient of $x^{p+q}$ this means that $A_{p,q} = \frac{pq}{2(p+q)}\binom{2p}{q}\binom{2q}{q}$ as needed.
    \end{proof}

\begin{lemma}\label{thm:ApqCountingFormula}
     \begin{equation}
        \frac{2x(C(x)-C_{\floor{q}}(x))}{1-2xC(x)} = \sum_{i=1}^{\infty} \frac{i}{i+q}\binom{2i}{i}\binom{2q}{q}x^{i+q}.
    \end{equation}
\end{lemma}
\begin{proof}
    In order to determine the coefficients of this power series we will examine the power series associated with the following integral.
    \begin{equation}\label{eq:i}
        I_q(x) = \int_0^x 2 z^q (1-4z)^{-3/2} dz 
    \end{equation}
    We will evaluate $I_q$ in two different ways. First, notice the integrand has a power series expansion given by \Cref{eqn:DerivativeOfBinomial}. By integrating this power series we find that: 
    \begin{equation}\label{eq:i1}
        I_q(x) = \sum_{i=1}^{\infty} \frac{i}{i+q}\binom{2i}{i}x^{i+q} 
    \end{equation}
    Second, we use the standard calculus method of substitution to find that 
    \begin{equation}\label{eq:i2}
        I_q(x) = R(x)(1-4x)^{-1/2} - R(0)  
    \end{equation}
    where $R(x)$ is some polynomial of degree $q$. \\
    We claim $R(x) = \binom{2q}{q}^{-1} (1-2xC_{\floor{q}}(x))$. We verify this claim in the following two steps. 
    
    First, by comparing the two different power series representations of $I_q$ obtained in \Cref{eq:i1} and \Cref{eq:i2}, we may see that $R(x)(1-4x)^{-1/2}$ must have coefficient zero on $x^i$ in its power series for $1 \leq i \leq q$. The only polynomials of degree $q$ which we can multiply $(1-4x)^{-1/2}$ and achieve this are constant multiples of $(1-2xC(x))_{\floor{q+1}}$ since $1-2xC(x)$ is the inverse of $(1-4x)^{-1/2}$. Thus we have $R(x) = R(0)(1-2xC_{\floor{q}}(x))$. 
    
    Now we may evaluate $R(0)$ by comparing the $x^{q+1}$ terms of each of the power series representations. From \Cref{eq:i1}, we have the $q+1$ term is $\frac{2}{1+q}x^{q+1}$. From \Cref{eq:i2}, we find that the $q+1$ term is 
    \begin{equation}
        R(0)\left(\binom{2q}{q}-2\sum_{i=1}^{q}C_{i-1}\binom{2(q+1-i)}{q+1-i}\right)x^{q+1} = R(0)(2C_q) x^{q+1}
    \end{equation}
    since $1-2xC(x)$ is the inverse power series of $\sum_{i=0}^\infty \binom{2i}{i}x^i$. Thus we find that 
    \begin{equation}
        R(0) = \binom{2q}{q}^{-1}.
    \end{equation}
    Finally, multiplying through by $\binom{2q}{q}$ , we obtain the equation 
    \begin{equation}
        \sum_{i=1}^{\infty} \frac{i}{i+q}\binom{2i}{i}\binom{2q}{q}x^{i+q}  = \frac{1-2xC_{\floor{q}}(x)}{\sqrt{1-4x}} - 1  = \frac{2x(C(x)-C_{\floor{q}}(x))}{1-2xC(x)}.
    \end{equation}
    \end{proof}

    Next we will show a similar proof for the $\Affine{D}_n$ case. We will simply write $\Affine{D}_n$ for the number of tagged triangulations of a twice punctured disk with $n-2$ marked points on the boundary. As before we build on the combinatorics in the finite case. Recall that $D_n = \frac{3n-2}{n}\binom{2(n-1)}{n-1}$ is the number of tagged triangulations of a once punctured disk with $n$ marked points. For notational convenience let $D_0 = 1$. This lets us define the generating function $D(x) = \sum\limits_{i=0}^{\infty}D_i x^i$\\
    In this case the statement of \Cref{thm:AffineCountingFacets} becomes:
    \begin{equation}
        \Affine{D}_n = 9(n-2)\binom{2(n-2)}{(n-2)}, n \geq 3
    \end{equation}
    
    \begin{proof}[Proof of \Cref{thm:AffineCountingFacets} for $\Affine{D}_n$]
    In \Cref{thm:DnRecurrence} we show that \[\Affine{D}_{n+1} = 2\sum\limits_{i=0}^{n-3} C_i \Affine{D}_{n-i} + 2 \sum\limits_{j=0}^{n}D_jD_{n-j} \]
    Let $\Affine{D}(x) = \sum_{i=3}^\infty \Affine{D}_i x^i$ be the generating function for $\Affine{D}_i$. The recurrence above becomes:
     \begin{equation}
        \Affine{D}(x) = 2xC(x)\Affine{D}(x) + 2x(D(x)^2 - 1 - 2x) 
    \end{equation}
    
    Again solving for $\Affine{D}(x)$ we find
    \begin{equation}
        \Affine{D}(x) = \frac{2x(D(x)^2 - 1 - 2x) }{1-2xC(x)}
    \end{equation}
    We can see easily that $D(x) = 3xB(x) - 2xC(x) +1 = 3xB(x) + B^{-1}(x)$. Thus the previous equation becomes 
    \begin{equation}
        \Affine{D}(x) = \frac{2x(9x^2B^2(x) + B^{-2}(x) + 6x - 1 - 2x) }{1-2xC(x)}
    \end{equation}
    and using the fact that $B^2(x) = \frac{1}{1-4x}$ (\Cref{eqn:BinomialFunction}) we have 
    \begin{equation}
        \Affine{D}(x) = 18x^3(1-4x)^{-3/2} = \sum_{i=3} 9(i-2)\binom{2(i-2)}{(i-2)}x^{i}
    \end{equation}
    as desired.

    \end{proof}

\section{Doubly Extended Cluster Algebras}\label{sec:DoubleExtended}

In this section we consider $Q=\T$ to be of doubly-extended type, i.e we have $\chi=0$. Let $\mathcal{A}$ be the cluster algebra associated to $Q$. There are only finitely many possibilities for $\vec{n},\vec{w}$ with $\chi=0$ listed in \Cref{fig:doubleExtendedOptions}. Other than $\db{A}{1}$, which has to be treated separately, only $\db{D}{4}$ is associated to a surface (the four punctured sphere).

We will not consider the $A$ or $BC$ cases for the first part of this section, and treat them separately later. Since our $\T$ quivers always have weight 1 middle nodes, we will only construct quivers for the types on the left hand side of the table in \Cref{fig:doubleExtendedOptions}. The types on the right hand side are dual to types with $\T$ quivers.

\begin{figure}
    \centering
\renewcommand{\arraystretch}{1.5}
    \begin{tabular}{c c c c c c}
     Type &  $\vec{n}$ & $\vec{w}$ & $|N|$ & $ord(r)$ & dual \\
     \hline\hline
      $A_1^{(1,1)}$   &  N/A & N/A  & 1 & 1 & self\\
      $D_4^{(1,1)}$   & $(2,2,2,2)$ & $(1,1,1,1)$ & 196 & 2 & self \\
      $E_6^{(1,1)}$ & $(3,3,3)$ & $(1,1,1)$ & 54 & 3 & self \\
      $E_7^{(1,1)}$ & $(4,4,2)$ & $(1,1,1)$ & 16 & 4 & self\\
      $E_8^{(1,1)}$ & $(6,3,2)$ & $(1,1,1)$ & 6 & 6 & self\\
       \hline\hline
      $BC_1^{(4,1)}$ & $(2)$ & $(4)$ & 1 & 1 & $BC_1^{(4,4)}$ \\
      $B_2^{(2,1)}$ & $(2,2)$ & $(2,2)$ & 4 & 2 & self \\
      $BC_2^{(4,2)}$ & $(2,2)$ & (BC-Type) & 2 & 2 & self \\
      $B_3^{(1,1)}$ & $(2,2,2)$ & $(1,1,2)$ & 24 & 2 & $C_3^{(2,2)}$ \\
      $F_4^{(1,1)}$ & $(3,3)$  & $(1,2)$ & 3 & 3 & $F_4^{(2,2)}$ \\
      $F_4^{(2,1)}$ & $(4,2)$ & $(2,1)$ & 4 & 4 & self\\
      $G_2^{(1,1)}$ & $(2,2)$ & $(1,3)$ & 2 & 2 & $G_2^{(3,3)}$\\
      $G_2^{(3,1)}$ & $(3)$ & $(3)$ & 3 & 3 & self\\
    \end{tabular}

    \caption{All possible values of $\T$ that result in doubly extended cluster algebras.}
    \label{fig:doubleExtendedOptions}
\end{figure}

\subsection{Structure of the cluster modular group}\label{sec:DoubleExtendedModularGroup}
Let $\Gamma$ be the cluster modular group of $\mathcal{A}$. Let $Q'=T'_{\vec{n},\vec{w}}$ be the underlying affine-type quiver of the doubly extended type quiver, $Q$. Let $s$ be the source-sink mutation path on $Q'$,  $\chi'= \chi(Q')$ and arrange that $n_1=\max(n_i)$ and that $w_1$ is minimal if there are multiple tails of the same maximal length. It is easy to verify in each case that $s^{(\chi'n_1)^{-1}}$ returns to an isomorphic quiver. Thus $\delta = (s^{(\chi'n_1)^{-1}},\text{id}) \in \Gamma$.

\begin{theorem}
$\Gamma$ is generated by $\Gamma_{\tau}, \Aut{Q}$ and $\delta$. 
\end{theorem}
\begin{proof}
 This is checked in a case by case way for each of the simply-laced doubly extended cluster modular groups. Most of these groups have been computed elsewhere. Fraser has presentations for the $E_7$ and $E_8$ cases using the Grassmannian cluster algebra structures of $\gr{4}{8}$ and $\gr{3}{9}$ respectively (\cite{fraser_braid_2020}). We note that our notion of the cluster modular group does not include arrow reversing quiver automorphisms, so our groups are the orientation preserving subgroups of his. 
 
 Its a simple matter to check that each of Fraser's generators can be written with the above elements. For example Fraser's presentation of $\Gamma_{\db{E}{8}}$ is 
 \begin{equation}
     \groupgenby{\rho, P, t, : \rho^3 = P^2 = t^{2}, \rho^9 = 1, t\rho = \rho t, tP = Pt }. 
 \end{equation} In our notation 
 \begin{equation}
     \rho = r\delta\tau_1, \quad P = r^2\delta\tau_1\delta, \quad t = r
 \end{equation}
 where $r$ is the reddening element.
 
Fraser's presentation of the cluster modular group for $\db{E}{7}$ is 
\begin{align}
     \groupgenby{\sigma_1,\sigma_2,\sigma_3,t| & \sigma_1\sigma_2\sigma_1 = \sigma_2\sigma_1\sigma_2, \quad  
     \sigma_2 \sigma_3\sigma_2 = \sigma_3\sigma_2\sigma_3, \quad
     \sigma_1\sigma_3 = \sigma_3\sigma_1, \\
      & \sigma_1\sigma_2\sigma_3^2\sigma_2\sigma_1 = (\sigma_3\sigma_2\sigma_1)^8 = 1, \quad
      (\sigma_3\sigma_2\sigma_1)^4 = t^2, \quad t\sigma_i = \sigma_i t}.
 \end{align}
 In our presentation we have 
 \begin{equation}
     \sigma_1 =\tau_1 \quad \sigma_2 =r\delta \quad \sigma_3 =\tau_2 \quad t = r.
 \end{equation}
 
 The $E_6$ case is new and we have computed it using \Cref{rem:FiniteMutationTypeAlgorithm} and \Cref{them:DoubleRelations} below. It has the following presentation:
 \begin{align*}
     \groupgenby{\tau_1,\tau_2,\tau_3,\sigma_{23}, \omega, \delta |& \tau_i \tau_j = \tau_j \tau_i,~ \tau_i^3 = \tau_j^3=\gamma,\\
     & \sigma_{23}^2 = 1, ~\omega^3 = 1,~ \sigma_{23}\omega=\omega^{-1}\sigma_{23},~ \tau_2=\omega \tau_1 \omega^{-1},~ \tau_3=\omega \tau_2 \omega^{-1},\\
     &\tau_1 \delta \tau_1 = \delta \tau_1 \delta,~ (\tau_1 \delta)^3 = r^2\sigma_{23}}
 \end{align*}
 where $r = \tau_1\tau_2\tau_3\gamma^{-1}$ is the reddening element. The automorphism group of $T_{3,3,3}$ is generated by $\sigma_{23}$ that swaps tails 2 and 3 and $\omega$ which rotates all three tails.\\
 The non-simply laced cases follow from \Cref{rem:FoldingT_nw}. 
\end{proof}

To best describe the relations between $\delta$ and the other generators of $\Gamma$, it will be helpful to recall some basic properties of the rank 2 Artin-Tits braid groups of type $A_2, B_2$ and $G_2$. The groups $\mathcal{B}(X_2)$ have the presentation
\begin{align}
    \mathcal{B}(A_2) &= \{a,b | aba=bab\} \\
    \mathcal{B}(B_2) &= \{a,b | abab=baba\} \\
    \mathcal{B}(G_2) &= \{a,b | ababab=bababa\} 
\end{align}

\begin{remark}\label{rem:braid_subgroups}
$\mathcal{B}(A_2)$  is generally known as the braid group on 3 strands, $\mathcal{B}_3$.
The center, $\mathcal Z$ of these groups is an infinite cyclic group generated by $z=ababab$, $z=abab$ and $z=ababab$ for $\mathcal{B}(A_2)$, $\mathcal{B}(B_2)$, $\mathcal{B}(G_2)$ respectively. We have an isomorphism 
\begin{equation}
    \mathcal{B}(A_2)/\mathcal{Z} \simeq \PSL{2}{Z}
\end{equation}
If we let $X_2(k) = A_2 , B_2, \text{ or } G_2$ if $k = 1, 2, \text{ or }  3$ respectively, then the subgroup of $\mathcal{B}(A_2)$ generated by $\{a,b^k\}$ is isomorphic to $\mathcal{B}(X_2(k))$
\end{remark}

\begin{claim}\label{claim:braids}
For each $i$ we have a map $\psi_i: \mathcal{B}(X_2(n_1w_i/n_i)) \rightarrow \Gamma$ given by $\{a,b\} \rightarrow \{\tau_i, r\delta\}$. Moreover, the image of the element $z$ is shown in \Cref{fig:DoubleCenters}.
\end{claim}
\begin{proof}
In each case it suffices to check the images satisfy the braid relations.
\end{proof}

\begin{figure}[hb]
    \centering
    \begin{tabular}{c c c c c}
        Type \  & $i=1$& $i=2$& $i=3$ & $i=4$\\ \hline \hline 
        $\db{D}{4}$ & id & id & id & id\\
         $\db{E}{6}$ & $r^2\sigma_{23}$ & $r^2\sigma_{13}$ & $r^2\sigma_{12}$ & - \\
         $\db{E}{7}$ & $r^2$ & $r^2$ & $r \sigma_{12}$ & - \\
         $\db{E}{8}$ & $r^2$ & $r^4$ & $r$ & - \\ \hline 
         $\dbf{B}{2}{2}{1}$ & $r$ & $r$ & - & - \\
         $\db{B}{3}$ & id & id & $r \sigma_{12}$& -  \\
         $\db{F}{4}$ & $r^2$ & $ r$ & - & - \\
         $\dbf{F}{4}{2}{1}$ & $r$ & $r$ & - & - \\
         $\db{G}{2}$ & id & $r$ & - & - \\
         $\dbf{G}{2}{3}{1}$ & $r$ & - & -& -  
    \end{tabular}
    \caption{Images of the central element $\psi_i(z) = c$ for the group homorphisms of \Cref{claim:braids}.}
    \label{fig:DoubleCenters}
\end{figure}

Let $N = \Gamma_\tau^\circ \rtimes \Aut{Q} $ where $\Gamma_\tau^\circ$ was defined in \Cref{rem:Tgroup} by the following exact sequence: 
\begin{equation}
    1 \rightarrow \Gamma_\tau^\circ \rightarrow \Gamma_{\tau} \rightarrow \mathbb{Z} \rightarrow 1 
\end{equation}

\begin{theorem}\label{them:DoubleRelations}
The following sequence is exact:
\begin{equation}
    1 \rightarrow N \rightarrow \Gamma \rightarrow \mathcal{B}(X_2(w_1))/\mathcal{Z} \rightarrow 1 
\end{equation}
\end{theorem}
\begin{proof}
First, it is necessary to check that $N$ is a normal subgroup, which we may do for each of the four simply laced cases and fold to get the non simply laced cases. To see that the quotient is as described, we only need to show that the induced map 
\begin{equation}
    \mathcal{B}(X_2(w_1))/\mathcal{Z} \rightarrow \Gamma/N 
\end{equation}
from \Cref{claim:braids} is an isomorphism. Since the $\mathbb{Z}$ component of $\tau_1$ is the $\gcd$ of all possible $\mathbb{Z}$ components in $\Gamma_\tau $ and $\Aut{Q} \subset N$,  $\tau_1$ generates $\Gamma_\tau \rtimes \Aut{Q} / N \simeq \mathbb{Z}$. Thus $\tau_1 $ and $ \delta$ generate $\Gamma/N $. Therefore, we only need to check that the only relations come from those in the braid group modulo its center. In the simply laced cases, this will follow by checking that the only relations between $\delta$ and $\tau_1$ is $(\delta \tau_1)^3 = \text{id}$.

We first check the $\db{D}{4}$ case, since this algebra is associated with a 4-punctured sphere. Here $\delta$ and $\tau_1$ correspond half twists sharing a single puncture and thus to elements of $\PSL{2}{Z}$ as a quotient group of the mapping class group. 

Then, we can check that the maps of cluster modular groups induced by folding operations of \Cref{fig:double_extended_family} preserve this subgroup faithfully. By \Cref{thm:FoldingClusterModularGroups} we know the cluster modular group of the folded algebra is a subquotient of the unfolded algebra. It then suffices to verify that $\delta$ and $\tau_1$ appear in the image and no extra relations are added by the quotient.

Let $\mathcal A \rightarrow \mathcal B$ be any folding of doubly extended type cluster algebras. Let $n = n_1(\mathcal{A}), w= w_1(\mathcal{A})$ and $m= n_1(\mathcal{B}), z= w_1( \mathcal{B})$ be the length and weights of the first tail of $\T$ quivers representing seeds of these algebras. Let $\tau , \eta$ be the twist elements of the first tails and $\delta ,\epsilon$ be the extra generators in the modular groups of $\mathcal{A}$ and $\mathcal{B}$ respectively.

The double arrows corresponding to Langlands dual obviously preserve the subgroup. The solid edges, corresponding to folding the $\T$ quivers directly, only quotient by elements in $N$, which are zero in $\Gamma/N$. This follows since we fold by an automorphism of the $\T$ quiver which are contained in $N$.

Furthermore, we clearly have that $\delta = \epsilon$ in the standard folding case. We see that if $w=z$ we have that $\tau$ directly descends to the cluster modular group of the folded algebra. In this case we have an isomorphism $\ClusterModularGroup{\mathcal{A}}/N_{\mathcal{A}}$ with $\ClusterModularGroup{\mathcal{B}}/N_{\mathcal{B}}$.

Otherwise, $z$ tails of length $n$ are folded, and we have that $\eta$ is equivalent to successive twists around each of these unfolded tails. In the quotient $\Gamma_\mathcal{A}/N_\mathcal{A}$ we have that successive twists around $z$ tails of the same length is equal to $\tau^z$. Thus \Cref{rem:braid_subgroups} extends the theorem across algebras related by a standard folding.  

To connect the graph of double extended cluster algebras in \Cref{fig:double_extended_family} we only need to check the  nonstandard unfolding (dashed arrows) $\dbf{C}{3}{2}{2} \dashleftarrow \db{E}{7}$, $\dbf{G}{2}{3}{3} \dashleftarrow \db{E}{8}$ and the nonstandard folding $\db{E}{8} \dashrightarrow \dbf{F}{4}{2}{2}$. See \Cref{fig:nonstandardFolding} to see the folds in each case. One checks for each of these cases that these automorphisms are also contained in $N$. We will dualize each of these folded algebras so that we may compare their cluster modular groups using the presentation coming from $\T$ quivers.

We start with the unfolding $\db{G}{2} \Leftrightarrow \dbf{G}{2}{3}{3} \dashleftarrow \db{E}{8}$. We have a path of valid folds and unfolds from $\db{D}{4}$ to $\dbf{G}{2}{1}{1}$ so we know their are no extra relations in $\dbf{G}{2}{1}{1}$. Thus it suffices to write $\delta^{E}$ and $\tau_1^{E}$ in terms of $\delta^{G}$ and $\tau_1^{G}$. Let 
\begin{equation}
   P = (2_2 1_4 1_5 1_3 1_2 1_4 2_3 N_1 N_\infty) 
\end{equation}
be a path of mutations from $T_{6,3,2}$ to the triangular quiver shown in \Cref{fig:nonstandardFoldingG2}. Then 
\begin{align}
    \delta^E &= P \delta^G \tau^G (\delta^G)^{-2}\tau^G P^{-1} \\
    \tau_1^E &= P \tau^G P^{-1}.
\end{align}
By replacing $P$ with $P'=P(\tau^G)^2$ and using braid relations, we can see that 
\begin{align}
    \delta^E &= P' \delta^G P'^{-1} \\
    \tau_1^E &= P' \tau^G P'^{-1}.
\end{align}
This shows the theorem holds for $\db{E}{8}$.

The next case to consider is the folding $\db{E}{8} \dashrightarrow \dbf{F}{4}{2}{2} \Leftrightarrow \db{F}{4}$. Once again if we can write $\tau_1^F$ and $\delta^F$ in terms of the generators $\tau_1^E$ and $\delta^E$ any extra relations in $\dbf{F}{4}{1}{1}$ would descend to relations in $\db{E}{8}$ which we just showed didn't have extra relations. A simple computation shows that 
\begin{equation}
    P = (1_6 3_2 1_5 1_4 1_6 1_3 1_2 N_1)
\end{equation}
is a path from $T_{6,3,2}$ to the quiver shown in \Cref{fig:nonstandardFoldingF4}. Then 
\begin{align}
    \tau_1^F &= P^{-1} \tau_1^E P \\
    \delta^F &= P^{-1} (\tau_1^E)^{-1} (\delta^E)^{-3}(\tau_1^E)^{-1}(\delta^{E})^{-1}(\tau_1^E)P = P^{-1}.
\end{align} 
Again using braid relations we can see in the quotient that $\delta^F = P^{-1}\delta^E P$.

The final case is the unfolding $\db{B}{3} \Leftrightarrow \dbf{C}{3}{2}{2} \dashleftarrow \db{E}{7}$. Here we have a path of valid folds and unfolds $\db{D}{4} \rightarrow \db{B}{3}$. So all that remains is to write the generators for $\db{E}{7}$, $\tau^E$ and $\delta^E$ in terms of the generators for $\dbf{B}{3}{1}{1}$, $\delta^B$ and $\tau^B$. Let \begin{equation}
    P = (2_4 1_4 2_3 2_2 1_3 1_2 N_1).
\end{equation} Then 
\begin{align}
    \delta^E &= P \delta^{B} P^{-1} \\
    \tau_1^E &= P\tau_1^B P^{-1} .
\end{align}
             
\begin{figure}
    \centering
    %E8 -> G2
    \begin{subfigure}{\textwidth}
        \centering
        \begin{tikzpicture}
        \node[base,fill=yellow] (7) [] {};%{$1_3$};
         \node[base,fill=blue] (2) [right of = 7, above of = 7] {};%{$N_1$};
        \node[base,fill=red] (4) [right of = 2,below of = 2] {};%{$2_2$};
        \node[base] (3) [right of = 4,above of = 4] {};%{$3_3$};
        \node[base,fill=blue] (5) [right of = 3,below of = 3] {};%{$2_3$};
        \node[base,fill=red] (1) [right of = 5,above of = 5] {};%{$N_\infty$};
        \node[base,fill=yellow] (9) [right of = 1,below of = 1] {};%{$1_5$};
       
        \node[base,fill=red] (6) [right of = 2,above of = 2] {};%{$1_2$};
        \node[base,fill=blue] (8) [right of = 3,above of = 3] {};%{$1_4$};
        \node[base,fill=yellow] (10) [above of = 6,right of = 6] {};%{$1_6$};
        \path[->] (7) edge [] node {} (2); 
        \path[->] (2) edge [] node {} (4);
        \path[->] (4) edge [] node {} (7);
        
        \path[->] (2) edge [] node {} (6);
        \path[->] (6) edge [] node {} (3);
        \path[->] (3) edge [] node {} (2);
        
        \path[->] (3) edge [] node {} (8);
        \path[->] (8) edge [] node {} (1); 
        \path[->] (1) edge [] node {} (3);
        
        \path[->] (4) edge [] node {} (3);
        \path[->] (3) edge [] node {} (5);
        \path[->] (5) edge [] node {} (4);
        
        \path[->] (1) edge [] node {} (9);
        \path[->] (9) edge [] node {} (5);
        \path[->] (5) edge [] node {} (1);
        
        \path[->] (6) edge [] node {} (10);
        \path[->] (10) edge [] node {} (8);
        \path[->] (8) edge [] node {} (6);
    \end{tikzpicture}
        \hspace*{1pc}
    \begin{tikzpicture}
        \node[fat3,fill=yellow] (1) [] {};
        \node[fat3,fill=red] (4) [right of = 1] {};
        \node[fat3,fill=blue] (3) [above of = 4] {};
        \node[base] (2) [right of =4] {};

        \path[->] (1) edge [] node {} (3); 
        \path[->] (3) edge [double] node {} (4);
        \path[->] (4) edge [] node {} (1);
        \path[->] (4) edge [] node {} (2);
        \path[->] (2) edge [] node {} (3);

    \end{tikzpicture}
    \hspace*{1pc}
    \begin{tikzpicture}
        \node[base,fill=yellow] (1) [] {};
        \node[base,fill=red] (4) [right of = 1] {};
        \node[base,fill=blue] (3) [above of = 4] {};
        \node[fat3] (2) [right of =4] {};

        \path[->] (1) edge [] node {} (3); 
        \path[->] (3) edge [double] node {} (4);
        \path[->] (4) edge [] node {} (1);
        \path[->] (4) edge [] node {} (2);
        \path[->] (2) edge [] node {} (3);

    \end{tikzpicture}
    \caption{$\db{E}{8}\dashrightarrow \dbf{G}{2}{3}{3}\Leftrightarrow \db{G}{2}$}
    \label{fig:nonstandardFoldingG2}
    \end{subfigure}
    
    %E8 -> F4
    \begin{subfigure}{1\textwidth}
        \centering
        \begin{tikzpicture}[node distance = .6cm]
        \node[base,fill=purple] (0) [] {};
        \node[base,fill=yellow] (1) [right of = 0] {};
        \node[base,fill=blue] (2) [right of = 1,above of = 1] {};
        \node[base,fill=red] (3) [below of = 1,right of = 1] {};
        \node[base] (4) [right of = 2, below of = 2] {};
        \node[base,fill=red] (5) [right of = 4,above of = 4] {};
        \node[base,fill=blue] (6) [below of = 4,right of = 4] {};
        \node[base,fill=yellow] (7) [right of = 6,above of = 6] {};
        \node[base,fill=purple] (8) [right of = 7] {};
        \node[invis] (10) [below of = 4] {};
        \node[base] (9) [below of = 10] {};
        \path[->] (0) edge [] node {} (1); 
        \path[->] (1) edge [] node {} (2);
        \path[->] (2) edge [] node {} (3); 
        \path[->] (3) edge [] node {} (1);
        \path[->] (3) edge [] node {} (4);
        \path[->] (4) edge [] node {} (2);
        \path[->] (2) edge [] node {} (5);
        \path[->] (5) edge [] node {} (4); 
        \path[->] (4) edge [] node {} (6);
        \path[->] (6) edge [] node {} (5);
        \path[->] (6) edge [] node {} (3);
        
        \path[->] (5) edge [] node {} (7);
        \path[->] (7) edge [] node {} (6);
        \path[->] (8) edge [] node {} (7);
        \path[->] (9) edge [] node {} (4);
    \end{tikzpicture}
    \hspace*{0pc}
    \begin{tikzpicture}[node distance = 1cm]
        \node[fat2,fill=purple] (0) [] {};
        \node[fat2,fill=yellow] (1) [right of = 0] {};
        \node[fat2,fill=red] (2) [right of = 1] {};
        \node[fat2,fill=blue] (5) [above of = 2] {};
        \node[base] (3) [right of = 2] {};
        \node[base] (4) [right of = 3] {};
       
        \path[->] (0) edge [] node {} (1); 
        \path[->] (5) edge [double] node {} (2);
        \path[->] (2) edge [] node {} (1);
        \path[->] (1) edge [] node {} (5);
        \path[->] (4) edge [] node {} (3);
        
        \path[->] (2) edge [] node {} (3);
        \path[->] (3) edge [] node {} (5);
    \end{tikzpicture}
    \hspace*{1pc}
    \begin{tikzpicture}[node distance = 1cm]
        \node[base,fill=purple] (0) [] {};
        \node[base,fill=yellow] (1) [right of = 0] {};
        \node[base,fill=red] (2) [right of = 1] {};
        \node[base,fill=blue] (5) [above of = 2] {};
        \node[fat2] (3) [right of = 2] {};
        \node[fat2] (4) [right of = 3] {};
       
        \path[->] (0) edge [] node {} (1); 
        \path[->] (5) edge [double] node {} (2);
        \path[->] (2) edge [] node {} (1);
        \path[->] (1) edge [] node {} (5);
        \path[->] (4) edge [] node {} (3);
        
        \path[->] (2) edge [] node {} (3);
        \path[->] (3) edge [] node {} (5);
    \end{tikzpicture}
    \caption{$\db{E}{8}\dashrightarrow \dbf{F}{4}{2}{2}\Leftrightarrow \db{F}{4}$}
    \label{fig:nonstandardFoldingF4}
    \end{subfigure}

    %E7 -> C3
    \begin{subfigure}{.8\textwidth}
        \centering
        \begin{tikzpicture}
        \node[base,fill=blue] (0) [] {};
        \node[base,fill=red] (1) [right of = 0] {};
        \node[base,fill=yellow] (2) [right of = 1] {};
        \node[base,fill=purple] (3) [below of = 0] {};
        \node[base] (4) [right of = 3] {};
        \node[base,fill=purple] (5) [right of = 4] {};
        \node[base,fill=yellow] (6) [below of = 3] {};
        \node[base,fill=red] (7) [right of = 6] {};
        \node[base,fill=blue] (8) [right of = 7] {};
        \path[->] (0) edge [] node {} (1); 
        \path[->] (3) edge [] node {} (0);
        \path[->] (1) edge [] node {} (3); 
        \path[->] (2) edge [] node {} (1);
        \path[->] (1) edge [] node {} (5);
        \path[->] (5) edge [] node {} (2);

        \path[->] (4) edge [] node {} (1); 
        \path[->] (4) edge [] node {} (7);
        \path[->] (3) edge [] node {} (4);
        \path[->] (5) edge [] node {} (4);
        
        \path[->] (6) edge [] node {} (7);
        \path[->] (7) edge [] node {} (3);
        \path[->] (3) edge [] node {} (6);
        \path[->] (7) edge [] node {} (5);
        \path[->] (5) edge [] node {} (8);
        \path[->] (8) edge [] node {} (7);
    \end{tikzpicture}
        \hspace*{2pc}
    \begin{tikzpicture}
        \node[base] (0) [] {};
        \node[fat2,fill=purple] (1) [right of = 0] {};
        \node[fat2,fill=blue] (2) [right of = 1] {};
        \node[fat2,fill=red] (3) [above of = 1] {};
        \node[fat2,fill=yellow] (4) [right of = 3] {};
       
        \path[->] (0) edge [] node {} (3); 
        \path[->] (3) edge [double] node {} (1);
        \path[->] (1) edge [] node {} (0);
        
        \path[->] (4) edge [] node {} (3);
        \path[->] (1) edge [] node {} (4);
        
        \path[->] (2) edge [] node {} (3);
        \path[->] (1) edge [] node {} (2);
    \end{tikzpicture}
    \hspace*{2pc}
    \begin{tikzpicture}
        \node[fat2] (0) [] {};
        \node[base,fill=purple] (1) [right of = 0] {};
        \node[base,fill=blue] (2) [right of = 1] {};
        \node[base,fill=red] (3) [above of = 1] {};
        \node[base,fill=yellow] (4) [right of = 3] {};
       
        \path[->] (0) edge [] node {} (3); 
        \path[->] (3) edge [double] node {} (1);
        \path[->] (1) edge [] node {} (0);
        
        \path[->] (4) edge [] node {} (3);
        \path[->] (1) edge [] node {} (4);
        
        \path[->] (2) edge [] node {} (3);
        \path[->] (1) edge [] node {} (2);
    \end{tikzpicture}
    \caption{$\db{E}{7}\dashrightarrow \dbf{C}{3}{2}{2} \Leftrightarrow \db{B}{3}$}
    \label{fig:nonstandardFoldingC3}
    \end{subfigure}\\
    
    \caption{Nonstandard folding of doubly extended quivers. The first fold is by the 3-fold rotational symmetry and the last folds are by the 180 degree rotational symmetry.}
    \label{fig:nonstandardFolding}
\end{figure}
\end{proof}

The following commutative diagram summarizes the structure of the cluster modular groups of doubly extended cluster algebras in each case where $w_1 =1$.
\begin{equation}\label{eq:double_cmgs}
    \begin{tikzcd}
1 \arrow[r] & \mathcal{Z} \arrow[d, "\substack{z\\ \downmapsto\\~ c}"] \arrow[r, hook] & \mathcal{B}_3 \arrow[d] \arrow[r, two heads] & \PSL{2}{Z} \arrow[r] \arrow[d, Rightarrow, no head] & 1 \\
1 \arrow[r] & N \arrow[r, hook]                                          & \Gamma \arrow[r, two heads]                  & \PSL{2}{Z} \arrow[r]                                & 1
\end{tikzcd}                                                                                   
\end{equation}

\begin{corollary}
Cluster modular groups are generated by ``cluster Dehn twists'' of  \cite{Ishibashi:nielsen-thurston}. 
\end{corollary}

\begin{proof}

Consider the twist generators $\tau_i \in \Gamma_\tau$. From \Cref{thm:TwistPowers}, we saw that $\tau_i^{n_i} = \gamma^{w_i}$. In the surface cases $\gamma$ is a Dehn twist and in the exceptional cases is a cluster Dehn twist. 

Furthermore, the element $\delta^{n_1}=s^{1/\chi'}$ can be seen to be conjugate to $\gamma$ in the following way. First by freezing nodes $1_{n_1}$ and $N_\infty$ we are left with the corresponding finite type quiver. Let $g$ be the sources sinks mutation pattern on this finite type quiver and let $h$ be the order of this element.  Then we have $\alpha = \{g^{h/2},(1_{n_1} N_\infty)\} \in \Gamma$ and $\alpha \gamma \alpha^{-1} = \delta^{n_1}$. Thus $\delta$ is a cluster Dehn twist. As in \Cref{sec:AffineProofs} when $h$ is odd we interpret $g^{h/2}$ as $\floor{h/2}$ applications of $g$ followed by mutation only at the sources.  

Finally, we see that the elements of $\Aut{Q}$ each are periodic elements akin to periodic mapping class group elements. It is possible to generate these elements in each case using cluster Dehn twists. The images of central element, $c$, for various maps from braid groups is always generated by the cluster Dehn twists $\tau_i$ and $\delta$. We can see in \Cref{fig:DoubleCenters} that quiver automorphisms can be obtained in case from this central element. We note that in the $\db{D}{4}$ case we obtain $\sigma_{12} = r (\tau_3\tau_4r\delta)^2$, as can be seen via the folding $\db{D}{4} \rightarrow \db{B}{3}$

\end{proof}

\subsection{Other cases}

In the previous section, we ignored the $A$ and $BC$ cases. These cases are simpler, so we simply show their cluster modular groups.
\begin{align}
    \Gamma_{\db{A}{1}} &= \mathcal{B}(A_2)/\mathcal{Z} = \PSL{2}{Z} \\
    \Gamma_{\dbf{BC}{1}{4}{1}} &= \Gamma_{\dbf{BC}{1}{4}{1}} = \mathcal{B}(B_2)/\mathcal{Z} = \mathbb{Z}*\cyclicgroup{2}\\
    \Gamma_{\dbf{BC}{2}{4}{2}} &= \mathcal{B}(B_2)/\mathcal{Z} \times \cyclicgroup{2} = (\mathcal{Z}*\cyclicgroup{2})\times \cyclicgroup{2}
\end{align}

\subsection{Special quotients and doubly extended associahedra}\label{sec:SpecialQuotients}

We will construct a special finite quotient of the cluster modular group of each of the simply laced doubly extended cluster algebras. We will use this normal subgroup to construct a finite quotient of the cluster complex and thereby construct a doubly extended generalized associahedron. 

Following the ideas in the affine case, we would like to quotient $\Gamma$ by $\groupgenby{\gamma}$. However, $\groupgenby{\gamma}$ is no longer a normal subgroup. We will now construct free normal subgroups $\mathcal{N}$, such that  $\gamma^k \in \mathcal{N} \lhd \Gamma$ and $\Gamma/\mathcal{N}$ is finite group containing the normal subgroup $N$. 

Let $n = ord(r)$ be the order of the reddening element. We can see that in the quotient $\Gamma/N = \PSL{2}{Z}$, we have 
\begin{equation}
    \gamma = \begin{bmatrix} 1 & n \\ 0 & 1
    \end{bmatrix}
\end{equation} 
in each case. 
We denote the normal closure in $\Gamma$ of the group element $\gamma$ by $\NormalClosure{\gamma}$. This is a finite index subgroup of the cluster modular group in all cases other than $\db{E}{8}$ since $\NormalClosure{\gamma}/N$ is finite index in $\PSL{2}{Z}$. This group is not free in the $\db{E}{6} or \db{E}{8}$ cases, but $\NormalClosure{\gamma r^2}$ and $\NormalClosure{\gamma r^4}$ are free in these cases respectively. 

\begin{claim}
For $\db{D}{4}$, the group $\NormalClosure{\gamma}$ is the puncture preserving mapping class group of a four punctured sphere.
\end{claim}

We can verify this claim easily by seeing that $\gamma$ is a Dehn twist. Thus by \cite{Margalit:PMCG} we have that $\NormalClosure{\gamma} \simeq \FreeGroup{2}$. We have an exact sequence
\begin{equation}
    1 \rightarrow \NormalClosure{\gamma} \rightarrow \Gamma_{\db{D}{4}} \rightarrow H \rightarrow 1
\end{equation}
where $H$ is a group of order 1152 given by an extension
\begin{equation}
    1 \rightarrow N \rightarrow H \rightarrow S_3 \rightarrow 1.
\end{equation}

\begin{claim}
For $\db{E}{7}$, the group $\NormalClosure{\gamma}$ is a finite index free group. It is isomorphic to the congruence subgroup $\Bar{\Gamma}(4)$ of $\PSL{2}{Z}$
\end{claim}

 In $\db{E}{7}$ the image of center $\mathcal{Z} = \Z$ is generated by $r$ which has order 4 in $\Gamma$. Thus taking the quotient of the top row of \Cref{eq:double_cmgs} by $2\Z$ results in the following diagram:
\begin{equation}
    \begin{tikzcd}
1 \arrow[r] & \cyclicgroup{2} \arrow[d, hook] \arrow[r, hook] & \SL{2}{Z} \arrow[d, hook] \arrow[r, two heads] & \PSL{2}{Z} \arrow[r] \arrow[d, Rightarrow, no head] & 1                                              \\
1 \arrow[r] & N \arrow[r, hook]                               & \Gamma \arrow[r, two heads]              & \PSL{2}{Z} \arrow[r]                                & 1                            
\end{tikzcd}
\end{equation}
The element $\gamma \in \Gamma$ is in the image of the map from $\SL{2}{Z}$ and is given by the matrix $\begin{bmatrix} 1 & 4 \\ 0 & 1 \end{bmatrix}$. The normal closure of this matrix in $\SL{2}{Z}$ is the level 4 congruence subgroup $\Gamma(4)$ and is torsion free. Since $\gamma$ commutes with all of $N$, its normal closure in $\Gamma$ is isomorphic to its normal closure in $\SL{2}{Z}$. 

Thus we have the following diagram
\begin{equation}
    \begin{tikzcd}
            &                                                  & 1 \arrow[d]                                                  & 1 \arrow[d]                                   &   \\
            &                                                  & \NormalClosure{\gamma} \arrow[d, hook] \arrow[r, Rightarrow, no head]     & \Gamma(4) \arrow[d, hook]                     &   \\
1 \arrow[r] & N \arrow[r, hook] \arrow[d, Rightarrow, no head] & \Gamma \arrow[r, two heads] \arrow[d, two heads]             & \PSL{2}{Z} \arrow[r] \arrow[d, two heads]     & 1 \\
1 \arrow[r] & N \arrow[r, hook]                           & \Gamma/\NormalClosure{\gamma} \arrow[r, two heads] \arrow[d] & S_4 \arrow[r] \arrow[d] & 1 \\
            &                                                  & 1                                                            & 1                                             &  
\end{tikzcd}.
\end{equation}

\begin{claim} The normal closure $\NormalClosure{\gamma r^{2}}$ is a finite index free group of the cluster modular group of $\db{E}{6}$. It is isomorphic to the congruence subgroup $\Bar{\Gamma}(3)$ of $\PSL{2}{Z}$. 
\end{claim}

We have the following diagram of exact sequences: 
\begin{equation}
    \begin{tikzcd}
            &                                                  & 1 \arrow[d]                                                  & 1 \arrow[d]                                   &   \\
            &                                                  & \NormalClosure{\gamma r^2} \arrow[d, hook] \arrow[r, Rightarrow, no head]     & \Gamma(3) \arrow[d, hook]                     &   \\
1 \arrow[r] & N \arrow[r, hook] \arrow[d, Rightarrow, no head] & \Gamma \arrow[r, two heads] \arrow[d, two heads]             & \PSL{2}{Z} \arrow[r] \arrow[d, two heads]     & 1 \\
1 \arrow[r] & N \arrow[r, hook]                           & \Gamma/\NormalClosure{\gamma r^2} \arrow[r, two heads] \arrow[d] & A_4 \arrow[r] \arrow[d] & 1 \\
            &                                                  & 1                                                            & 1                                             &  
\end{tikzcd}.
\end{equation}

\begin{claim}
In the $\db{E}{8}$ case, the normal closure $\NormalClosure{\gamma r^{4}}$ is a free group, but is not of finite index. The groups $\mathcal{N}_k = \NormalClosure{\gamma r^{4}, (r\delta)^k(\tau)^k}$ are free groups of index $36 , 108$ and $144$ for $k = 1, 2, 3$.  
\end{claim}

The images of the groups $\mathcal{N}_k$ in $\PSL{2}{Z}$ are normal subgroups of index $6, 18$ and $24$ for $k =1,2,3$. We denote these groups and their respective quotients by $G_{k, 6/k}$ and $F_{k,6/k}$, see \cite{Newman:modular_subgroups}. 

We have the following diagram of exact sequences: 
\begin{equation}
    \begin{tikzcd}
            &                                                  & 1 \arrow[d]                                                  & 1 \arrow[d]                                   &   \\
            &                                                  & \mathcal{N}_k \arrow[d, hook] \arrow[r, Rightarrow, no head]     & G_{k,6/k} \arrow[d, hook]                     &   \\
1 \arrow[r] & \cyclicgroup{6} \arrow[r, hook] \arrow[d, Rightarrow, no head] & \Gamma \arrow[r, two heads] \arrow[d, two heads]             & \PSL{2}{Z} \arrow[r] \arrow[d, two heads]     & 1 \\
1 \arrow[r] & \cyclicgroup{6} \arrow[r, hook]                           & \Gamma/\mathcal{N}_k \arrow[r, two heads] \arrow[d] & F_{k,6/k} \arrow[r] \arrow[d] & 1 \\
            &                                                  & 1                                                            & 1                                             &  
\end{tikzcd}.
\end{equation}

While we have not explicitly described them, there are analogous finite index normal free subgroups of each of the non simply laced doubly extended cluster modular groups. These can be understood by folding the simply laced algebras. We can define doubly extended generalized associahedra by first quotienting the cluster complexes by the action of these subgroups and then dualizing.  

\subsection{Counting facets in doubly extended associahedra}

We can compute the total number of cluster variables and clusters in the quotient of doubly extended cluster complexes by the distinguished normal subgroup $\mathcal{N}$. The number of cluster variables is equal to the number of corank 1 subalgebras of our given algebra and is equal to the number of codimension 1 facets of the generalized associahedra. 

Recall from \Cref{sec:SpecialQuotients} that the cluster modular group modulo $\mathcal{N}$ factors as
\begin{align*}
%    1 \rightarrow N = \Gamma_\tau^\circ \rtimes \Aut{\T} \rightarrow \ClusterModularGroup{} \rightarrow \PSL{2}{\Z} \rightarrow 1\\
    1 \rightarrow N \rightarrow \ClusterModularGroup{}/\mathcal{N} \rightarrow F \rightarrow 1 
\end{align*}  
where $N = \Gamma_\tau^\circ \rtimes \Aut{\T}$ as usual and $F$ is a finite group. We will use this factorization to count elements of $\ClusterComplex{}/\mathcal{N}$ by counting the size of orbits by $N$ and then considering the action of $F$ on these orbits.

\begin{theorem}\label{thm:DoubleExtendedCountVariables}
    The number of cluster variables in $\ClusterComplex{}/\mathcal{N}$ is 
    \begin{equation}
    d|F| \frac{w_1}{n_1}(\sum(n_i-1)n_i)
\end{equation} where $d=1$ unless $\A$ is self dual and nonsimply laced, in which case we have $d=2$. 
\end{theorem}
\begin{proof}
    We show in \Cref{thm:DoublyExtendedVariableLocation} that every cluster variable appears on the tail of $\T$ quiver. In the case when $\A$ is self dual and non-simply laced, then every cluster variable appears uniquely on the tail of a $\T$ quiver or its dual. In this case we will multiply our final count by 2. 
    
    Thus it suffices to count the variables in a $\ClusterModularGroup{}/\mathcal{N}$ orbit of a $\T$ quiver. By the factorization, we first look at orbits of $N$ and then orbits of $F$.     Since the red to green element is in $N$ and the action of this element twists all of the tails we see that each orbit of $N$ on a single $\T$ quivers has $(\sum(n_i-1)n_i) $ cluster variables. 
    
    Now we consider the action of $F$ on these orbits. In each case the element $\gamma$ maps to the matrix $\begin{bmatrix} 1 & n_1 \\ 0 & 1 \end{bmatrix}$ while the twist $\tau_1$ maps to $\begin{bmatrix} 1 & w_1 \\ 0 & 1 \end{bmatrix}$. The image of $\tau_1$ in $F$ clearly fixes the orbits of cluster variables by $N$ we just counted. This must generate the stabilizer since these are the only parabolic elements in $\PSL{2}{\Z}$ with the same stabilizer. Therefore to get the full count we multiply by $|F|\frac{w_1}{n_1}$.
\end{proof}

\begin{lemma}\label{thm:DoublyExtendedVariableLocation}
    Every cluster variable in a doubly extended cluster algebra appears on the tail of a $\T$ quiver (or its dual for nonsimply laced self dual cluster algebras)
\end{lemma}
\begin{proof}
    In each of the finitely many simply laced cases one can check that every cluster variable appears in a $\T$ quiver not on the double edge. This is a finite computation, as we only have to check each location in each quiver isomorphism class has a mutation path to a quiver with a double edge. For most cases this requires extensive computational aid.\footnote{See \url{https://zngzag42.github.io/DoubleExtendedStructureProof/} for full computational details.} However $\db{D}{4}$ only has 4 isomorphism classes so we can show the full computation in \Cref{fig:D4DoubleDoubleEdge}.

\begin{figure}[hb]
    \centering
    \begin{subfigure}{.23\textwidth}
        \includegraphics[width=\textwidth]{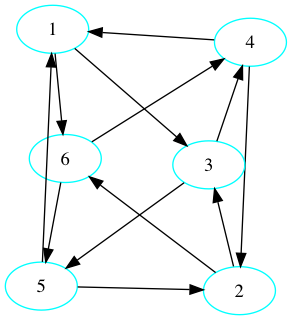}
    \caption{$Q_1$}        
    \end{subfigure}
    \begin{subfigure}{.23\textwidth}
        \includegraphics[width=\textwidth]{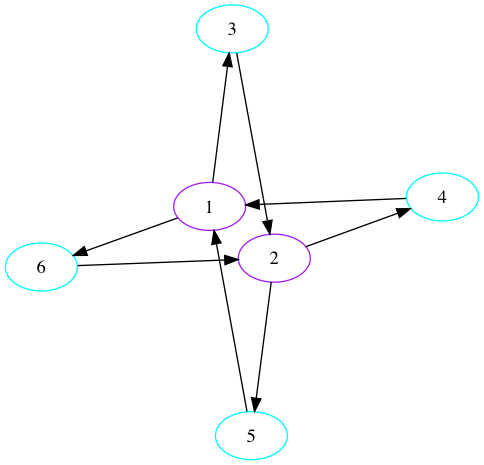}
    \caption{$Q_2$}        
    \end{subfigure}
    \begin{subfigure}{.23\textwidth}
        \includegraphics[width=\textwidth]{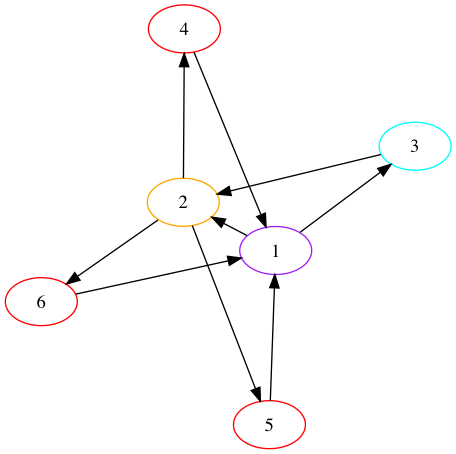}
    \caption{$Q_3$}        
    \end{subfigure}
    \begin{subfigure}{.23\textwidth}
        \includegraphics[width=\textwidth]{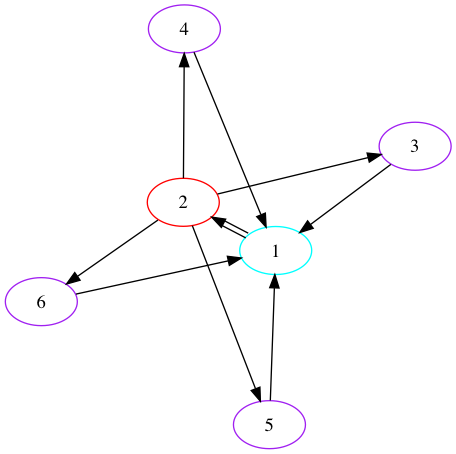}
    \caption{$Q_4$}        
    \end{subfigure}\\
    \begin{subfigure}{.9\textwidth}
        \centering
        \begin{tabular}{ c c c c c }
        i & $Q_1$ & $Q_2$ & $Q_3$ & $Q_4$\\
         \hline\hline
        1 & 6,4,5 & 2,6,5,4 & 2,6,5,4 & 4,2,3,5,6\\
         2 & 6,4,5 & 1,5,6,3 & 1,6,5,4 & 4,1,3,5,6\\
         3 & 5,1,2 & 4,5 & 6,4,5 & $\langle\rangle$\\
         4 & 2,3,6 & 3,6 & 3 & $\langle\rangle$\\
         5 & 2,3,6 & 3,6 & 3 & $\langle\rangle$\\
         6 & 5,1,2 & 4,5 & 3 & $\langle\rangle$\\
        \end{tabular}
    \end{subfigure}
    \caption{The four quiver isomorphism classes for $\db{D}{4}$ and mutation paths so that vertex $i$ is in a double edge quiver without mutating $i$}
    \label{fig:D4DoubleDoubleEdge}
\end{figure}
  We note that for a nonsimply laced cluster algebra the analogous computation shows that in self dual quivers each variable occurs either on $\T$ or its dual.
\end{proof}

We can now count the number of clusters in the cluster complex modulo $\mathcal{N}$. Each cluster variable corresponds to a corank 1 subalgebra of the our cluster algebra. By the proof \Cref{thm:DoubleExtendedCountVariables}, these subalgebras can always be found by freezing variables on the tails of  $\T$ quivers. Thus, every corank 1 subalgebra is affine type.

Let $\A_{i_j}$ be the affine subalgebra obtained by freezing the tail node $i_j$ and let $C_{i_j}$ be the number of clusters in $\A_{i_j}$ up to $\gamma$. Since we are quotienting by $\mathcal{N}$ which contains $\gamma$, the number of clusters in each affine subalgebra in this quotient complex is equal $d|F|\frac{w_1}{n_1}C_{i_j}$. Then the total number of clusters is
\begin{equation}\label{eq:CountingDouble}
    \frac{1}{n}d|F|\frac{w_1}{n_1}\sum_i n_i\sum_{j=2}^{n_i} C_{i_j}
\end{equation} where $n$ is the rank of the doubly extended cluster algebra we are considering. The factor of $1/n$ appears since each cluster appears in $n$ corank 1 subalgebras. 

More generally, we can count the number of any dimension facets on a doubly extended associahedron using the formula 
\begin{equation}
    |C_k(\A)| = \frac{1}{n-k}\sum_{\mathcal{B} \in C^{1}(\A)}C_k(\mathcal{B})
\end{equation} of \Cref{thm:AffineCountingFacets}.

\begin{example}
We will compute the number of clusters in the quotient complex of type $\db{E}{7}$ by $\mathcal{N}$. 
By freezing nodes on a tail of length 4 we can obtain subalgebras of type $\Affine{E}_7$, $\Affine{D}_{6}\times A_1$, $A_{2,4}\times A_2$. These have sizes $25 200$, $5040$, and $1400$ respectively. There is only one node to freeze on the tail of length 2 corresponding to a subalgebra of type $A_{4,4}$ which contains $4900$ clusters up to the action of $\gamma$. So the total number of clusters in $\ClusterComplex{}/\mathcal{N}$ is
\begin{equation}
    \frac{24}{9}\left(2\cdot 4\left(\frac{25,200}{4}+\frac{5040}{4}+\frac{1400}{4}\right) + 2\frac{4900}{4}\right) = 24\frac{65730}{9}= 175 280.
\end{equation}

\end{example}

\begin{example}
    In \Cref{fig:B2G2Dubs} we see the doubly extended affine associahedra of types $\dbf{B}{2}{2}{1}$ and $\db{G}{2}$. 
\end{example}

We note that doubly extended associahedra are not expected to be homotopy equivalent to spheres. Let $\A$ be a doubly extended cluster algebra of rank $n+2$. We conjecture the following:

\begin{conjecture}\label{conj:DoubleAssociahedra}
The exchange complex of $\A$ is homotopy equivalent to $S^{n-1}$. The doubly extended associahedron associated with $\A$ is homotopy equivalent to $S^{n-1} \times S^2$ in all cases other than $\db{E}{8}$ where it instead is homomorphic to $S^7 \times S^1 \times S^1$.
\end{conjecture}

\begin{figure}
    \centering
    \begin{subfigure}{.49\textwidth}
    \centering
    \includegraphics[width=\textwidth]{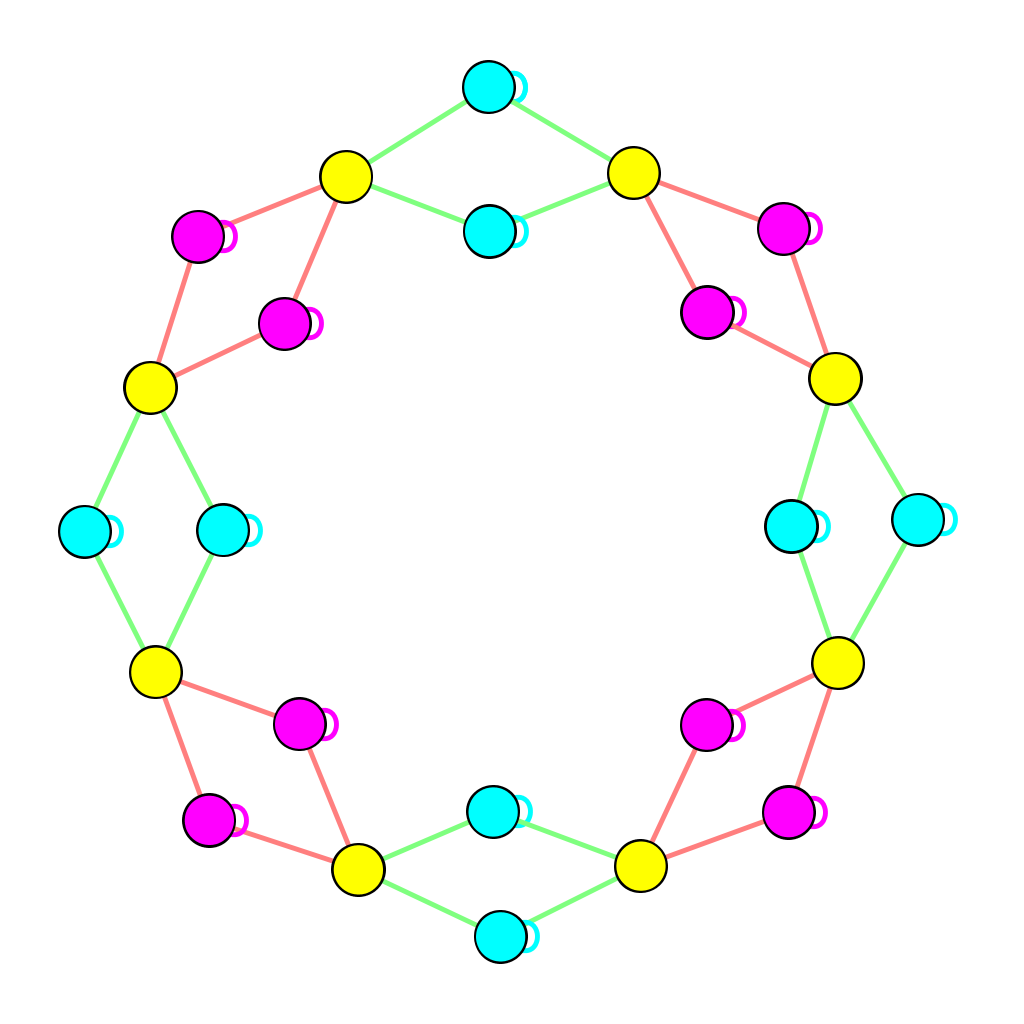}
    \caption{$\dbf{B}{2}{2}{1}$}
    \end{subfigure}
    \begin{subfigure}{.49\textwidth}
    \centering
    \includegraphics[width=\textwidth]{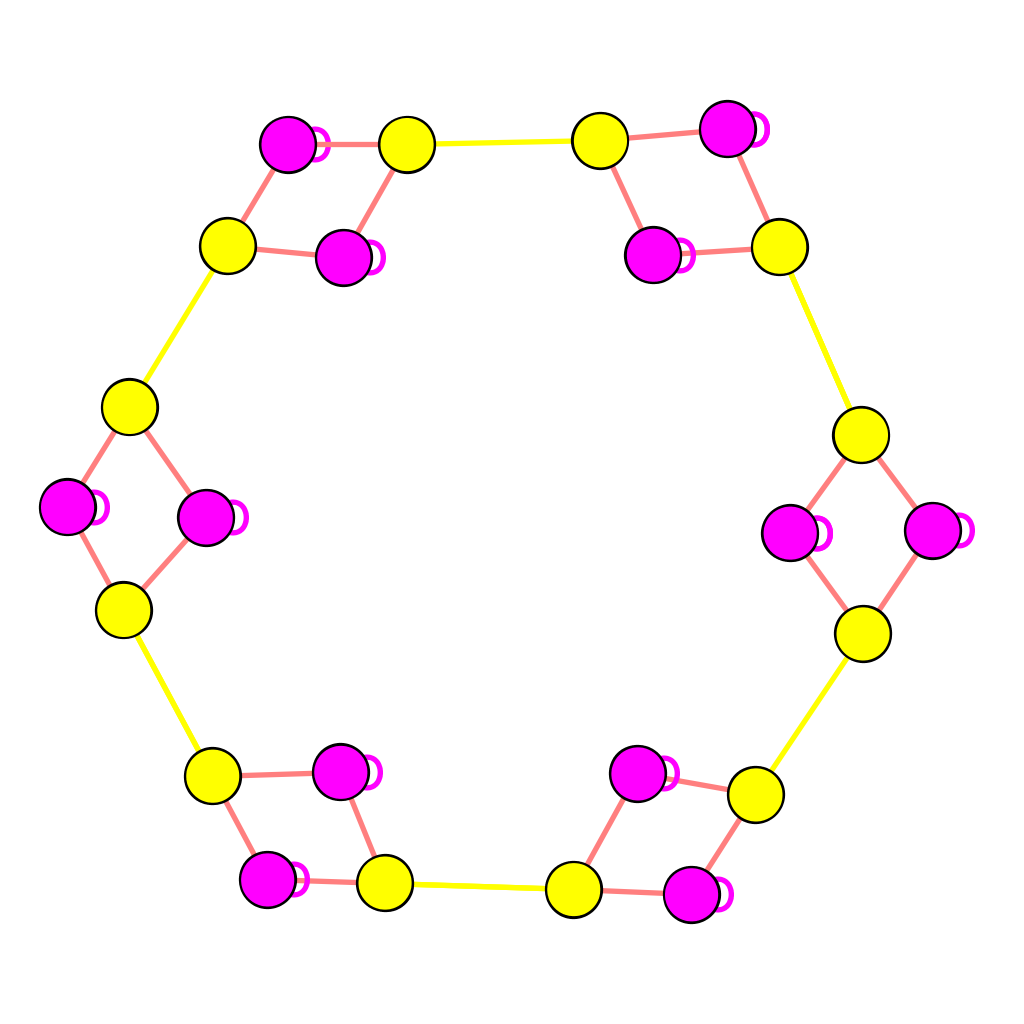}
    \caption{$\db{G}{2}$}
    \end{subfigure}
    \caption{The 1-skeleton of the doubly extended associahedra of types $\dbf{B}{2}{2}{1}$ and $\db{G}{2}$ }
    \label{fig:B2G2Dubs}
\end{figure}

\Cref{fig:DoublyExtendedCountingClusters} contains the results of the counting arguments for the number of clusters in the other doubly extended cases. We include the $A$ and $BC$ cases, which can be done individually and are somewhat degenerate. \Cref{fig:DoublyExtendedCodimCounts} shows the total count of codimension $k$ subalgebras obtained by inductively counting corank $1$ subalgebras. This allows use to compute the Euler characteristic of each space, which match the topologies given in \Cref{conj:DoubleAssociahedra}. 
\begin{figure}[hb]
    \centering
    \begin{tabular}{c c c c}
   Type  & Number of cluster variables in $\ClusterModularGroup{}/\mathcal{N}$ & $|F=(\Gamma /\mathcal{N})/N|$ &  Number of clusters in $\ClusterModularGroup{}/\mathcal{N}$\\
   \hline
   \hline
    $\db{A}{1}$ & 3 & 1 & 1\\
    \hline
    $\db{D}{4}$ & 24 & 6 & 432\\
    \hline
    $\db{E}{6}$ & 72 & 12 & 18,900\\
    \hline
    $\db{E}{7}$ & 156 & 24 & 175,280\\
    \hline
    $\db{E}{8}$ & 38\hspace{1pc}114\hspace{1pc}152 & 6\hspace*{1pc}18\hspace*{1pc}24 & 204,630\hspace*{1pc}613,890\hspace*{1pc}818,520\\
    \hline
    \hline
    $\dbf{BC}{1}{4}{1}$ & 3 & 2 & 2\\
    \hline 
    $\dbf{B}{2}{2}{1}$ & 16 & 2 & 24\\
    \hline
    $\dbf{BC}{2}{4}{2}$ & 16 & 2 & 24\\
    \hline
    $\dbf{G}{2}{1}{1}$ & 12& 6& 24\\
    \hline
    $\dbf{G}{2}{3}{1}$ & 36 & 3 & 63  \\
    \hline
    $\dbf{B}{3}{1}{1}$ & 18 & 6 & 108 \\
    \hline
    $\dbf{F}{4}{1}{1}$ & 48 & 12 & 1260  \\
    \hline
    $\dbf{F}{4}{2}{1}$ & 112 & 8 & 2784  \\
    \hline
\end{tabular}
    \caption{Counting clusters in quotient of doubly extended cluster algebras}
    \label{fig:DoublyExtendedCountingClusters}
\end{figure}
\begin{landscape}
\begin{figure}
    \centering
    \begin{tabular}{c c c c c c c c c c c}
        Type  & 1 & 2 & 3& 4& 5 & 6& 7 & 8 & 9 & 10\\
        \hline
        \hline
        $\db{A}{1}$ & 3 & $\frac{3}{2}$ & 1\\
        \hline
        $\db{D}{4}$ &24& 192& 768& 1,464& 1,296& 432\\
    \hline
    $\db{E}{6}$ & 72& 1,422& 11,772& 47,466& 102,816& 122,472& 75,600& 18,900\\
    \hline
    $\db{E}{7}$ & 156& 4,776& 53,504& 288,840& 857,760& 1,478,400& 1,474,080& 788,760& 175280\\
    \hline
    \multirow{3}{*}{$\db{E}{8}$} & 38& 1,881& 28,046& 196,345& 763,398& 177,6042& 2,531,988& 2,167,722& 1,023,150& 204,630 \\
    
      & 114 & 5,643& 84,138& 589,035& 2,290,194& 5,328,126& 7,595,964& 6,503,166 & 3,069,450& 613,890\\
     & 152& 7,524& 112,184& 785,380& 3,053,592& 7,104,168& 10,127,952& 8,670,888 & 4,092,600& 818,520\\
     \hline
     \hline
     $\dbf{BC}{1}{4}{1}$ & 3 & 3 & 2 \\
     \hline
     $\dbf{B}{2}{2}{1}$ & 16 & 40 & 48 & 24 \\
     \hline
     $\dbf{BC}{2}{4}{2}$ & 16 & 40 & 48 & 24\\
     \hline
     $\db{G}{2}$ & 12 & 36 & 48 & 24\\
     \hline
     $\dbf{G}{2}{3}{1}$ & 36 & 99 & 126 & 63\\
     \hline
     $\db{B}{3}$ & 18 & 96 & 244 & 270 & 108\\
     \hline
     $\db{F}{4}$ & 48 & 516 & 2196 & 4248 & 3780  & 1260\\
     \hline
     $\dbf{F}{4}{2}{1}$ & 112 & 1152 & 4864 & 9392 & 8352 & 2784 \\
     \hline
    \end{tabular}
    \caption{Number of codimension $k$ facet in the doubly extended generalized associahedra.}
    \label{fig:DoublyExtendedCodimCounts}
\end{figure}
\end{landscape}

\clearpage

\bibliographystyle{alpha}
\bibliography{modulargroup.bib}

\begin{thebibliography}{GHKK18}

\bibitem[ASS12]{Schiffler:cluster_automorphisms}
Ibrahim Assem, Ralf Schiffler, and Vasilisa Shramchenko.
\newblock Cluster automorphisms.
\newblock {\em Proc. Lond. Math. Soc. (3)}, 104(6):1271--1302, 2012.

\bibitem[Bas10]{basak:CombinatorialCellComplex}
Tathagata Basak.
\newblock Combinatorial cell complexes and {P}oincar\'e duality.
\newblock {\em Geometriae Dedicata}, 147(1):357--387, 2010.

\bibitem[DWZ10]{DWZ:quivers_with_potentials}
Harm Derksen, Jerzy Weyman, and Andrei Zelevinsky.
\newblock Quivers with potentials and their representations {II}: applications to cluster algebras.
\newblock {\em J. Amer. Math. Soc.}, 23(3):749--790, 2010.

\bibitem[FG06]{FockGonch:Moduli_of_local_systems}
Vladimir~V. Fock and Alexander~B. Goncharov.
\newblock Moduli spaces of local systems and higher {T}eichm\"{u}ller theory.
\newblock {\em Publ. Math. Inst. Hautes \'{E}tudes Sci.}, (103):1--211, 2006.

\bibitem[FG09]{FockGonch:cluster_ensembles}
Vladimir~V. Fock and Alexander~B. Goncharov.
\newblock Cluster ensembles, quantization and the dilogarithm.
\newblock {\em Ann. Sci. \'{E}c. Norm. Sup\'{e}r. (4)}, 42(6):865--930, 2009.

\bibitem[FM12]{Margalit:PMCG}
Benson Farb and Dan Margalit.
\newblock {\em A primer on mapping class groups}, volume~49 of {\em Princeton Mathematical Series}.
\newblock Princeton University Press, Princeton, NJ, 2012.

\bibitem[Fra20]{fraser_braid_2020}
Chris Fraser.
\newblock Braid group symmetries of {Grassmannian} cluster algebras.
\newblock {\em Selecta Mathematica}, 26(2):17, May 2020.

\bibitem[FST08]{FST:Triangulated_surfaces}
Sergey Fomin, Michael Shapiro, and Dylan Thurston.
\newblock Cluster algebras and triangulated surfaces. {I}. {C}luster complexes.
\newblock {\em Acta Math.}, 201(1):83--146, 2008.

\bibitem[FST12]{FST:finite_mutation_via_unfoldings}
Anna Felikson, Michael Shapiro, and Pavel Tumarkin.
\newblock Cluster algebras of finite mutation type via unfoldings.
\newblock {\em Int. Math. Res. Not. IMRN}, (8):1768--1804, 2012.

\bibitem[FZ02]{fomin_clusterI}
Sergey Fomin and Andrei Zelevinsky.
\newblock Cluster algebras. {I}. {F}oundations.
\newblock {\em J. Amer. Math. Soc.}, 15(2):497--529, 2002.

\bibitem[FZ03]{fomin_y-systems_2003}
Sergey Fomin and Andrei Zelevinsky.
\newblock Y-systems and generalized associahedra.
\newblock {\em Annals of Mathematics. Second Series}, 158(3):977--1018, 2003.

\bibitem[GHKK18]{GHHK:canonical}
Mark Gross, Paul Hacking, Sean Keel, and Maxim Kontsevich.
\newblock Canonical bases for cluster algebras.
\newblock {\em J. Amer. Math. Soc.}, 31(2):497--608, 2018.

\bibitem[GS18]{goncharov_donaldsonthomas_2018}
Alexander Goncharov and Linhui Shen.
\newblock Donaldson–{Thomas} transformations of moduli spaces of {G}-local systems.
\newblock {\em Advances in Mathematics}, 327:225--348, 2018.

\bibitem[Ish19]{Ishibashi:nielsen-thurston}
Tsukasa Ishibashi.
\newblock On a {N}ielsen-{T}hurston classification theory for cluster modular groups.
\newblock {\em Ann. Inst. Fourier (Grenoble)}, 69(2):515--560, 2019.

\bibitem[Ish20]{Ishibashi:cluster_modular}
Tsukasa Ishibashi.
\newblock Presentations of cluster modular groups and generation by cluster {D}ehn twists.
\newblock {\em SIGMA Symmetry Integrability Geom. Methods Appl.}, 16:Paper No. 025, 22, 2020.

\bibitem[Kau24]{KaufmanSpecialFoldingQuivers2024a}
Dani Kaufman.
\newblock Special folding of quivers and cluster algebras.
\newblock {\em MATHEMATICA SCANDINAVICA}, 130(2), May 2024.

\bibitem[Mul16]{Muller_maximalgreensequences}
Greg Muller.
\newblock The existence of a maximal green sequence is not invariant under quiver mutation.
\newblock {\em Electron. J. Combin.}, 23(2):Paper 2.47, 23, 2016.

\bibitem[New67]{Newman:modular_subgroups}
Morris Newman.
\newblock Classification of normal subgroups of the modular group.
\newblock {\em Trans. Amer. Math. Soc.}, 126:267--277, 1967.

\bibitem[NZ12]{Zelevinsky:Tropical_dualities}
Tomoki Nakanishi and Andrei Zelevinsky.
\newblock On tropical dualities in cluster algebras.
\newblock In {\em Algebraic groups and quantum groups}, volume 565 of {\em Contemp. Math.}, pages 217--226. Amer. Math. Soc., Providence, RI, 2012.

\bibitem[Pen08]{Penner:arc_complexes}
R.~C. Penner.
\newblock The structure and singularities of quotient arc complexes.
\newblock {\em J. Topol.}, 1(3):527--550, 2008.

\bibitem[Sai85]{Saito:Extended_affine_root_systems}
Kyoji Saito.
\newblock Extended affine root systems. {I}. {C}oxeter transformations.
\newblock {\em Publ. Res. Inst. Math. Sci.}, 21(1):75--179, 1985.

\bibitem[Wil14]{williams_ClusterAlgebrasAnIntroduction}
Lauren~K. Williams.
\newblock Cluster algebras: an introduction.
\newblock {\em Bull. Amer. Math. Soc. (N.S.)}, 51(1):1--26, 2014.

\end{thebibliography}
\pagebreak\clearpage

\appendix
\renewcommand{\thesection}{\Alph{section}}
\section{Dynkin Diagrams}\label{app:DynkinDiagrams}

For reference we include all the finite, affine, and doubly extended Dynkin diagrams. To align with the cluster algebras, we draw the non simply laced diagrams using ``fat'' nodes whose weight (\Cref{fig:dynkinNodeWeight}) corresponds to the number of nodes ``folded'' together from the simply laced diagram. In the standard root system language these fat nodes correspond to the shorter roots of the root system. In the $B,C,F$ cases the fat nodes are all weight 2. In the $BC$ case there are nodes of weight 2 and 4. The $G$ case has nodes of weight $3$.

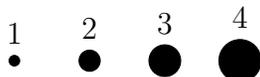
\begin{figure}[hb]
    \centering
    \begin{tikzpicture}
    \node[base] (0) [label= 1] {};
    \node[fat2] (1) [right of = 0,label=2] {};
    \node[fat3] (2) [right of = 1,label=3] {};
    \node[fat4] (3) [right of = 2,label=4] {};
    \end{tikzpicture}
    \caption{Weights of nodes in Dynkin Diagrams}
    \label{fig:dynkinNodeWeight}
\end{figure}

%%Finite
\begin{figure}[hb]
    \centering
    \begin{subfigure}{.4\textwidth}
    \centering
     \begin{tikzpicture}
        \node[base] (0) [] {};
        \node[base] (1) [right of = 0] {};
        \node[base] (2) [right of = 1] {};
        \node[base] (3) [right of = 2] {};
        \node[base] (4) [right of = 3] {};
        \node[base] (5) [right of = 4] {};
        \path[-] (0) edge [left] node {} (1); 
        \path[-] (1) edge [left] node {} (2);
        \path[-] (2) edge [left,dotted] node {} (3); 
        \path[-] (3) edge [left] node {} (4); 
        \path[-] (4) edge [left] node {} (5); 
    \end{tikzpicture}
    \caption{$A_n$}
    \end{subfigure}
  \begin{subfigure}{.4\textwidth}
      \centering
        \begin{tikzpicture}
        \node[base] (0) [] {};
        \node[base] (2) [right of = 0] {};    \node[base] (1) [above of = 2] {};
        \node[base] (3) [right of = 2] {};
        \node[base] (4) [right of = 3] {};
        \node[base] (5) [right of = 4] {};
        \path[-] (0) edge [left] node {} (2); 
        \path[-] (1) edge [left] node {} (2);
        \path[-] (2) edge [left] node {} (3); 
        \path[-] (3) edge [left,dotted] node {} (4); 
        \path[-] (4) edge [left] node {} (5); 
        \end{tikzpicture}
        \caption{$D_n$}
  \end{subfigure}\\
    \begin{subfigure}{.4\textwidth}
        \centering
            \begin{tikzpicture}
            \node[base] (0) [] {};
            \node[base] (2) [right of = 0] {};    
            \node[base] (3) [right of = 2] {};
            \node[base] (1) [above of = 3] {};
            \node[base] (4) [right of = 3] {};
            \node[base] (5) [right of = 4] {};
            \path[-] (0) edge [left] node {} (2); 
            \path[-] (1) edge [left] node {} (3);
            \path[-] (2) edge [left] node {} (3); 
            \path[-] (3) edge [left] node {} (4); 
            \path[-] (4) edge [left] node {} (5);
            \end{tikzpicture}
        \caption{$E_6$}
    \end{subfigure}
    \begin{subfigure}{.4\textwidth}
      \centering
        \begin{tikzpicture}
        \node[base] (0) [] {};
        \node[base] (2) [right of = 0] {};    
        \node[base] (3) [right of = 2] {};
        \node[base] (1) [above of = 3] {};
        \node[base] (4) [right of = 3] {};
        \node[base] (5) [right of = 4] {};
        \node[base] (6) [right of = 5] {};
        \path[-] (0) edge [left] node {} (2); 
        \path[-] (1) edge [left] node {} (3);
        \path[-] (2) edge [left] node {} (3); 
        \path[-] (3) edge [left] node {} (4); 
        \path[-] (4) edge [left] node {} (5);
        \path[-] (5) edge [left] node {} (6);
        \end{tikzpicture}
        \caption{$E_7$}
    \end{subfigure}
  \begin{subfigure}{.7\textwidth}
      \centering
        \begin{tikzpicture}
        \node[base] (0) [] {};
        \node[base] (2) [right of = 0] {};    
        \node[base] (3) [right of = 2] {};
        \node[base] (1) [above of = 3] {};
        \node[base] (4) [right of = 3] {};
        \node[base] (5) [right of = 4] {};
        \node[base] (6) [right of = 5] {};
        \node[base] (7) [right of = 6] {};
        \path[-] (0) edge [left] node {} (2); 
        \path[-] (1) edge [left] node {} (3);
        \path[-] (2) edge [left] node {} (3); 
        \path[-] (3) edge [left] node {} (4); 
        \path[-] (4) edge [left] node {} (5);
        \path[-] (5) edge [left] node {} (6);
        \path[-] (6) edge [left] node {} (7);
        \end{tikzpicture}
     \caption{$E_8$}
     \end{subfigure}
  \caption{Simply Laced Finite Dynkin Diagrams}
    \label{fig:DynkinFiniteSimplyLaced}
  \end{figure}
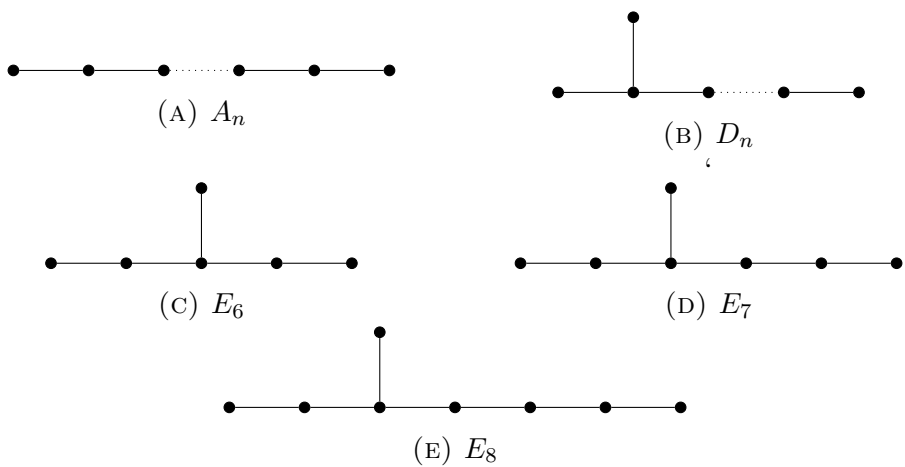 
%%Finite Folded  
\begin{figure}[hb]
      
   \centering
  \begin{subfigure}{.4\textwidth}
        \centering
        \begin{tikzpicture}
        \node[fat2] (0) [] {};
        \node[base] (1) [right of = 0] {};
        \node[base] (2) [right of = 1] {};
        \node[base] (3) [right of = 2] {};
        \node[base] (4) [right of = 3] {};
        \node[base] (5) [right of = 4] {};
        \path[-] (0) edge [left] node {} (1); 
        \path[-] (1) edge [left] node {} (2);
        \path[-] (2) edge [left,dotted] node {} (3); 
        \path[-] (3) edge [left] node {} (4); 
        \path[-] (4) edge [left] node {} (5); 
    \end{tikzpicture}
  \caption{$B_n$}
  \end{subfigure}

 \begin{subfigure}{.4\textwidth}
        \centering
         \begin{tikzpicture}
        \node[base] (0) [] {};
        \node[fat2] (1) [right of = 0] {};
        \node[fat2] (2) [right of = 1] {};
        \node[fat2] (3) [right of = 2] {};
        \node[fat2] (4) [right of = 3] {};
        \node[fat2] (5) [right of = 4] {};
        \path[-] (0) edge [left] node {} (1); 
        \path[-] (1) edge [left] node {} (2);
        \path[-] (2) edge [left,dotted] node {} (3); 
        \path[-] (3) edge [left] node {} (4); 
        \path[-] (4) edge [left] node {} (5); 
    \end{tikzpicture}
        \caption{$C_n$}
  \end{subfigure}

    \begin{subfigure}{.4\textwidth}
        \centering
        \begin{tikzpicture}
        \node[base] (0) [] {};
        \node[base] (1) [right of = 0] {};
        \node[fat2] (2) [right of = 1] {};
        \node[fat2] (3) [right of = 2] {};
        \path[-] (0) edge [left] node {} (1);
        \path[-] (1) edge [left] node {} (2);
        \path[-] (2) edge [left] node {} (3);
        \end{tikzpicture}
        \caption{$F_4$}
    \end{subfigure}
    \begin{subfigure}{.4\textwidth}
        \centering
        \begin{tikzpicture}
        \node[base] (0) [] {};
        \node[fat3] (1) [right of = 0] {};
        \path[-] (0) edge [left] node {} (1);
        \end{tikzpicture}
        \caption{$G_2$}
    \end{subfigure}
  \caption{Folded Finite Dynkin Diagrams}
    \label{fig:DynkinFiniteFolded}
  \end{figure}
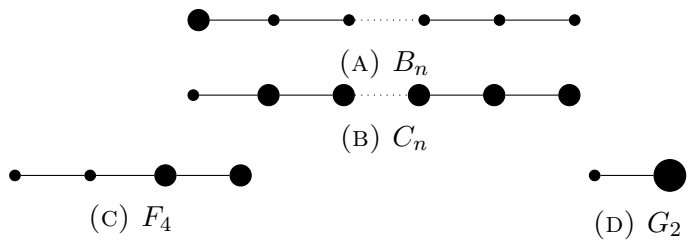
 \clearpage 
  Each affine diagram can be formed by adding a single node to the corresponding finite diagram. In  \Cref{fig:DynkinAffineSimplyLaced,fig:DynkinAffineFolded,fig:DynkinAffineTwisted} the nodes that could be the extension are colored red.
%%Affine
\begin{figure}[hb]
    \centering
    \begin{subfigure}{.4\textwidth}
    \centering
     \begin{tikzpicture}
        \node[affine] (0) [] {};
        \node[affine] (1) [right of = 0] {};
        \node[affine] (2) [right of = 1] {};
        \node[affine] (3) [right of = 2] {};
        \node[affine] (4) [right of = 3] {};
        \node[affine] (5) [right of = 4] {};
        \path[-] (0) edge [left] node {} (1); 
        \path[-] (1) edge [left] node {} (2);
        \path[-] (2) edge [left,dotted] node {} (3); 
        \path[-] (3) edge [left] node {} (4); 
        \path[-] (4) edge [left] node {} (5); 
        \path[-] (5) edge [bend left] node {} (0);
    \end{tikzpicture}
    \caption{$\Affine{A}_n$}
    \end{subfigure}
  \begin{subfigure}{.4\textwidth}
      \centering
        \begin{tikzpicture}
        \node[affine] (0) [] {};
        \node[base] (2) [right of = 0] {};    \node[affine] (1) [above of = 2] {};
        \node[base] (3) [right of = 2] {};
        \node[base] (4) [right of = 3] {};
        \node[base] (5) [right of = 4] {};
        \node[affine] (6) [right of = 5] {};
        \node[affine] (7) [above of = 5] {};
        \path[-] (0) edge [left] node {} (2); 
        \path[-] (1) edge [left] node {} (2);
        \path[-] (2) edge [left] node {} (3); 
        \path[-] (3) edge [left,dotted] node {} (4); 
        \path[-] (4) edge [left] node {} (5); 
        \path[-] (5) edge [left] node {} (6); 
        \path[-] (7) edge [left] node {} (5); 
        \end{tikzpicture}
        \caption{$\Affine{D}_n$}
  \end{subfigure}\\
    \begin{subfigure}{.4\textwidth}
        \centering
            \begin{tikzpicture}
            \node[affine] (0) [] {};
            \node[base] (2) [right of = 0] {};    
            \node[base] (3) [right of = 2] {};
            \node[base] (1) [above of = 3] {};
            \node[base] (4) [right of = 3] {};
            \node[affine] (5) [right of = 4] {};
            \node[affine] (7) [above of = 1] {};
            \path[-] (0) edge [left] node {} (2); 
            \path[-] (1) edge [left] node {} (3);
            \path[-] (2) edge [left] node {} (3); 
            \path[-] (3) edge [left] node {} (4); 
            \path[-] (4) edge [left] node {} (5);
            \path[-] (7) edge [left] node {} (1);
            \end{tikzpicture}
        \caption{$\Affine{E}_6$}
    \end{subfigure}
    \begin{subfigure}{.5\textwidth}
      \centering
        \begin{tikzpicture}
        \node[affine] (0) [] {};
        \node[base] (1) [right of = 0] {};    
        \node[base] (2) [right of = 1] {};
        \node[base] (6) [right of = 2] {};
        \node[base] (5) [right of = 6] {};
        \node[base] (4) [right of = 5] {};
        \node[affine] (3) [right of = 4] {};
        \node[base] (8) [above of = 6] {};
        
        \path[-] (0) edge [] node {} (1);
        \path[-] (1) edge [] node {} (2);
        \path[-] (2) edge [] node {} (6);

        \path[-] (3) edge [] node {} (4);
        \path[-] (4) edge [] node {} (5);
        \path[-] (5) edge [] node {} (6);

        \path[-] (6) edge [] node {} (8);
        \end{tikzpicture}

        \caption{$\Affine{E}_7$}
    \end{subfigure}
  \begin{subfigure}{.7\textwidth}
      \centering
        \begin{tikzpicture}
        \node[base] (0) [] {};
        \node[base] (2) [right of = 0] {};    
        \node[base] (3) [right of = 2] {};
        \node[base] (1) [above of = 3] {};
        \node[base] (4) [right of = 3] {};
        \node[base] (5) [right of = 4] {};
        \node[base] (6) [right of = 5] {};
        \node[base] (7) [right of = 6] {};
        \node[affine] (8) [right of = 7] {};
        \path[-] (0) edge [left] node {} (2); 
        \path[-] (1) edge [left] node {} (3);
        \path[-] (2) edge [left] node {} (3); 
        \path[-] (3) edge [left] node {} (4); 
        \path[-] (4) edge [left] node {} (5);
        \path[-] (5) edge [left] node {} (6);
        \path[-] (6) edge [left] node {} (7);
        \path[-] (7) edge [left] node {} (8);
        \end{tikzpicture}
     \caption{$\Affine{E}_8$}
     \end{subfigure}
  \caption{Simply Laced Affine Dynkin Diagrams}
    \label{fig:DynkinAffineSimplyLaced}
  \end{figure}
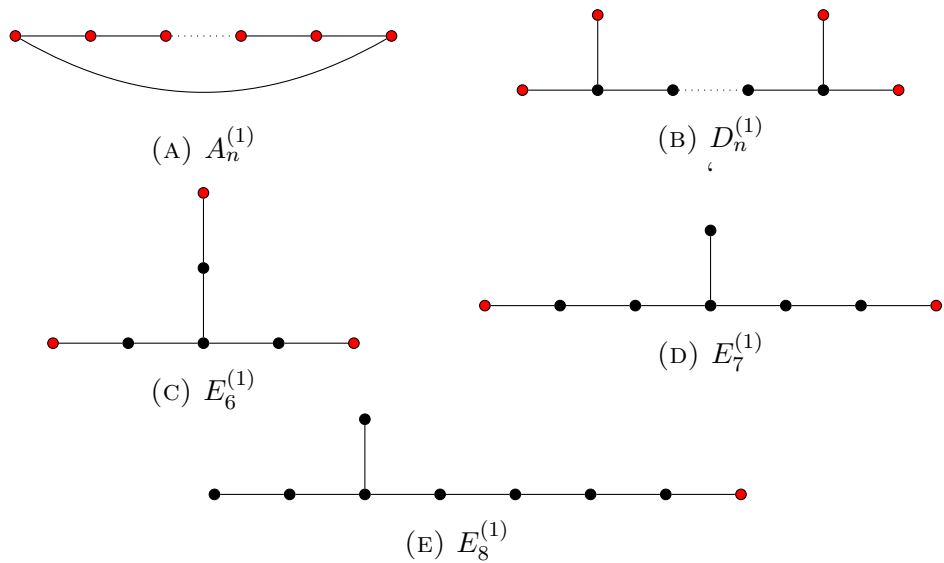 
%%Affine Folded
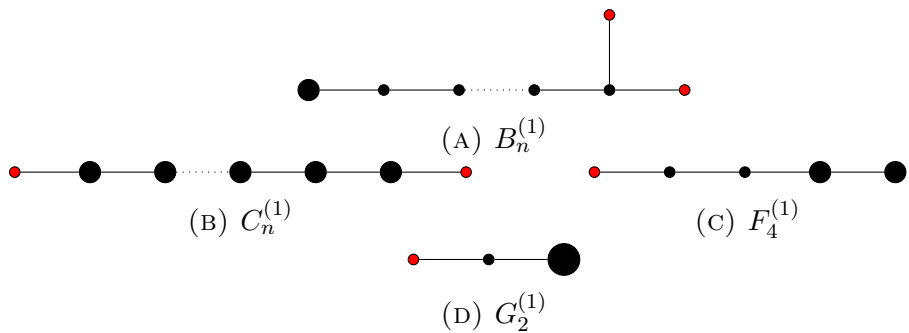
\begin{figure}[hb]
  \centering
  \begin{subfigure}{.4\textwidth}
        \centering
        \begin{tikzpicture}
        \node[fat2] (0) [] {};
        \node[base] (1) [right of = 0] {};
        \node[base] (2) [right of = 1] {};
        \node[base] (3) [right of = 2] {};
        \node[base] (4) [right of = 3] {};
        \node[affine] (5) [right of = 4] {};
        \node[affine] (6) [above of = 4] {};
        \path[-] (0) edge [left] node {} (1); 
        \path[-] (1) edge [left] node {} (2);
        \path[-] (2) edge [left,dotted] node {} (3); 
        \path[-] (3) edge [left] node {} (4); 
        \path[-] (4) edge [left] node {} (5);
        \path[-] (4) edge [] node {} (6);
    \end{tikzpicture}
  \caption{$\Affine{B}_n$}
  \end{subfigure}
    \hspace*{2pc}
    \begin{subfigure}{.47\textwidth}
        \centering
         \begin{tikzpicture}
            \node[affine] (0) [] {};
            \node[fat2] (1) [right of = 0] {};
            \node[fat2] (2) [right of = 1] {};
            \node[fat2] (3) [right of = 2] {};
            \node[fat2] (4) [right of = 3] {};
            \node[fat2] (5) [right of = 4] {};
            \node[affine] (6) [right of = 5] {};
            \path[-] (0) edge [left] node {} (1); 
            \path[-] (1) edge [left] node {} (2);
            \path[-] (2) edge [left,dotted] node {} (3); 
            \path[-] (3) edge [left] node {} (4); 
            \path[-] (4) edge [left] node {} (5);
            \path[-] (5) edge [left] node {} (6);
        \end{tikzpicture}
        \caption{$\Affine{C}_n$}
  \end{subfigure}

    \begin{subfigure}{.4\textwidth}
        \centering
        \begin{tikzpicture}
        \node[base] (0) [] {};
        \node[base] (1) [right of = 0] {};
        \node[fat2] (2) [right of = 1] {};
        \node[fat2] (3) [right of = 2] {};
        \node[affine] (4) [left of = 0] {};
        \path[-] (0) edge [] node {} (4);
        \path[-] (0) edge [left] node {} (1);
        \path[-] (1) edge [left] node {} (2);
        \path[-] (2) edge [left] node {} (3);
        \end{tikzpicture}
        \caption{$\Affine{F}_4$}
    \end{subfigure}
    \begin{subfigure}{.4\textwidth}
        \centering
        \begin{tikzpicture}
        \node[affine] (2) [left of = 0] {};
        \node[base] (0) [] {};
        \node[fat3] (1) [right of = 0] {};
        \path[-] (2) edge [] node {} (0);
        \path[-] (0) edge [left] node {} (1);
        \end{tikzpicture}
        \caption{$\Affine{G}_2$}
    \end{subfigure}
    \caption{Folded Affine Dynkin Diagrams}
    \label{fig:DynkinAffineFolded}
  \end{figure}
%%Affine Twisted  
  \begin{figure}[hb]
  \centering
  \begin{subfigure}{.47\textwidth}
        \centering
        \begin{tikzpicture}
        \node[base] (0) [] {};
        \node[fat2] (1) [right of = 0] {};
        \node[fat2] (2) [right of = 1] {};
        \node[fat2] (3) [right of = 2] {};
        \node[fat2] (4) [right of = 3] {};
        \node[affine2] (5) [right of = 4] {};
        \node[affine2] (6) [above of = 4] {};
        \path[-] (0) edge [left] node {} (1); 
        \path[-] (1) edge [left] node {} (2);
        \path[-] (2) edge [left,dotted] node {} (3); 
        \path[-] (3) edge [left] node {} (4); 
        \path[-] (4) edge [left] node {} (5);
        \path[-] (4) edge [] node {} (6);
    \end{tikzpicture}
  \caption{Twisted $\TwistedAffine{C}{2}_n$}
  \end{subfigure}
  \hspace*{1pc}
  \begin{subfigure}{.48\textwidth}
        \centering
         \begin{tikzpicture}
            \node[affine2] (0) [] {};
            \node[base] (1) [right of = 0] {};
            \node[base] (2) [right of = 1] {};
            \node[base] (3) [right of = 2] {};
            \node[base] (4) [right of = 3] {};
            \node[base] (5) [right of = 4] {};
            \node[affine2] (6) [right of = 5] {};
            \path[-] (0) edge [left] node {} (1); 
            \path[-] (1) edge [left] node {} (2);
            \path[-] (2) edge [left,dotted] node {} (3); 
            \path[-] (3) edge [left] node {} (4); 
            \path[-] (4) edge [left] node {} (5);
            \path[-] (5) edge [left] node {} (6);
        \end{tikzpicture}
        \caption{Twisted $\TwistedAffine{B}{2}_n$}
  \end{subfigure}
  \begin{subfigure}{.48\textwidth}
        \centering
         \begin{tikzpicture}
            \node[affine] (0) [] {};
            \node[fat2] (1) [right of = 0] {};
            \node[fat2] (2) [right of = 1] {};
            \node[fat2] (3) [right of = 2] {};
            \node[fat2] (4) [right of = 3] {};
            \node[fat2] (5) [right of = 4] {};
            \node[affine4] (6) [right of = 5] {};
            \path[-] (0) edge [left] node {} (1); 
            \path[-] (1) edge [left] node {} (2);
            \path[-] (2) edge [left,dotted] node {} (3); 
            \path[-] (3) edge [left] node {} (4); 
            \path[-] (4) edge [left] node {} (5);
            \path[-] (5) edge [left] node {} (6);
        \end{tikzpicture}
        \caption{Twisted $\TwistedAffine{BC}{4}_n$}
  \end{subfigure}
    \hspace*{1pc}
  \begin{subfigure}{.4\textwidth}
        \centering
        \begin{tikzpicture}
        \node[fat2] (0) [] {};
        \node[fat2] (1) [right of = 0] {};
        \node[base] (2) [right of = 1] {};
        \node[base] (3) [right of = 2] {};
        \node[affine2] (4) [left of = 0] {};
        \path[-] (0) edge [] node {} (4);
        \path[-] (0) edge [left] node {} (1);
        \path[-] (1) edge [left] node {} (2);
        \path[-] (2) edge [left] node {} (3);
        \end{tikzpicture}
        \caption{Twisted $\TwistedAffine{F}{2}_4$}
    \end{subfigure}
  \begin{subfigure}{.4\textwidth}
        \centering
        \begin{tikzpicture}
        \node[affine3] (2) [left of = 0] {};
        \node[fat3] (0) [] {};
        \node[base] (1) [right of = 0] {};
        \path[-] (2) edge [] node {} (0);
        \path[-] (0) edge [left] node {} (1);
        \end{tikzpicture}
        \caption{Twisted $\TwistedAffine{G}{3}_2$}
    \end{subfigure}
  \caption{Twisted Affine Dynkin Diagrams}
    \label{fig:DynkinAffineTwisted}
\end{figure}
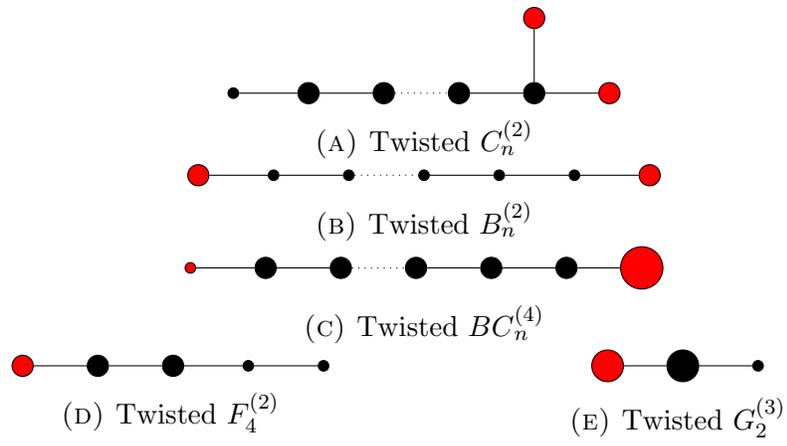
  \clearpage 
  Similarly each double extended diagram can be formed by adding two nodes to a finite diagram or one node to the affine diagram. Each red node in \Cref{fig:DynkinDouble,fig:DynkinDoubleFolded} is a possible extension of the corresponding affine Dynkin diagram.
%%Doubly Extended
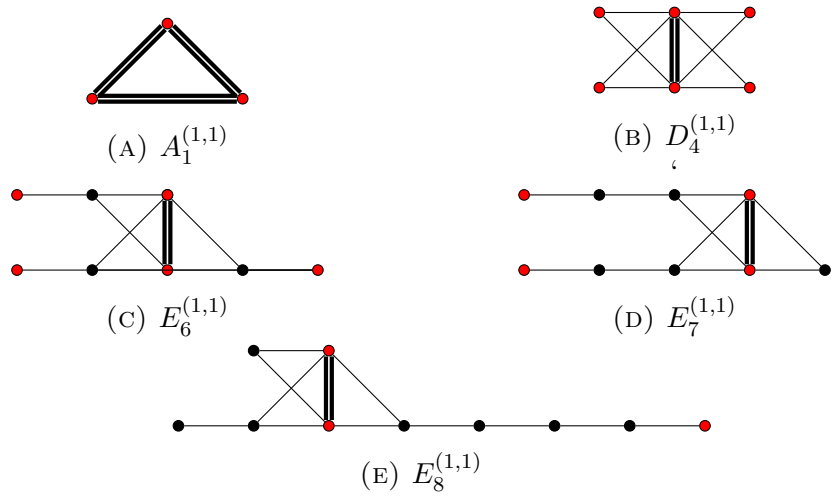
\begin{figure}[hb]
    \centering
    \begin{subfigure}{.4\textwidth}
    \centering
     \begin{tikzpicture}
        \node[affine] (0) [] {};
        \node[affine] (1) [right of = 0,above of = 0] {};
        \node[affine] (2) [right of = 1,below of = 1] {};
        \path[-] (0) edge [double,ultra thick] node {} (1); 
        \path[-] (1) edge [double,ultra thick] node {} (2);
        \path[-] (2) edge [double,ultra thick] node {} (0); 
    \end{tikzpicture}
    \caption{$\db{A}{1}$}
    \end{subfigure}
  \begin{subfigure}{.4\textwidth}
      \centering
        \begin{tikzpicture}
        \node[affine] (0) [] {};
        \node[affine] (2) [right of = 0] {};    
        \node[affine] (1) [above of = 0] {};
        \node[affine] (3) [right of = 2] {};
        \node[affine] (4) [above of = 3] {};
        \node[affine] (5) [above of = 2] {};
        \path[-] (0) edge [] node {} (2); 
        \path[-] (1) edge [] node {} (2);
        \path[-] (2) edge [] node {} (3); 
        \path[-] (2) edge [] node {} (4); 
        \path[-] (2) edge [double,ultra thick] node {} (5);
        \path[-] (0) edge [] node {} (5); 
        \path[-] (1) edge [] node {} (5);
        \path[-] (5) edge [] node {} (3); 
        \path[-] (5) edge [] node {} (4);
        \end{tikzpicture}
        \caption{$\db{D}{4}$}
  `\end{subfigure}\\
    \begin{subfigure}{.4\textwidth}
        \centering
            \begin{tikzpicture}
            \node[affine] (0) [] {};
            \node[base] (2) [right of = 0] {}; 
            \node[affine] (3) [right of = 2] {};
            \node[base] (1) [above of = 2] {};
            \node[affine] (7) [left of = 1] {};
            \node[base] (4) [right of = 3] {};
            \node[affine] (5) [right of = 4] {};
            \node[base] (6) [above of = 3] {};
            \node[affine] (8) [above of = 3] {};
            \path[-] (0) edge [] node {} (2); 
            \path[-] (7) edge [] node {} (1);
            \path[-] (1) edge [] node {} (3);
            \path[-] (2) edge [] node {} (3); 
            \path[-] (3) edge [] node {} (4); 
            \path[-] (4) edge [] node {} (5);

            \path[-] (2) edge [] node {} (5); 
            \path[-] (5) edge [] node {} (4);
            \path[-] (8) edge [double,ultra thick] node {} (3);
            \path[-] (1) edge [] node {} (8);
            \path[-] (2) edge [] node {} (8); 
            \path[-] (8) edge [] node {} (4);
            \end{tikzpicture}
        \caption{$\db{E}{6}$}
    \end{subfigure}
    \begin{subfigure}{.4\textwidth}
      \centering
        \begin{tikzpicture}
        \node[affine] (0) [] {};
        \node[base] (1) [right of = 0] {};    
        \node[base] (2) [right of = 1] {};
        \node[affine] (3) [above of = 0] {};
        \node[base] (4) [right of = 3] {};
        \node[base] (5) [right of = 4] {};
        \node[affine] (6) [right of = 2] {};
        \node[affine] (7) [above of = 6] {};
        \node[base] (8) [right of = 6] {};
        
        \path[-] (0) edge [] node {} (1);
        \path[-] (1) edge [] node {} (2);
        \path[-] (2) edge [] node {} (6);
        \path[-] (2) edge [] node {} (7);
        
        \path[-] (3) edge [] node {} (4);
        \path[-] (4) edge [] node {} (5);
        \path[-] (5) edge [] node {} (6);
        \path[-] (5) edge [] node {} (7);
        
        \path[-] (6) edge [] node {} (8);
        \path[-] (7) edge [] node {} (8);
        
        \path[-] (6) edge [double,ultra thick] node {} (7);
        \end{tikzpicture}
        \caption{$\db{E}{7}$}
    \end{subfigure}
  \begin{subfigure}{.7\textwidth}
      \centering
        \begin{tikzpicture}
        \node[base] (0) [] {};
        \node[base] (2) [right of = 0] {};    
        \node[affine] (3) [right of = 2] {};
        \node[base] (1) [above of = 2] {};
        \node[base] (4) [right of = 3] {};
        \node[base] (5) [right of = 4] {};
        \node[base] (6) [right of = 5] {};
        \node[base] (7) [right of = 6] {};
        \node[affine] (8) [right of = 7] {};
        \node[affine] (9) [above of = 3] {};
        \path[-] (0) edge [] node {} (2); 
        \path[-] (1) edge [] node {} (3);
        \path[-] (2) edge [] node {} (3); 
        \path[-] (3) edge [] node {} (4); 
        \path[-] (4) edge [] node {} (5);
        \path[-] (5) edge [] node {} (6);
        \path[-] (6) edge [] node {} (7);
        \path[-] (7) edge [] node {} (8);
        \path[-] (3) edge [double,ultra thick] node {} (9);
        \path[-] (1) edge [] node {} (9);
        \path[-] (2) edge [] node {} (9); 
        \path[-] (9) edge [] node {} (4);
        \end{tikzpicture}
     \caption{$\db{E}{8}$}
     \end{subfigure}
  \caption{Simply Laced Doubly Extended Dynkin Diagrams}
    \label{fig:DynkinDouble}
  \end{figure}    
%Folded Doubly Extended
\begin{figure}[hb]
\centering
  \begin{subfigure}{.3\textwidth}
      \centering
        \begin{tikzpicture}
        \node[fat2] (0) [] {};
        \node[affine] (2) [right of = 0] {};    
        \node[affine] (3) [right of = 2] {};
        \node[affine] (4) [above of = 3] {};
        \node[affine] (5) [above of = 2] {};
        \path[-] (0) edge [] node {} (2); 
        \path[-] (2) edge [] node {} (3); 
        \path[-] (2) edge [] node {} (4); 
        \path[-] (2) edge [double,ultra thick] node {} (5);
        \path[-] (0) edge [] node {} (5); 
        \path[-] (5) edge [] node {} (3); 
        \path[-] (5) edge [] node {} (4);
        \end{tikzpicture}
     \caption{$\db{B}{3}$}
     \end{subfigure}
  \begin{subfigure}{.3\textwidth}
      \centering
        \begin{tikzpicture}
        \node[fat2] (0) [] {};
        \node[affine] (2) [right of = 0] {};    
        \node[fat2] (3) [right of = 2] {};
        \node[affine] (5) [above of = 2] {};
        \path[-] (0) edge [] node {} (2); 
        \path[-] (2) edge [] node {} (3); 
        \path[-] (2) edge [double,ultra thick] node {} (5);
        \path[-] (0) edge [] node {} (5); 
        \path[-] (5) edge [] node {} (3); 
        \end{tikzpicture}
     \caption{$\dbf{B}{2}{2}{1}$}
     \end{subfigure}
  \begin{subfigure}{.3\textwidth}
      \centering
        \begin{tikzpicture}
        \node[fat4] (0) [] {};
        \node[affine] (2) [right of = 0] {};    
        \node[affine] (5) [above of = 2] {};
        \path[-] (0) edge [] node {} (2);
        \path[-] (2) edge [double,ultra thick] node {} (5);
        \path[-] (0) edge [] node {} (5); 
        \end{tikzpicture}
     \caption{$\dbf{BC}{1}{4}{1}$}
     \end{subfigure}\\

  \begin{subfigure}{.3\textwidth}
      \centering
        \begin{tikzpicture}
        \node[base] (0) [] {};
        \node[affine2] (2) [right of = 0] {};    
        \node[affine2] (3) [right of = 2] {};
        \node[affine2] (4) [above of = 3] {};
        \node[affine2] (5) [above of = 2] {};
        \path[-] (0) edge [] node {} (2); 
        \path[-] (2) edge [] node {} (3); 
        \path[-] (2) edge [] node {} (4); 
        \path[-] (2) edge [double,ultra thick] node {} (5);
        \path[-] (0) edge [] node {} (5); 
        \path[-] (5) edge [] node {} (3); 
        \path[-] (5) edge [] node {} (4);
        \end{tikzpicture}
     \caption{$\dbf{C}{3}{2}{2}$}
     \end{subfigure}
  \begin{subfigure}{.3\textwidth}
      \centering
        \begin{tikzpicture}
        \node[base] (0) [] {};
        \node[affine2] (2) [right of = 0] {};    
        \node[fat4] (3) [right of = 2] {};
        \node[affine2] (5) [above of = 2] {};
        \path[-] (0) edge [] node {} (2); 
        \path[-] (2) edge [] node {} (3); 
        \path[-] (2) edge [double,ultra thick] node {} (5);
        \path[-] (0) edge [] node {} (5); 
        \path[-] (5) edge [] node {} (3); 
        \end{tikzpicture}
     \caption{$\dbf{BC}{2}{4}{2}$}
     \end{subfigure}
  \begin{subfigure}{.3\textwidth}
      \centering
        \begin{tikzpicture}
        \node[base] (0) [] {};
        \node[affine4] (2) [right of = 0] {};    
        \node[affine4] (5) [above of = 2] {};
        \path[-] (0) edge [] node {} (2); 
        \path[-] (2) edge [double,ultra thick] node {} (5);
        \path[-] (0) edge [] node {} (5); 
        \end{tikzpicture}
     \caption{$\dbf{BC}{1}{4}{4}$}
     \end{subfigure}\\
     
    \begin{subfigure}{.3\textwidth}
      \centering
        \begin{tikzpicture}
        \node[fat3] (0) [] {};
        \node[affine] (2) [left of = 0] {};    
        \node[affine] (5) [above of = 2] {};
        \node[affine] (1) [left of = 2] {};
        \path[-] (0) edge [] node {} (2);
        \path[-] (2) edge [double,ultra thick] node {} (5);
        \path[-] (0) edge [] node {} (5); 
        \path[-] (1) edge [] node {} (2);
        \path[-] (1) edge [] node {} (5);
        \end{tikzpicture}
     \caption{$\dbf{G}{2}{1}{1}$}
     \end{subfigure}
     \begin{subfigure}{.4\textwidth}
        \centering
            \begin{tikzpicture}
            \node[fat2] (0) [] {};
            \node[fat2] (2) [right of = 0] {}; 
            \node[affine] (3) [right of = 2] {};
            \node[base] (4) [right of = 3] {};
            \node[affine] (5) [right of = 4] {};
            \node[affine] (8) [above of = 3] {};
            \path[-] (0) edge [] node {} (2); 
            \path[-] (2) edge [] node {} (3); 
            \path[-] (3) edge [] node {} (4); 
            \path[-] (4) edge [] node {} (5);

            \path[-] (2) edge [] node {} (5); 
            \path[-] (5) edge [] node {} (4);
            \path[-] (8) edge [double,ultra thick] node {} (3);
            \path[-] (2) edge [] node {} (8); 
            \path[-] (8) edge [] node {} (4);
            \end{tikzpicture}
        \caption{$\db{F}{4}$}
    \end{subfigure}\\
       \begin{subfigure}{.3\textwidth}
      \centering
        \begin{tikzpicture}
        \node[base] (0) [] {};
        \node[affine3] (2) [left of = 0] {};    
        \node[affine3] (5) [above of = 2] {};
        \node[affine3] (1) [left of = 2] {};
        \path[-] (0) edge [] node {} (2);
        \path[-] (2) edge [double,ultra thick] node {} (5);
        \path[-] (0) edge [] node {} (5); 
        \path[-] (1) edge [] node {} (2);
        \path[-] (1) edge [] node {} (5);
        \end{tikzpicture}
     \caption{$\dbf{G}{2}{3}{3}$}
     \end{subfigure}
     \begin{subfigure}{.4\textwidth}
        \centering
            \begin{tikzpicture}
            \node[base] (0) [] {};
            \node[base] (2) [right of = 0] {}; 
            \node[affine2] (3) [right of = 2] {};
            \node[fat2] (4) [right of = 3] {};
            \node[affine2] (5) [right of = 4] {};
            \node[affine2] (8) [above of = 3] {};
            \path[-] (0) edge [] node {} (2); 
            \path[-] (2) edge [] node {} (3); 
            \path[-] (3) edge [] node {} (4); 
            \path[-] (4) edge [] node {} (5);

            \path[-] (5) edge [] node {} (4);
            \path[-] (8) edge [double,ultra thick] node {} (3);
            \path[-] (2) edge [] node {} (8); 
            \path[-] (8) edge [] node {} (4);
            \end{tikzpicture}
        \caption{$\dbf{F}{4}{2}{2}$}
    \end{subfigure}\\
         \begin{subfigure}{.3\textwidth}
        \centering
            \begin{tikzpicture}
            \node[fat3] (0) [] {};
            \node[fat3] (2) [right of = 0] {}; 
            \node[affine] (3) [right of = 2] {};
            \node[affine] (8) [above of = 3] {};
            \path[-] (0) edge [] node {} (2); 
            \path[-] (2) edge [] node {} (3); 
            \path[-] (8) edge [double,ultra thick] node {} (3);
            \path[-] (2) edge [] node {} (8); 
            \end{tikzpicture}
        \caption{$\dbf{G}{2}{3}{1}$}
    \end{subfigure}
    \begin{subfigure}{.4\textwidth}
        \centering
            \begin{tikzpicture}
            \node[fat2] (0) [] {};
            \node[fat2] (2) [right of = 0] {}; 
            \node[affine] (3) [right of = 2] {};
            \node[base] (4) [right of = 3] {};
            \node[fat2] (5) [left of = 0] {};
            \node[affine] (8) [above of = 3] {};
            \path[-] (0) edge [] node {} (2); 
            \path[-] (2) edge [] node {} (3); 
            \path[-] (3) edge [] node {} (4); 

            \path[-] (5) edge [] node {} (0);
            \path[-] (8) edge [double,ultra thick] node {} (3);
            \path[-] (2) edge [] node {} (8); 
            \path[-] (8) edge [] node {} (4);
            \end{tikzpicture}
        \caption{$\dbf{F}{4}{2}{1}$}
    \end{subfigure}
     \caption{Folded Doubly Extended Dynkin Diagrams}  
     \label{fig:DynkinDoubleFolded}

\end{figure}
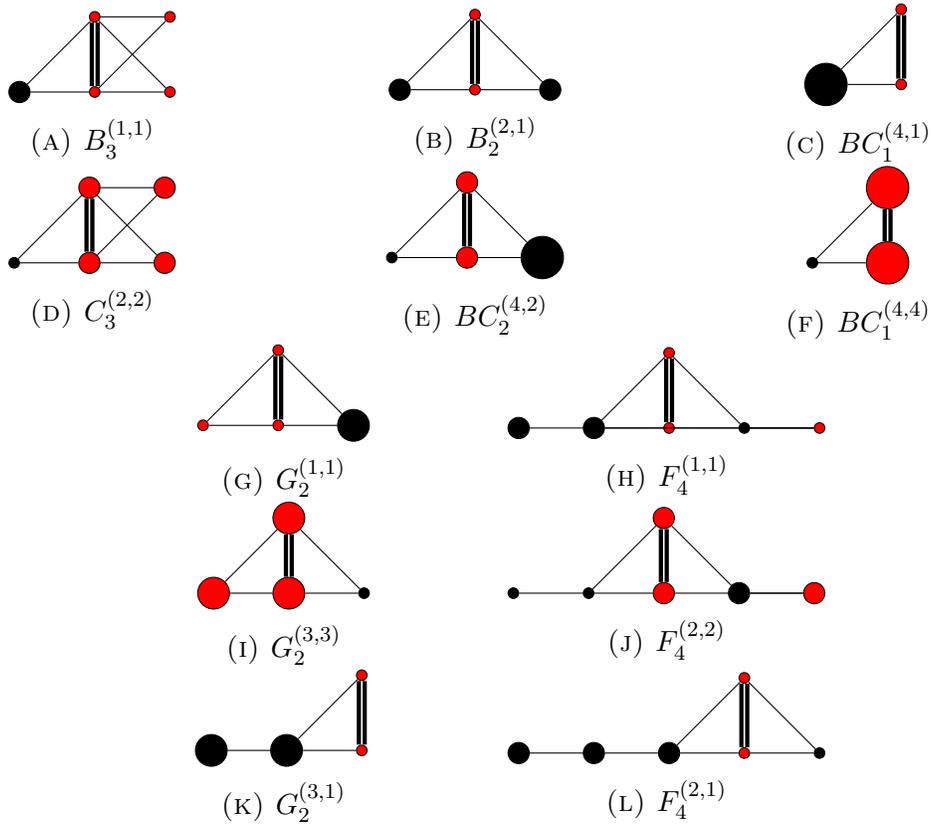

\begin{figure}[hb]
    \centering

\begin{tikzcd}
                                                                                                                                   &  & \dbf{B}{3}{1}{1} \arrow[lld, dashed] \arrow[rr, Leftrightarrow] \arrow[dd]                                    &  & \dbf{C}{3}{2}{2}  \arrow[dd]                                                            &  &                                                                                                                             \\
\db{A}{1} \arrow[Leftrightarrow, loop, distance=2em, in=215, out=145]                                                                        &  &                                                                                              &  &                                                                           &  & \db{E}{7} \arrow[llu, dashed] \arrow[lld, dashed] \arrow[llddd, dashed] \arrow[dd] \arrow[Leftrightarrow, loop, distance=2em, in=35, out=325] \\
                                                                                                                                   &  &  \dbf{B}{2}{2}{1} \arrow[llu, dashed] \arrow[dd] \arrow[Leftrightarrow, loop, distance=2em, in=35, out=325] &  & \dbf{BC}{2}{4}{2} \arrow[dd] \arrow[Leftrightarrow, loop, distance=2em, in=215, out=145] &  &                                                                                                                             \\
\db{D}{4} \arrow[uu, dashed] \arrow[rruuu] \arrow[rru] \arrow[rrd] \arrow[rrddd] \arrow[Leftrightarrow, loop, distance=2em, in=215, out=145] &  &                                                                                              &  &                                                                           &  & \dbf{F}{4}{2}{1} \arrow[lld, dashed] \arrow[llu, dashed] \arrow[Leftrightarrow, loop, distance=2em, in=35, out=325]                        \\
                                                                                                                                   &  & \dbf{BC}{1}{4}{1} \arrow[rr, Leftrightarrow]                                                                  &  & \dbf{BC}{1}{4}{4}                                                                    &  &                                                                                                                             \\
\dbf{G}{2}{3}{1} \arrow[Leftrightarrow, loop, distance=2em, in=215, out=145]                                                                      &  &                                                                                              &  &                                                                           &  &                                                                                                                             \\
                                                                                                                                   &  & \dbf{G}{2}{1}{1} \arrow[rr, Leftrightarrow]                                                                   &  & \dbf{G}{2}{3}{3}                                                                    &  &                                                                                                                             \\
\db{E}{6} \arrow[rrd] \arrow[uu] \arrow[Leftrightarrow, loop, distance=2em, in=215, out=145]                                         &  &                                                                                              &  &                                                                           &  & \db{E}{8} \arrow[llu, dashed] \arrow[lld, dashed] \arrow[Leftrightarrow, loop, distance=2em, in=35, out=325]                          \\
                                                                                                                                   &  & \dbf{F}{4}{1}{1} \arrow[rr, Leftrightarrow]                                                                               &  & \dbf{F}{4}{2}{2}                                                                     &  &                                                                                                                            
\end{tikzcd}
    \caption{The double-extended family tree. The solid arrows represent folding of $T_{\vec{n},\vec{w}}$ quivers, dashed arrows are nonstandard foldings, and the double arrows represent Langlands-duality}
    \label{fig:double_extended_family}
\end{figure}
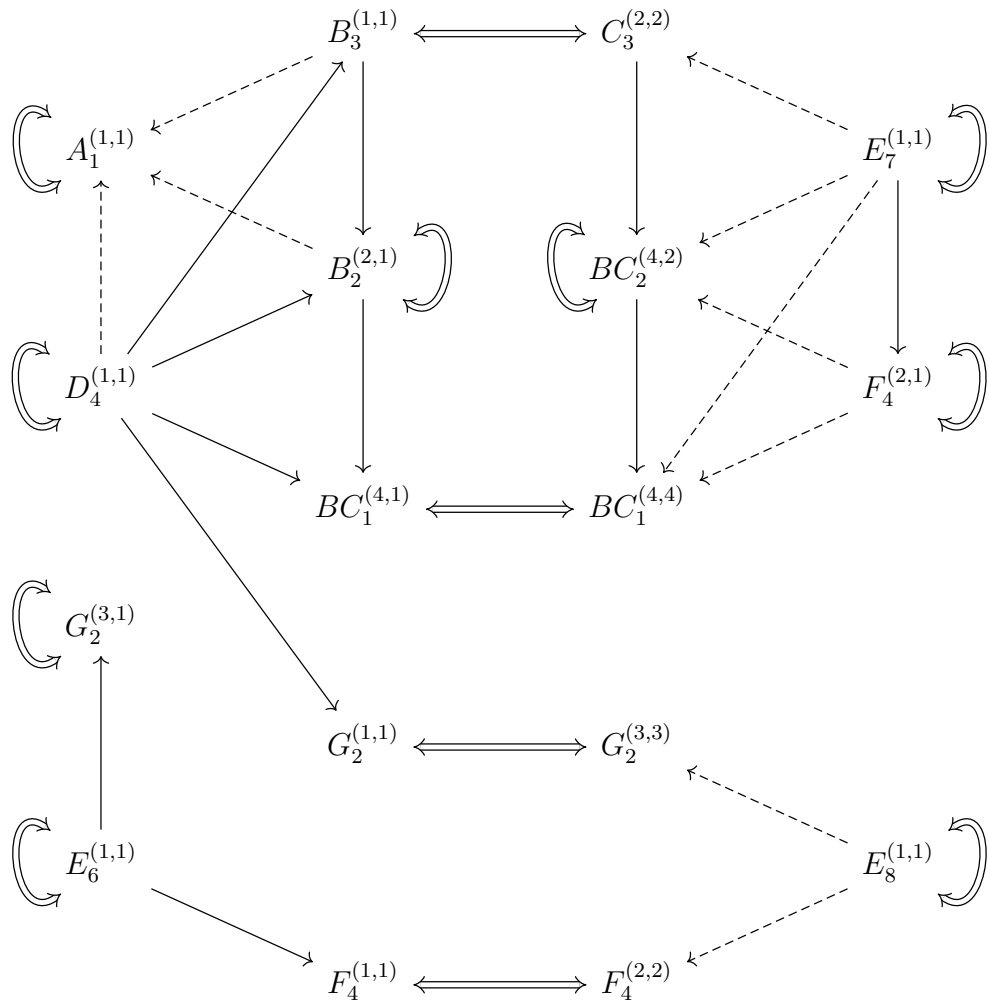
\clearpage

\section{Computations Using Marked Surfaces}

\subsection{Cluster combinatorics from surfaces}\label{sec:SurfaceClusters}
In this section, we will briefly review cluster algebras associated to surfaces. For a complete description see \cite{FST:Triangulated_surfaces} or Section 3 of \cite{williams_ClusterAlgebrasAnIntroduction}. 

\begin{definition}
A \emph{marked surface}, $S_{g,b,p,n}$ is an orientable surface of genus $g$ with $b$ boundary components, $p$ punctures and $n$ marked points on the boundary. We always require that each boundary component has at least one marked point. An \emph{arc} on a marked surface $S$ is a (non-contractible) isotopy class of curves between marked points or punctures on $S$. An \emph{ideal triangulation} of a marked surface is a maximal collection of non-crossing arcs on $S$. 
\end{definition}
Let $S$ be a marked surface. Given  an ideal triangulation $\Delta$ of $S$, we associate a quiver, $Q_\Delta$, as follows: For each arc $e \in \Delta$ we add a node $N_e$ and for each triangle $t \in \Delta$ we add a clockwise oriented cycle of arrows between the nodes associated with the arcs of $t$. In the situation where we have arrows between two nodes in opposite directions, we cancel them. The nodes associated to boundary edges are frozen. There are  $-3\chi(S)+2n $ total nodes and $n$ frozen nodes.

There is a correspondence between the cluster variables of $\ClusterAlgebra{S}:=\ClusterAlgebra{Q_\Delta}$ and the arcs on $S$.
Essentially, cluster variables correspond to arcs, clusters to triangulations and mutation corresponds to a ``flip'' of arcs.  Any two triangulations of a surface can be reached from each other by a sequence of flips. Therefore the quivers associated to two different triangulations of $S$ are in the same mutation class. 

There is one minor complication when $S$ has punctures. In this case it may be possible to have a ``self folded'' triangle in an ideal triangulation of $S$ see \Cref{fig:digon_folded}. In this case, the construction mentioned above does not produce the correct quiver. However, we can always find a triangulation of $S$ with no self folded triangles, and use this to construct a quiver associated with the triangulation.  

Then mutation of nodes in $Q_\Delta$ corresponds to a ``flip'' or ``Whitehead move'' in $\Delta$ at the corresponding arc. Again, there is a caveat to this when $S$ has punctures. The interior arc of a self folded triangle cannot be flipped, but the corresponding node in the quiver can be mutated. This is addressed in \cite{FST:Triangulated_surfaces} by the addition of ``tagged'' arcs. Essentially, we replace the outside arc of a self folded triangulation with a tagged arc as shown in \Cref{fig:digon_flips}. There is then a rule for flipping tagged arcs which agrees with the mutation rule for quivers. With this addition, we may always flip any arc and this always agrees with mutation of corresponding quivers.  We do not need the details of this in general. 

\begin{figure}
    \centering
    \begin{subfigure}{.49\textwidth}
    \centering
    \includegraphics[scale=.8]{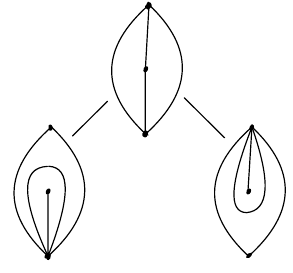}
    \caption{Arcs in a punctured digon.}
    \label{fig:digon_folded}
    \end{subfigure}
    \begin{subfigure}{.49\textwidth}
    \centering
    \includegraphics[scale=.8]{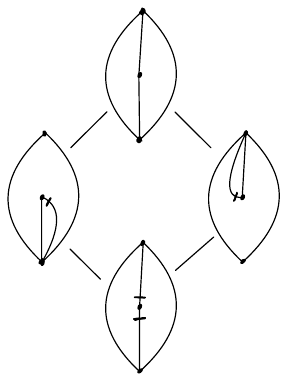}
    \caption{The tagged arc flip graph.}
    \label{fig:digon_tagged}
    \end{subfigure}
    \caption{Untagged vs tagged arcs in a punctured digon.}
    \label{fig:digon_flips}
\end{figure}

\begin{remark}\label{rem:DoubleEdgeTriangles}
A quiver associated to a surface can only have a double edge if the triangulation contains one of the two sub-triangulations in \Cref{fig:doubleEdgeTriangles}.
\end{remark}
\begin{figure}
    \centering
    \begin{subfigure}{.4\textwidth}
           \centering
           \includegraphics[width=\textwidth]{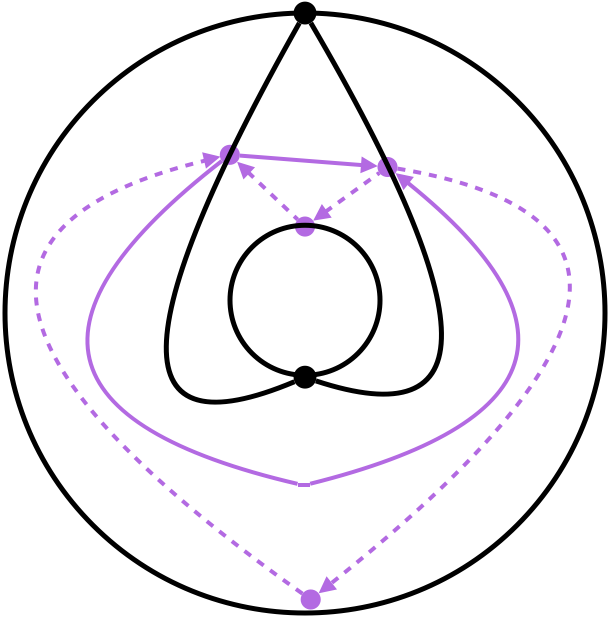}
    \end{subfigure}
        \begin{subfigure}{.4\textwidth}
           \centering
           \includegraphics[width=\textwidth]{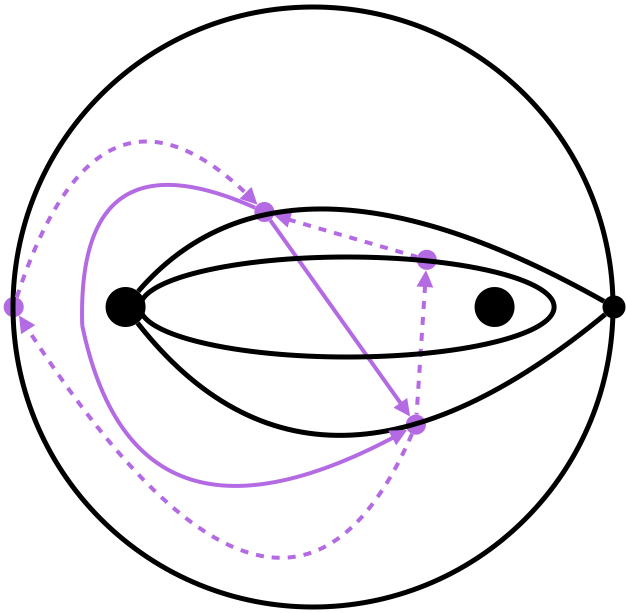}
    \end{subfigure}
    \caption{The only sub-triangulations that produce double edge quivers.}
    \label{fig:doubleEdgeTriangles}
\end{figure}

We can define an action of the mapping class group, $\MappingClassGroup{S}$, on the triangulations of $S$ and hence identify the mapping class group as a subgroup of the cluster modular group, $\Gamma_S$, of our cluster algebra $\ClusterAlgebra{S}$. We give an explicit construction of this subgroup here as a nice example of our notation. We refer to \cite{Margalit:PMCG} Section 2 for computations involving the mapping class group of selected surfaces.

\begin{lemma}\label{lem:TwistGroup}
Let $S$ be an annulus with $n$ marked points on the inner boundary component and $1$ marked point on the outer boundary. The the twist $\tau $ (\Cref{def:tau}) corresponds to rotating the inner boundary component $\frac{2\pi}{n}$ radians and $\gamma$ corresponds to a full Dehn twist and thus $\gamma = \tau^n$.
\end{lemma}
\begin{proof}
To analyze $\tau$ we break the mutation sequence into two pieces $[i_{\text{odd}}i_{\text{even}}]$, $[i_2, N_\infty, N_1]$. On the annulus, the arc associated with node $i_2$ begins and ends at $v_{1}$. Thus $[i_{\text{odd}}i_{\text{even}}]$ is a ``sinks then sources'' sequence inside an $n-$gon. This rotates the zig-zag triangulation clockwise one tick  so the outermost arc goes from $v_2$ clockwise around to $v_{1}$. Then treating this arc as an arc of the inner boundary component reduces puts us exactly in the situation of a $T_{(2),(1)}$ quiver. 

It is then a simple computation to see that the mutation path $[i_2,N_\infty, N_1]$ returns to a quiver isomorphic to the original but with the self loop around $v_2$ instead of $v_1$. Note that $N_1$ is now the self loop and $i_2$ and $N_\infty$ are the source and sink of the double edge respectively, justifying the permutation $(i_2,N_1,N_\infty)$. See \Cref{fig:singleTwist} for an example of a tail with length $4$.

Therefore each application of $\tau$ moves one tick clockwise around the inner boundary component.
Therefore $n$ twists returns to $v_1$ having made a full clockwise twist about the inner boundary component. Furthermore, the self loop at $v_1$, treated as the edge of the boundary component, always separates $N_1$ and $N_\infty$ from the rest of the tail. 

So it suffices to analyze $\gamma$ on the annulus with one marked point on each boundary component. Then it is clear applying $\gamma$ is equivalent twisting once clockwise around the inner boundary component and so is equal to $\tau^{n}$.
\end{proof}

\Cref{fig:singleTwist} shows the explicit action of twisting about a tail on the surface representation of the cluster algebra.
\begin{figure}[hb]
    \centering
    \includegraphics[width=.3\textwidth]{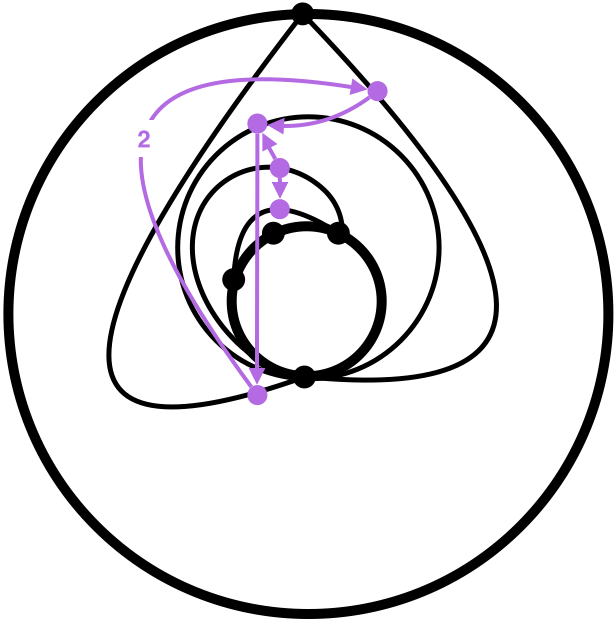}\hspace*{1pc}
    \includegraphics[width=.3\textwidth]{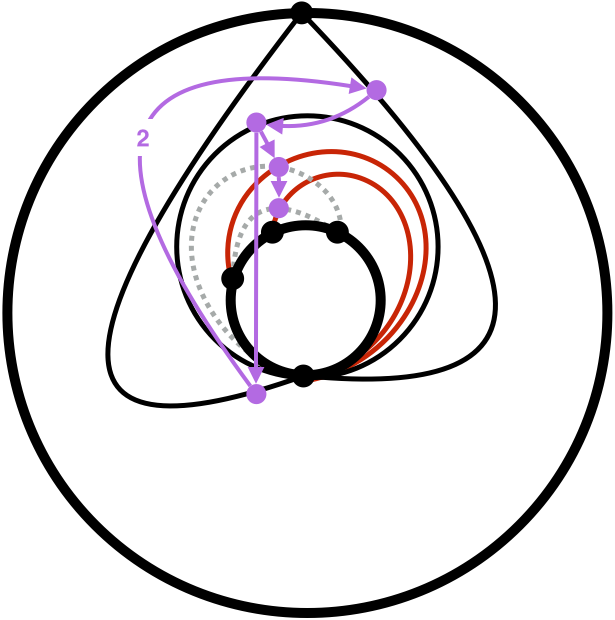}\hspace*{1pc}
    \includegraphics[width=.3\textwidth]{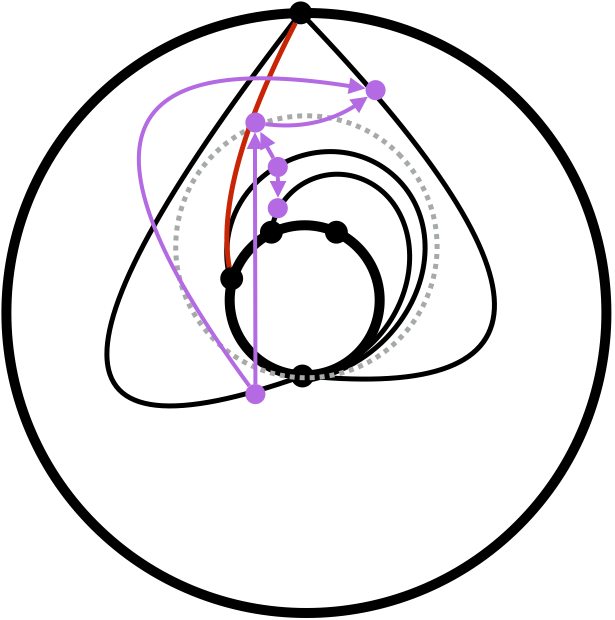}\\
    \includegraphics[width=.3\textwidth]{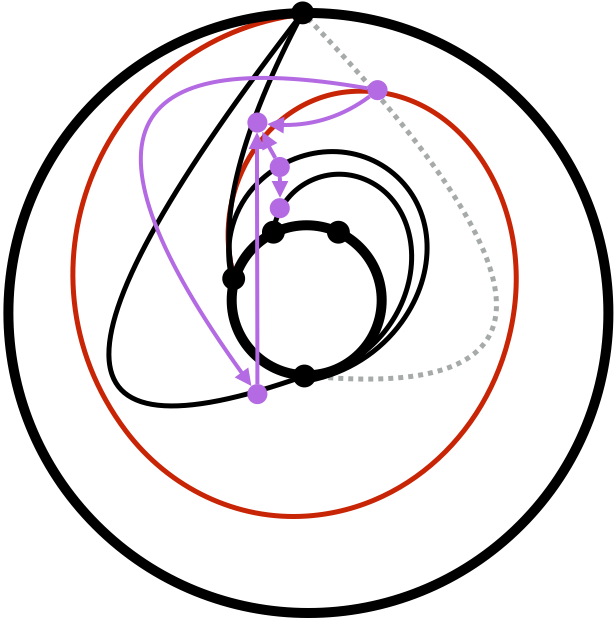}\hspace*{1pc}
    \includegraphics[width=.3\textwidth]{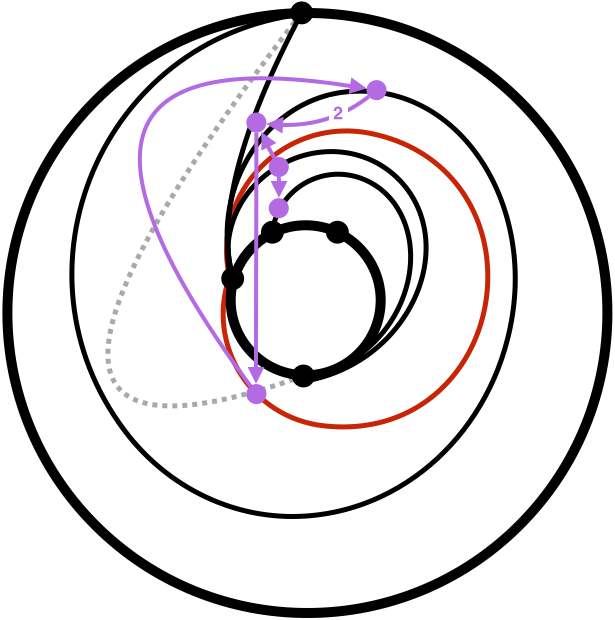}
    \caption{Application of single twist for a tail of length 4. The result is shown after $[i_{\text{odd}}i_{\text{even}}]$, $i_2$, $N_\infty$, and then $N_1$. At each stage the dashed gray edges are replaced with the red edges}
    \label{fig:singleTwist}
\end{figure}

\subsection{Proofs of Affine Surface Theorems}\label{sec:AffineProofs}
In order to prove \Cref{thm:AffineType} we need to carefully analyze the triangulations of both the annulus and the twice punctured disc. 
\begin{definition}\label{def:SurfaceArcClasses}
There are three classes of arcs on an annulus. \emph{Crossing arcs} connect two marked points on different boundary components. \emph{Boundary arcs} connect two marked points on the same boundary component.  A \emph{self loop} is a boundary arc between the same marked point that travels around the center.
\end{definition}

\begin{proof}[Proof of \Cref{thm:AffineType}]
First we note that we can write $T_{(n)} = T_{n,1,1}$ and $T_{(p,q)} = T_{p,q,1}$ so we can handle both of these cases together. Here we can construct a $T_{p,q,1}$ quiver from a triangulation of $S_{0,2,0,p+q}$, the annulus with $p $ marked points on one boundary and $q$ marked points on the other. We also construct a triangulation corresponding to an affine $A_{p,q}$ Dynkin diagram (\Cref{fig:AnnulusTriangles}). Since any two triangulations are related by a series of flips this shows $T_{p,q,1}$ is in the same mutation class as $A_{p,q}$ as needed.

 The first triangulation can be constructed by choosing a self loop on each boundary component. This divides the annulus into three regions: a $p$-gon, an annulus with one marked point on each boundary, and a $q$-gon. In the $p$-gon and $q$-gon, we then use the ``zig/zag'' triangulation starting from the self loop, to obtain portions of quiver that are a single line of nodes starting such that each node is a source or a sink. Finally add two distinct crossing arcs into the inner annulus completing the triangulation. See \Cref{fig:qA44_tails} for an example with $p=4$ and $q=4$.
 
The second triangulation will correspond to an orientation of the $A_{p,q}$ Dynkin diagram with a single source and sink. To construct this quiver, we first add a crossing arc between a marked point on each boundary. Next we connect the outer marked point of the initial arc to each inner marked point in a series of nested clockwise crossing arcs. Similarly attach the inner point of the initial arc to each other outer marked point in a series of nested counterclockwise crossing arcs, see \Cref{fig:qA44_cycle} for an example with $p=4$ and $q=4$.\\
\begin{figure}
    \centering
     \begin{subfigure}{.4\textwidth}
        \includegraphics[width=\textwidth]{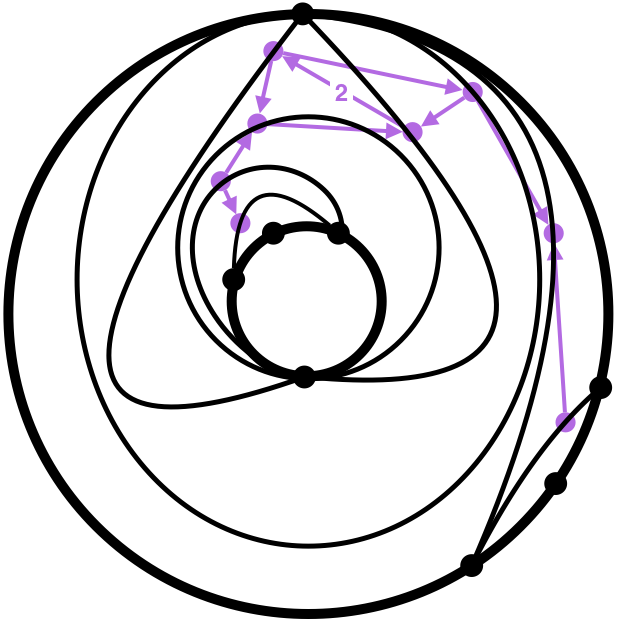}
            \caption{$T_{4,4}$}
            \label{fig:qA44_tails}
    \end{subfigure}
    \hspace*{3pc}
        \begin{subfigure}{.4\textwidth}
        \includegraphics[width=\textwidth]{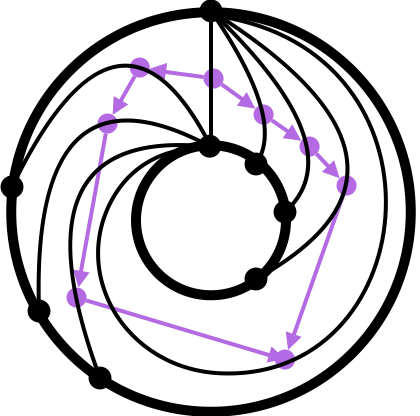}
            \caption{$A_{4,4}$ Dynkin diagram}
            \label{fig:qA44_cycle}
    \end{subfigure}
    \caption{Two different triangulations of an annulus with 4 marked points on each boundary component.}
    \label{fig:A44Triangulations}
\end{figure}

Similarly, $T_{(n,2,2)}$ occurs as the quiver obtained from a triangulation of twice punctured disk with $n$ marked points on the boundary. We also construct a triangulation of the twice punctured disk that corresponds to an $\Affine{D}_{n}$ Dynkin diagram. So as in the $\Affine{A}_n$ case this shows $T_{n,2,2}$ corresponds to the type $\Affine{D}_n$ cluster algebras. 

For the first triangulation, connect the punctures with an edge and a loop from one puncture around the other (tagged arc). Then the outside of this loop is an annulus with one marked point on the inner ``boundary'' and $n$ marked points on the outer boundary. We then complete the quiver using the construction of a $T_{n,1,1}$ quiver as described before (see \Cref{fig:qD6affine_tails}).

The second triangulation corresponding to a sources/sink orientation of a $\Affine{D}_{n}$ Dynkin diagram. First, connect each puncture to a different boundary vertex. Then add a self loop from the boundary vertex around the corresponding puncture. Outside these self loops is a disk with $n$ marked points that can be triangulated with a ``zig/zag'' starting from one self loop and ending at the other (see \Cref{fig:qD6affine_dynkin}).

\begin{figure}
    \centering
     \begin{subfigure}{.4\textwidth}
        \includegraphics[width=\textwidth]{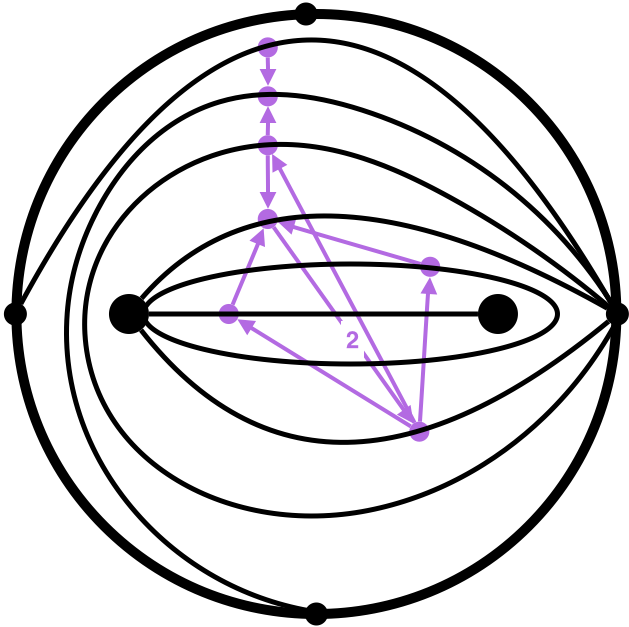}
            \caption{$T_{4,2,2}$}
            \label{fig:qD6affine_tails}
    \end{subfigure}
    \hspace*{3pc}
        \begin{subfigure}{.4\textwidth}
        \includegraphics[width=\textwidth]{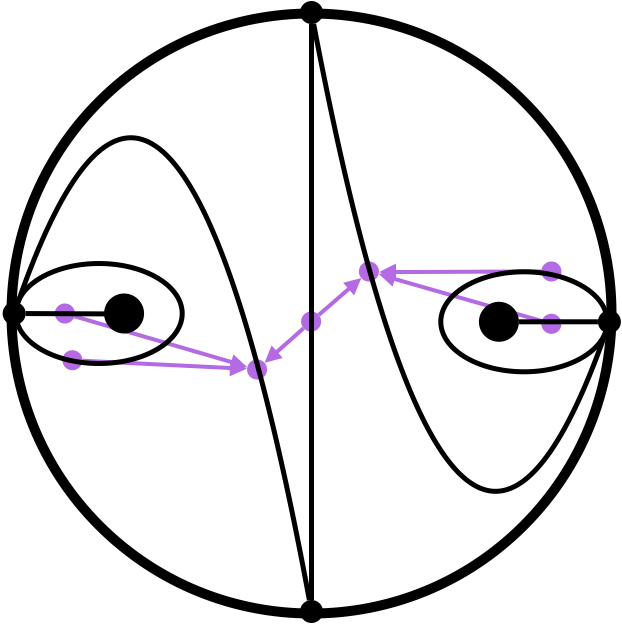}
            \caption{$\Affine{D}_{6}$ Dynkin diagram}
            \label{fig:qD6affine_dynkin}
    \end{subfigure}
    \caption{Two different triangulations of a twice punctured disk with 4 marked points on the boundary.}
    \label{fig:D6affineTriangulations}
\end{figure}

For $k= 3,4,5$ observe that $T'_{k,3,2}$ is an  $E_{k+3}$ finite Dynkin diagram oriented so every vertex is a source or a sink. Let $g = [N_1,i_{\text{odd}},i_{\text{even}},i_{2}]$ be the mutation path corresponding to the sources/sinks move for $E_{k+3}$. One can verify that $g^{h/2}$ transforms $T_{k,3,2}$ into the affine Dynkin diagram for $\Affine{E}_{k+3}$ where $h$ is the order of $g$ in $E_{k+3}$ ($h=7,10,16$ respectively). Note that applying $g$ $\frac{7}{2}$ times for $T_{3,3,2}$ means apply $g$ 3 times, then mutate at the sources $[N_1,i_{\text{odd}}]$ one more time to achieve a sources/sinks orientation of the $\Affine{E}_{6}$ diagram.

For the non simply laced cases we have explicit foldings of the simply laced cases. First consider $T_{(n,2),(1,2)}$ which we claim has type $\Affine{B}_{n+1}$. This quiver can be obtained from the $\Affine{D}_{n+2}$ by folding the length 2 tails of the $T_{n,2,2}$ quiver. As in the other cases doing $h/2$ applications of the underlying finite sources sink mutation transforms this quiver into the standard Dynkin type quiver for $\Affine{B}_{n+1}$. Note this agrees with the usual Dynkin folding of $\Affine{D}_{n+2}$ into $\Affine{B}_{n+1}$. 

The other cases are similar, $\Affine{C}_{n}$ is obtained from folding the two tails $A_{n,n}$ which corresponds on the Dynkin side via $g^{h/2}$ to folding a $2n+1$ cycle in half. $\Affine{F}_{4}$ is obtained from $T_{(3,2),(2,1)}$ by folding the two length three tails of $T_{3,3,2}$ ($\Affine{E}_6)$. The final affine quiver $\Affine{G}_2$ is $T_{(2),(3)}$ obtained by folding all three tails in $T_{2,2,2}$.

Note that every possible affine Dynkin diagram (\Cref{fig:DynkinAffineSimplyLaced,fig:DynkinAffineFolded}) has appeared as one of these cases. 
\end{proof}

\begin{remark}
    In the proof of the previous theorem we observed the mutation path $g = [N_1,i_{\text{odd}},i_{\text{even}},i_{2}]$ corresponds to the sources/sink move on the associated finite cluster algebra. The order $h$ of the sources/sinks move is computed in \cite{fomin_y-systems_2003} in terms of the Coexter number $h'$ of the associated root system. 
    
    Moreover, $g^{h/2}$ takes the affine $T_{p,q,r}$ quiver to an orientation of the affine Dynkin diagram in every case except $A_{p,q}$ with $|p-q|>  3$. In these remaining cases the affine Dynkin diagram cannot be found along this path since it is not possible to find a sources-sinks oriented finite $A_n$ quiver as a sub quiver of an affine $A_{p,q}$ quiver.
\end{remark}

We now prove \Cref{thm:AffineModularGroup} by showing the cluster modular group is $\Gamma_{\tau}\rtimes \Aut{Q}$ in each case. It is clear that  $\Gamma_\tau \rtimes \Aut{Q}$ is a subgroup of the cluster modular group, so it suffices to show their are no other possible cluster modular group elements.

\begin{proof}[Proof of \Cref{thm:AffineModularGroup} for $A_{p,q}$]
Any cluster modular group element must send our original $T_{p,q}$ quiver to another $T_{p,q}$ quiver. So it suffices to construct every possible $T_{p,q}$ quiver on the annulus and show they are in the image of the proposed group.

Once again we will rely on the correspondence between seeds in the cluster algebra and triangulations of an annulus. Since this quiver has a double edge, by \Cref{rem:DoubleEdgeTriangles} the only possible construction of a $T_{p,q,1}$ quiver is the one given in the proof of \Cref{thm:AffineType}.

However there was some freedom in this construction. The first is the choice of marked point on each boundary component to add a self loop around. There are $pq$ total possible choices for this. The other more subtle degree of freedom is the action of the mapping class group of the annulus, generated by a single Dehn twist about the center. Note the Dehn twist only changes crossing arcs which correspond to nodes $N_1$ and $N_\infty$. A simple analysis shows that $\gamma$ corresponds exactly to the action of the Dehn twist.
 
 Then $\Gamma_{\tau}/\groupgenby{\gamma} = \cyclicgroup{p}\times \cyclicgroup{q}$ has order $pq$. Therefore each distinct copy of $T_{p,q,1}$ up to mapping class group is the image of a distinct twist as needed. Since no other triangulation produce an isomorphic quiver we are done as long as $p \neq q$.
 
 When $p=q$ there is an extra symmetry of the triangulation given by swapping the inner and outer boundary components. However this is exactly automorphism of $T_{p,p,1}$ that swaps each tail. This corresponds exactly to the action of $\Aut{T_{p,p,1}}$ on $\Gamma_{\tau}$ as needed.
\end{proof}

\begin{proof}[Proof of \Cref{thm:AffineModularGroup} for $\Affine{D}_{n}$]
 As in the $A_{p,q}$ case the only possible construction of the $T_{n,2,2}$ quiver is the one described in the proof of \Cref{thm:AffineType}. Thus we look at the ambiguity of the construction of the $T_{n,2,2}$ quiver. The obvious choices are which puncture is inside the self loop, the boundary vertex that is attached to the puncture, and the winding number of these crossing edges. There is an additional subtle choice from the tagged arc complex. In this generalization the self loop around a puncture is replaced with a singly tagged arc between the two punctures. There is then an additional way to get an isomorphic quiver by switching the tagging at a puncture. This operation at the puncture with a tagged arc simply swaps the two arcs between the punctures and thus corresponds to the extra semidirect product with $\cyclicgroup{2}$ when $n \neq 4$. However flipping the tagging at the other puncture results in a new triangulation in every case. Putting this all together gives $4n$ triangulations up to winding number. Mutation along the double edge correspond to the Dehn twist around both punctures so we can again see that $\Gamma_\tau / \groupgenby{\gamma} = \cyclicgroup{n} \times \cyclicgroup{2} \times \cyclicgroup{2}$ has order $4n$ and so reaches every possibility.
 
 When $n=4$ not every automorphism of $T_{2,2,2}$ corresponds to a symmetry of the twice punctured disk as described above, but otherwise the analysis is exactly the same.\\
\end{proof}

\begin{proof}[Proof of \Cref{thm:AffineModularGroup} for $\Affine{E}_6,\Affine{E}_7,\Affine{E}_8$]
 In \cite{Schiffler:cluster_automorphisms} they compute the cluster modular group for the Dynkin type quivers as $\Z \times S_3$, $\Z \times \cyclicgroup{2}$, and $\Z$ for $\Affine{E}_6,\Affine{E}_7,\Affine{E}_8$ respectively. In each case the $\Z$ is generated by the full sources/sinks move on the Dynkin quiver. This is the reddening element $r$ by \Cref{thm:tnw_reddening}. By \Cref{rem:Tgroup} the group $\Gamma_\tau$ can be written as a subgroup of $\Z \times \prod \Z_n$. For example in $\Affine{E}_6$, $\Gamma_\tau$ is the subgroup of $\Z\times \Z_3 \times \Z_3 \times \Z_2$  generated by 
 \[\gamma = (18,0,0,0) \hspace{2pc} \tau_1 = (6,1,0,0) \hspace{2pc} \tau_2 = (6,0,1,0) \hspace{2pc} \tau_3 = (9,0,0,1)\]
 Here the reddening element $r = \tau_1\tau_2\tau_3\gamma^{-1} = (3,1,1,1)$. The final generator is the element $\sigma$ of $\Aut{T_{3,3,2}}$ which swaps the two tails of length 3. It satisfies the relations $\sigma \tau_1 \sigma = \tau_2$ and $\sigma \tau_3 \sigma = \tau_3$. Now we verify that the set $\{r, \sigma, \tau_1\tau_2^{-1}\}$ generates the full group with $\groupgenby{r}=\Z$ and $\groupgenby{\sigma, \tau_1\tau_2^{-1}} = S_3$ as needed. It suffices to verify
 \[\tau_1 = (\tau_1\tau_2^{-1})^{-1}r^2 \hspace{2pc} \tau_2 = (\tau_1\tau_2^{-1})r^2 \hspace{2pc} \tau_3 = r^3 \]
 \[\sigma^2 = 1 \hspace{2pc} (\tau_1\tau_2^{-1})^3 = 1 \hspace{2pc} \sigma (\tau_1\tau_2^{-1})\sigma = \tau_1^{-1}\tau_2 \hspace{2pc}\]
 
 In $\Affine{E}_7$, $\Aut{T_{4,3,2}}$ is trivial. So $\Gamma = \Gamma_\tau$ is the subgroup of $\Z\times \Z_4\times\Z_3\times \Z_2$ generated by
\[\gamma = (24,0,0,0) \hspace{2pc} \tau_1 = (6,1,0,0) \hspace{2pc} \tau_2 = (8,0,1,0) \hspace{2pc} \tau_3 = (12,0,0,1)\]
We compute $r = (2,1,1,1)$ and verify that $\{r,\tau_1^2\tau_3^{-1}\}$ generates the full group. This follows from the following computations
\[ \tau_1 = (\tau_1^{2}\tau_3^{-1})r^3 \hspace{2pc }\tau_2 = r^4 \hspace{2pc} \tau_3 = (\tau_1^2\tau_3^{-1})r^6\]
Finally in $\Affine{E}_8$, $\Aut{T_{5,3,2}}$ is trivial and $\Gamma = \Gamma_\tau$ is realized as the subgroup of $\Z\times \Z_5\times \Z_3\times \Z_2$ generated by
\[\gamma = (30,0,0,0) \hspace{2pc} \tau_1 = (6,1,0,0) \hspace{2pc} \tau_2 = (10,0,1,0) \hspace{2pc} \tau_3 = (15,0,0,1)\]
Here $r = (1,1,1,1)$ and we see $\tau_1 = r^6$, $\tau_2 = r^{10}$ and $\tau_3 = r^{15}$. Thus $\Gamma = \groupgenby{r} = \Z$ as claimed.
We remark that in each the presentation of \cite{Schiffler:cluster_automorphisms} is given by $\groupgenby{r} \times (\Gamma_{\tau}^{\circ} \rtimes \Aut{Q})$ where $\Gamma_{\tau}^{\circ}$ is the finite subgroup generated by combinations of twists with finite order.
 
\end{proof}

\begin{lemma}
Folding the tails of the $\T$ quivers only changes the cluster modular group by reducing automorphism group of the quiver and identifying the generators corresponding to twists about the folded tails.
\end{lemma}
\begin{proof}
This follows from \Cref{rem:FoldingTwists} that weight 2 or 3 twists are equivalent to simultaneous twists of the corresponding number of equal length tails. Finally \Cref{thm:FoldingClusterModularGroups} shows that there are no extra elements of the folded cluster modular group.
\end{proof}
\begin{proof}[Proof of \Cref{thm:AffineModularGroup} for non simply laced diagrams]
To prove each non simply laced affine $\T$ corresponded to an affine diagram, we gave an explicit folding of each simply laced $T_{\vec{n},\vec{1}}$ quiver and so the previous lemma applies.
\end{proof}

\begin{lemma}\label{thm:ApqRecurrence}
    $A_{p+1,q} = \displaystyle{2\sum_{i=0}^{p-1}C_iA_{p-i,q} + qC_{p+q} }$.
    \end{lemma}
    \begin{proof}
    We can obtain this recurrence by partitioning the set of triangulations by the triangle that contains the edge between $o_1$ and $o_2$ on the outer boundary. The third vertex of the triangle can either be on the outer or inner boundary. If the third vertex is some $o$ the edges can either go clockwise or counterclockwise around the center. In either case it splits the annulus into a polygon with $i+2$ sides and an annulus with $p-i$ outer marked points and $q$ inner marked points. The triangulations of the polygon are fixed by $\gamma$ and there are $C_{i}$ ways to triangulate an $i+2$ gon. So there are $2 \sum\limits_{i=0}^{p-1} C_i A_{p-i,q}$ possible triangulations where the third vertex is on the outer boundary component.
    
    If the third vertex is on the inside there is only one possible triangle up to $\gamma$. Once this triangle is picked, it leaves a $p+q+2$ sided polygon regardless of which of the $q$ possible points we choose. So there are $qC_{p+1}$ ways in this case. See \Cref{fig:AnnulusTriangles} for a visual of all three cases.

    \begin{figure}
        \centering
        \includegraphics[width=.9\textwidth]{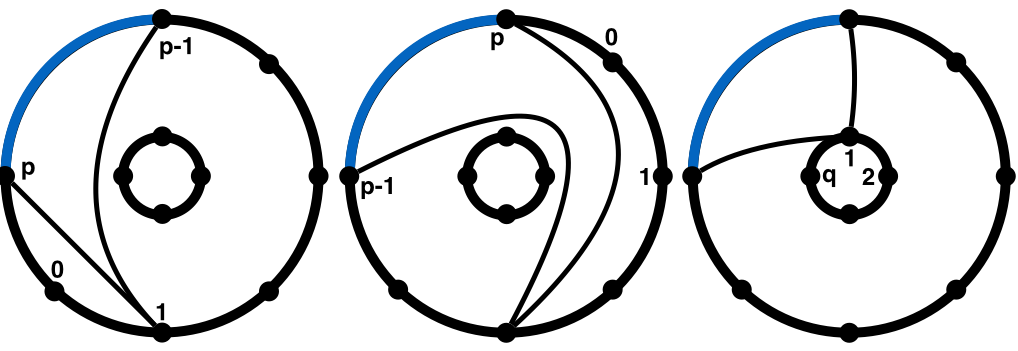}
        \caption{All kinds of triangles including the blue edge up to the action of the mapping class group.}
        \label{fig:AnnulusTriangles}
    \end{figure}
    
   \end{proof}

  \begin{lemma}\label{thm:DnRecurrence}
    $\Affine{D}_{n+1} = 2\sum\limits_{i=0}^{n-3} C_i \Affine{D}_{n-i} + 2 \sum\limits_{j=0}^{n}D_jD_{n-j} $
    \end{lemma}
    \begin{proof}
     As in the $\Affine{A}_n$ case we partition the triangulations based on the triangle containing a fixed boundary edge. In this case there are six cases up to a full twist around both punctures (\Cref{fig:TwicePuncturedTriangles}). The first two cases correspond to triangles with third vertex on the boundary with edges going around both punctures (clockwise or counter clockwise). In either case the triangle splits the region into a $i$ sided polygon and a twice punctured disk with  $n-i$ marked points. This covers the first summation in the recurrence.
     
     The next two cases correspond to triangle where the edges go between the punctures. If we label the $n-2$ marked points $1$ to $n-1$, the triangle between the punctures going to vertex $j$ splits the region into a punctured disk with $j$ marked points and one with $n-j$ marked points. This covers the terms $2\sum\limits_{j=1}^{n-1}D_j D_{n-j}$.
     
     The final two cases are the triangles with endpoint on a puncture. Up to the full twist there is only one way to reach each puncture. There is an additional tagged triangulation in each case. In any of these cases the remaining region is a disk with $n$ marked points. Since we took $D_0 = 1$ we can write the number of triangulations in this case as $D_0 D_n$ and $D_n D_0$ covering the missing terms in the second summation of the recurrence.

      \begin{figure}
        \centering
        \includegraphics[width=.9\textwidth]{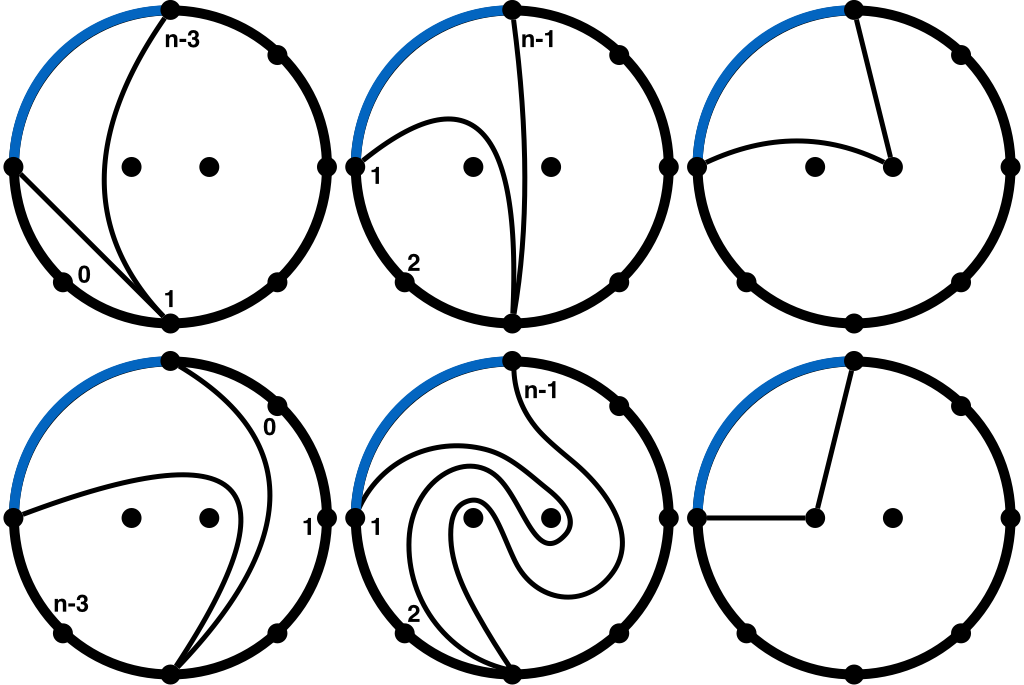}
        \caption{All kinds of triangles including the blue edge up to the action of the mapping class group}
        \label{fig:TwicePuncturedTriangles}
    \end{figure}
    
    \end{proof}

\end{document}